\documentclass[11pt]{amsart}

 \usepackage{amsmath,amsthm,amsfonts,amssymb,verbatim}
 \usepackage{url}
 \usepackage{graphicx}
 \usepackage[all]{xy}
 \usepackage{wrapfig}
 \usepackage{picins}
 \usepackage{pinlabel}
 \usepackage{subfigure}
 
\def\co{\colon\thinspace}

\newcommand{\sQ}{\mathcal{Q}}

\newcommand{\bN}{\mathbb{N}}
\newcommand{\bZ}{\mathbb{Z}}
\newcommand{\bQ}{\mathbb{Q}}

\newcommand{\bC}{\mathbb{C}}

\newcommand{\bF}{\mathbb{F}}

\newcommand{\pq}{\frac{p}{q}}
\newcommand{\pqzero}{\frac{p_0}{q_0}}
\newcommand{\pqone}{\frac{p_1}{q_1}}
\newcommand{\overzero}{\frac{1}{0}}

\newcommand{\del}{\delta}

\newcommand{\si}{\sigma}

\newcommand{\lm}{\lambda_M}
\newcommand{\ord}{\operatorname{ord}}

\newcommand{\longto}{\longrightarrow}

\newcommand{\into}{\hookrightarrow}

\newcommand{\half}{\textstyle\frac{1}{2}}

\newcommand{\wmax}{w_{\max}}
\newcommand{\wmin}{w_{\min}}

\newcommand{\Supp}{\operatorname{Supp}}
\newcommand{\Br}{\mathbf{\Sigma}}
\newcommand{\fibre}{\varphi}
\newcommand{\boldtau}{\boldsymbol\tau}

\newcommand{\HFhat}{\widehat{\operatorname{HF}}}

\newcommand{\Khred}{\widetilde{\operatorname{Kh}}}

\newcommand{\Kh}{\operatorname{Kh}}

\newcommand{\Khsig}{\widetilde{\operatorname{Kh}}_\si}

\newcommand{\rk}{\operatorname{rk}}

\newcommand{\fix}{\operatorname{Fix}}
\newcommand{\End}{\operatorname{End}}

\newtheorem{theorem}{Theorem}[section]

\newtheorem{definition}[theorem]{Definition}

\newtheorem{corollary}[theorem]{Corollary}
\newtheorem{proposition}[theorem]{Proposition}
\newtheorem{remark}[theorem]{Remark}
\newtheorem{lemma}[theorem]{Lemma}

\newtheorem{property}[theorem]{Property}

\newtheorem*{namedtheorem}{\theoremname}
\newcommand{\theoremname}{testing}
\newenvironment{named}[1]{\renewcommand{\theoremname}{#1}
        \begin{namedtheorem}}
        {\end{namedtheorem}}

\newcommand{\unknot}
{\raisebox{-0.055in}{\includegraphics[scale=0.125]{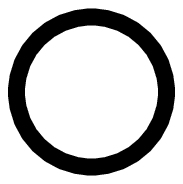}}}
\newcommand{\positive}
	{\raisebox{-2pt}{\includegraphics[scale=0.1]{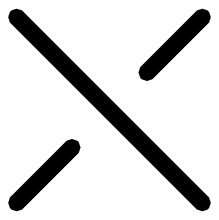}}}
\newcommand{\negative}
	{\raisebox{-2pt}
	{\includegraphics[scale=0.1]{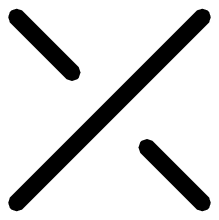}}}
\newcommand{\rightcross}
	{\raisebox{-2pt}
	{\includegraphics[scale=0.1]{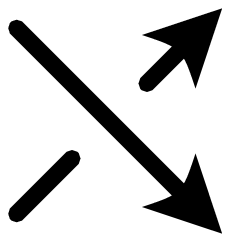}}}
\newcommand{\leftcross}
	{\raisebox{-2pt}
{\includegraphics[scale=0.1]{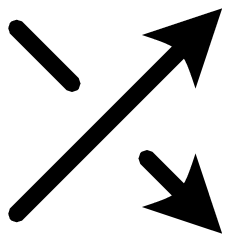}}}
\newcommand{\zero}
	{\raisebox{-2pt}
	{\includegraphics[scale=0.1]{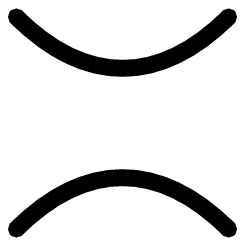}}}
\newcommand{\one}
	{\raisebox{-2pt}
	{\includegraphics[scale=0.1]{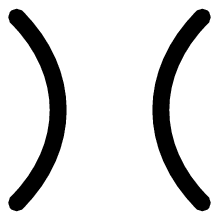}}}

\title{Surgery obstructions from Khovanov homology}
\author{Liam Watson}
\thanks{Supported by a Canada Graduate Scholarship (NSERC)}
\address{Department of Mathematics, UCLA, 520 Portola Plaza, Los Angeles, CA 90095.}
\urladdr{http://www.math.ucla.edu/~lwatson}
\email{lwatson@math.ucla.edu}
\date{First version: July 8, 2008. This version: August 22, 2011.}

\begin{document}

\begin{abstract}
For a 3-manifold with torus boundary admitting an appropriate involution, we show that Khovanov homology provides obstructions to certain exceptional Dehn fillings. For example, given a strongly invertible knot in $S^3$, we give obstructions to lens space surgeries, as well as obstructions to surgeries with finite fundamental group. These obstructions are based on homological width in Khovanov homology, and in the case of finite fundamental group depend on a calculation of the homological width for a family of Montesinos links. \\


\end{abstract}

\maketitle


\section{Introduction}\label{sec:introduction}

In his pioneering work on the geometry and topology of 3-manifolds, Thurston showed that a hyperbolic manifold $M$ with torus boundary admits a finite number of exceptional Dehn fillings \cite{Thurston1980,Thurston1982}. That is, those closed manifolds obtained from $M$ by attaching a solid torus to the boundary that are non-hyperbolic. Since then, the question of understanding and classifying exceptional surgeries has received considerable attention (see surveys by Gordon \cite{Gordon1991} and Boyer \cite{Boyer2002}).

Perhaps the simplest non-hyperbolic manifold is a lens space. Restricting to complements of knots in $S^3$, Moser \cite{Moser1971} showed that torus knots always admit lens space surgeries. Subsequently, Bailey and Rolfsen \cite{BR1977} constructed an example of a lens space surgery on a non-torus knot (a particular cable of the trefoil), and Fintushel and Stern \cite{FS1980} obtained further examples, including hyperbolic knots, that admit lens space surgeries. 

In \cite{Berge} Berge gives a list of knots in $S^3$ that admit lens space surgeries. These knots are referred to as Berge knots, and it has since been conjectured that this list is complete. That is, if a knot in $S^3$ admits a lens space surgery then it must be a Berge knot; this has become known as the Berge conjecture. Since Berge knots are genus 2, they are strongly invertible by a result of Osborne \cite{Osborne1981}. Consequently, the Berge conjecture may be restated in two steps: first show that any knot admitting a lens space surgery is strongly invertible, then show that a strongly invertible knot admitting a lens space surgery is a Berge knot.

The complement of a strongly invertible knot admits a tangle associated with the quotient of the strong inversion. This {\em associated quotient tangle} (see Definition \ref{def:associated-tangle}) can be a useful object when studying surgery on the knot in question. Indeed, surgeries on the strongly invertible knot correspond in a natural way to closures of the associated quotient tangle by a rational tangle. In this way, the manifold obtained by surgery is naturally a two-fold branched cover of $S^3$, branched over an appropriate closure of the associated quotient tangle. While both parts of the restatement of the Berge conjecture given above remain open, in light of the second it may be useful to use invariants of knots in $S^3$ to study the branch sets associated with such surgeries.   

In this setting Khovanov homology \cite{Khovanov2000} may be used to give information about the manifold obtained via surgery in the cover:

\begin{named}{Theorem \ref{thm:lens}}Let $T$ be the associated quotient tangle of a strongly invertible knot $K$ in $S^3$. If $T$ has the property that the links obtained from it by attaching rational tangles have thick Khovanov homology, then $K$ does not admit a non-trivial lens space surgery. Moreover, under mild hypothesis on $T$ it suffices to verify only a finite number of branch sets associated with integer surgeries to obtain the conclusion for all possible fillings.\end{named}
This depends on a result that combines work of Hodgson and Rubinstein \cite{HR1985} and Lee \cite{Lee2005}: lens spaces arise only as the branched covers of homologically thin links  (see Theorem \ref{thm:lens-bound}). As a result of a particular stable behaviour of Khovanov homology for branch sets associated with such surgeries (see Lemma \ref{lem:general-stability}), it suffices to calculate a finite collection of Khovanov homology groups to apply this obstruction. Indeed, we will see in application that it is often enough to calculate one or two Khovanov homology groups and that the required genericity hypothesis alluded to on $T$ (see Definition \ref{def:generic}) is easily satisfied.  

Recently, the question of lens space surgeries has been treated by Ozsv\'ath and Szab\'o \cite{OSz2005-lens} and Rasmussen \cite{Rasmussen2004-1} from the point of view of Heegaard Floer homology. Indeed, some progress on the Berge conjecture has been made by way of Heegaard Floer homology (see the programs put forth by Baker, Grigsby and Hedden \cite{BGH2008,Hedden2007} and Rasmussen \cite{Rasmussen2007}). Moreover, Ozsv\'ath and Szab\'o \cite{OSz2005-branch} have shown that there is a close relationship (by way of a spectral sequence) between the Khovanov homology of a link and the Heegaard Floer homology of the two-fold branched cover of $S^3$, branched over the link. From this point of view, it is natural to ask how the obstructions from the two theories might be related. 
  
The work of Ozsv\'ath and Szab\'o \cite{OSz2005-lens} gives, more generally, obstructions to L-space surgeries. Interesting examples of L-spaces include two-fold branched covers of links with thin Khovanov homology (see \cite{OSz2005-branch} and Proposition \ref{prp:branch-over-thin}), as well as manifolds admitting elliptic geometry \cite{OSz2005-lens}. In a similar vein, it can be shown that the branch set associated with a manifold with finite fundamental group (viewed as a two-fold branched cover of $S^3$) has relatively simple Khovanov homology.
\begin{named}{Theorem \ref{thm:finite-bound}}
If the two-fold branched cover of $L$ has finite fundamental group, then the reduced Khovanov homology of $L$ is supported in at most two adjacent diagonals. 
\end{named}
As a result, we obtain the following:
\begin{named}{Theorem \ref{thm:finite}}
Let $T$ be the associated quotient tangle of a strongly invertible knot $K$ in $S^3$. If $T$ has the property that the links obtained from it by attaching rational tangles have reduced Khovanov homology supported in more than two adjacent diagonals, then $K$ does not admit a non-trivial surgery with finite fundamental group. Moreover, under mild hypothesis on $T$ it suffices to verify only a finite number of branch sets associated with integer surgeries to obtain the conclusion for all possible fillings. 
\end{named}

One may view the obstructions given above as a rough correspondence between the geometry of a two-fold branched cover and the Khovanov homology of the branch set: simple manifolds (in terms of geometry) tend to have simple branch sets (in terms of Khovanov homology). That such a correspondence exists is interesting in light of the fact that Khovanov homology, while relatively strong as an invariant of knots in $S^3$, lacks a complete geometric interpretation. As a result, this invariant has seen relatively few geometric applications despite receiving considerable attention since its inception. That said, the applications that have arisen have been particularly interesting. Perhaps most notable, Rasmussen \cite{Rasmussen2010} obtained a combinatorial proof of the Milnor conjecture using Khovanov homology. As such, the search for further applications of the theory is of central interest. 

It should be pointed out that the obstructions given here apply to a wider class of manifolds with torus boundary than complements of knots in $S^3$: any manifold that has the structure of a two-fold branched cover may be studied by way of the Khovanov homology of the branch set (see Theorem \ref{thm:general}). 

\subsection*{Organization of the paper} Section \ref{sec:kh} gives a brief review of Khovanov homology, and in particular the skein exact sequence, as the grading conventions used in this work are non-standard and adapted to the study of homological width. We prove a degenerate case of a version of the skein exact sequence due to Manolescu and Ozsv\'ath \cite{MO2007} (compare Proposition \ref{prp:mo} and Proposition \ref{prp:mo-perturbed}), and introduce the $\si$-normalized Khovanov homology. This is a useful $\bZ$-graded object in the context of this work, and seems to be a natural and interesting object in its own right.  

In Section \ref{sec:fillings} we review the required elements of surgery theory on 3-manifolds, and introduce the notion of a simple, strongly invertible knot manifold (Definition \ref{def:simple-strong}) and the associated quotient tangle (Definition \ref{def:associated-tangle}). This is precisely the family of manifolds for which fillings may be studied by way of Khovanov homology. 

Section \ref{sec:width} establishes the stability of Khovanov homology for branch sets associated with integer surgery (see Lemma \ref{lem:general-stability}). As a result, we may define the maximal and minimal width of the Khovanov homology for the branch set of an integer filling (Definition \ref{def:max-min}), and this is used to give upper bounds for width for the branch set of an arbitrary filling in terms of the branch sets associated with integer fillings. In pursuing this, we establish an interesting characterization of the trivial knot, among strongly invertible knots, from Khovanov homology (Theorem \ref{thm:trivial-knot}). We also discuss quasi-alternating links as they arise naturally in this context (see Theorem \ref{thm:quasi-alternating}), and in particular we demonstrate that L-spaces arising from large surgery on a Berge knot may be realized as the two-fold branched cover of a quasi-alternating link (see Proposition \ref{prp:Berge-quasi-alternating}). 

Section \ref{sec:estimates} treats lower bounds for width for the branch set of an arbitrary filling in terms of the branch sets associated with integer fillings. This relies on a notion of generic tangle introduced in Definition \ref{def:generic}.

Section \ref{sec:obstructions} applies the above material to the main results of this paper. We prove upper bounds for the width of the Khovanov homology of a branch set associated with a lens space surgery (Theorem \ref{thm:lens-bound}), as well as for the width of a finite filling (Theorem \ref{thm:finite-bound}). These bounds, combined with the stable behaviour of the associated Khovanov homology groups developed in Section \ref{sec:width} give rise to our main results on surgery obstructions (Theorem \ref{thm:lens} and Theorem \ref{thm:finite}). 

Finally, we turn to examples and applications in Section \ref{sec:examples}. As a first example, we recover the fact that the figure eight does not admit finite fillings. We also  show that surgery on the $(-2,p,p)$-pretzel knot does not yield a manifold with finite fundamental group (for $p\in\{5,7,\ldots,31\}$). Finite fillings on this family of Montesinos links was left unresolved in Mattman's extensive study of the problem using character variety methods \cite{Mattman2000} (Mattman's classification has very recently been completed by Ichihara and Jong applying Heegaard Floer homology techniques \cite{IJ2008}, and independently treated by Futer, Ishikawa, Kabaya, Mattman and Shimokawa \cite{FIKMS2009}).  

The characterization of the trivial knot (among strongly invertible knots) in terms of Khovanov homology (Theorem \ref{thm:trivial-knot}) seems interesting in light of the connection between Khovanov homology and Heegaard Floer homology for two-fold branched covers, together with the fact that knot Floer homology detects the trivial knot \cite{OSz2004-genus}. Indeed, it may be reasonable to guess that obstructions from Khovanov homology are closely related to those arising in Heegaard Floer homology. For this reason, we attempt to compare and contrast our obstructions with the powerful obstructions of Heegaard Floer homology through a final example:  We exhibit that the knot $14^n_{11893}$ (see Figure \ref{fig:14n11893-inversion}) does not admit surgeries with finite fundamental group. Since the Alexander polynomial of this knot does not rule out L-space surgeries (an obstruction that follows from Heegaard Floer homology \cite{OSz2005-lens}), this establishes that Khovanov homology obstructions can, in certain settings, be stronger than those provided by the Alexander polynomial. We remark that, while not pursued in this work, the obstructions given here may be applied to certain strongly invertible knots in manifolds other than $S^3$. In particular, one may consider knots in non-L-spaces, a setting wherein methods from Heegaard Floer homology can  be considerably more difficult to apply. 

\subsection*{On conventions and calculations} Throughout we will use Rolfsen's notation \cite{Rolfsen1976} for knots with 10 or fewer crossings, and {\tt Knotscape} notation \cite{Knotscape} for knots with more than 10 crossings, as has become standard in the literature. Calculations of the Khovanov homology groups given in this work were performed using the program {\tt JavaKh} by Bar-Natan and Green \cite{JavaKh}. 

\subsection*{Acknowledgements} This work benefited greatly from conversations with Michel Boileau, Steve Boyer, Matt Hedden, Patrick Ingram, Thomas Mattman, Peter Ozsv\'ath, Luisa Paoluzzi and Jeremy Van Horn-Morris. We are particularly indebted to the anonymous referee who made numerous suggestions that led to improvements throughout the paper.


\section{Khovanov homology}\label{sec:kh}

We begin by reviewing Khovanov homology \cite{Khovanov2000} to fix notation and conventions. We work with the reduced version of the theory with coefficients in the field $\bF=\bZ/2\bZ$. For definitions see Khovanov's work \cite{Khovanov2003}, as well as work of Shumakovitch \cite{Shumakovitch2004}. 

\subsection{Grading conventions.}To a link $L\into S^3$, reduced Khovanov homology associates a bi-graded group (or, $\bF$-vector space) $\Khred(L)$ with primary grading $\delta$ and secondary grading $q$. This grading is non-standard: to recover the original conventions from $\delta$ and $q$, the  primary (homological) grading is given by $u'=\delta+2q$, and secondary (quantum) grading by $q'=2q$ (compare Theorem \ref{thm:kh} below). With this grading convention, recall that an observation due to Shumakovitch \cite{Shumakovitch2004} ensures that, over $\bF$, the full Khovanov homology decomposes as two copies of the reduced theory. That is, $\Kh(L)\cong\Khred(L)\{+1\}\oplus\Khred(L)\{-1\}$ where $\{\cdot\}$ shifts the quantum grading $q'$. 

Unless otherwise specified, we will consider $\Khred(L)$ as a {\em relatively} $\bZ\oplus\bZ$-graded group. This has the advantage that $\Khred(L)$ becomes an invariant of unoriented links, resulting from the fact that the absolute grading in Khovanov homology is obtained by applying an overall shift to the cube of resolutions defining the chain complex \cite{Khovanov2000}.  

As an absolutely graded group, this homology theory categorifies to the Jones polynomial \cite{Jones1985} in the following sense:

\begin{theorem}[Khovanov \cite{Khovanov2003}]\label{thm:kh} Let $u=\delta+q$. Then there is a unique absolute $\bZ\oplus\frac{1}{2}\bZ$-grading (in $(u,q)$) on $\Khred(L)$ with the property that \[V_L(t)=\textstyle{\sum}_{u,q}(-1)^ut^q\rk\Khred^u_q(L),\] where $V_L(t)\in\bZ[t^{\frac{1}{2}},t^{-\frac{1}{2}}]$ is the Jones polynomial of the oriented link $L$. \end{theorem} 

We remark that the universal coefficient theorem ensures that the graded Euler characteristic (giving rise to the Jones polynomial) is invariant of the coefficient field.  

According to our grading conventions, the usual Euler characteristic $\chi(\Khred(L))=\textstyle{\sum}_\del(-1)^\del\rk\Khred^\del(L)$ is obtained by collapsing the $q$ grading. Note that this is only well defined up to sign as $\del$ is a relative integer grading; we fix the convention $\chi\ge0$. Recall that $\det(L)=|H_1(\Br(S^3,L);\bZ)|$ (with the convention that $\det(L)=0$ when this group is infinite) where $\Br(S^3,L)$ is the two-fold branched cover of $S^3$, branched over $L$.

\begin{proposition}\label{prp:Kh-Euler} With the above notation and conventions, $\chi(\Khred(L))=\det(L)$.\end{proposition}
\begin{proof} By definition, $\chi(\Khred(L))=\textstyle\big|\sum_{\del,q}(-1)^\del\rk\Khred^\del_q(L)\big|=\textstyle\big|\sum_{u,q}(-1)^{u}(-1)^{q}\rk\Khred^u_q(L)\big|,$ 
and the result follows from the well known identity $\det(L)=|V_L(-1)|$.
\end{proof}

\subsection{Homological width} Forgetting the $q$-grading as in Proposition \ref{prp:Kh-Euler} yields $\textstyle\Khred(L)\cong \bigoplus_{\delta=1}^k\bF^{b_\delta}$ for non-negative integers $b_i$ (where $b_1$ and $b_k$ are positive), which we will often denote by $\Khred(L)\cong\bF^{b_1} \oplus \cdots\oplus\bF^{b_k}$ as a means of representing the $\delta$-grading.  As a result, we arrive naturally at the following:
\begin{definition}\label{def:width} The homological width of $L$ is given by  the difference in highest and lowest $\delta$-grading plus one (for any choice of absolute $\delta$-grading). We donote this positive integer by $w(L)$. Links for which $w=1$ are called thin (or homologically thin), while links with $w>1$ are termed thick (or homologically thick).
\end{definition}
\begin{remark}\label{rmk:width-and-positive-betti-numbers}
Note that for $\Khred(L)\cong\bF^{b_1} \oplus \cdots\oplus\bF^{b_k}$ an equivalent definition for $w(L)=k$ is the number of $\delta$-gradings supporting the reduced Khovanov homology provided $b_i>0$ for all $i$. This definition can be useful in practice, though it is not known that this condition on rank is satisfied for the Khovanov homology of links in general;  in principal this value may be lower than the homological width. In this work, the support of the Khovanov homology of a link (see Section \ref{sub:support}) will always refer to the interval between the maximum and minimum $\delta$-gradings, rather than the subset of this interval on which the Khovanov homology is non-zero.
\end{remark}

A result due to Lee shows that non-split alternating links give a family of thin links \cite{Lee2005}; observing that $\chi(\Khred(L))=|\sum_{\delta=1}^k(-1)^\delta b_\delta|$ and $\rk\Khred(L)=\sum_{\delta=1}^kb_\delta$ gives rise to examples of thick links.

\begin{figure}\begin{center}
\labellist\small
	\pinlabel $1$ at 161 236 
	\pinlabel $1$ at 161 307 
	\pinlabel $1$ at 161 340
\endlabellist
\includegraphics[scale=0.35]{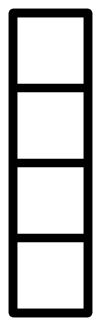} \qquad
\raisebox{0pt}{\includegraphics[scale=0.1875]{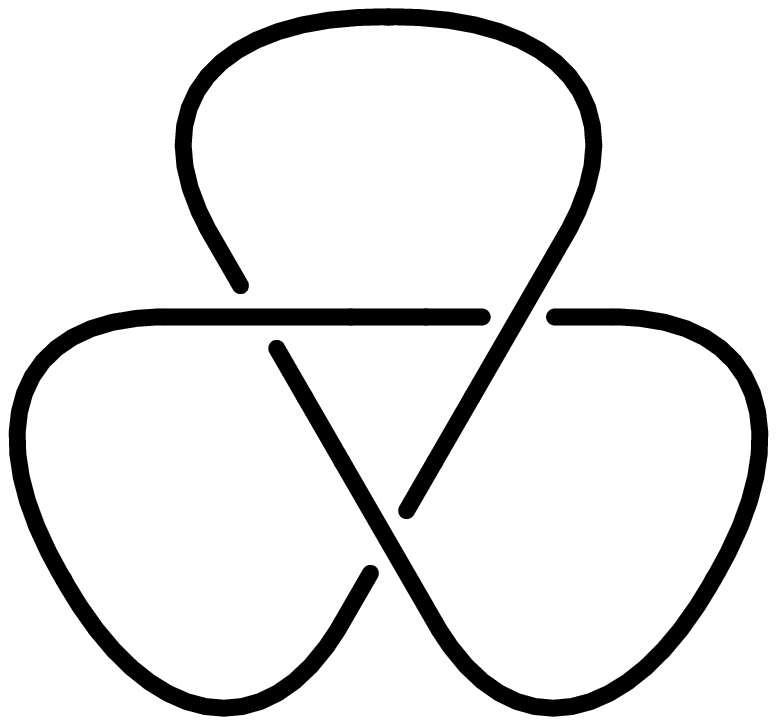}}
\qquad\qquad\qquad
\labellist\small
	\pinlabel $1$ at 161 236 
	\pinlabel $1$ at 161 269
	\pinlabel $1$ at 161 340
	
	\pinlabel $1$ at 196 269
	\pinlabel $1$ at 196 340
	\pinlabel $1$ at 196 378
	\pinlabel $1$ at 196 451
\endlabellist
\includegraphics[scale=0.35]{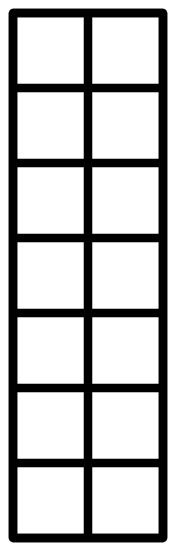} \qquad
\raisebox{0pt}{\includegraphics[scale=0.25]{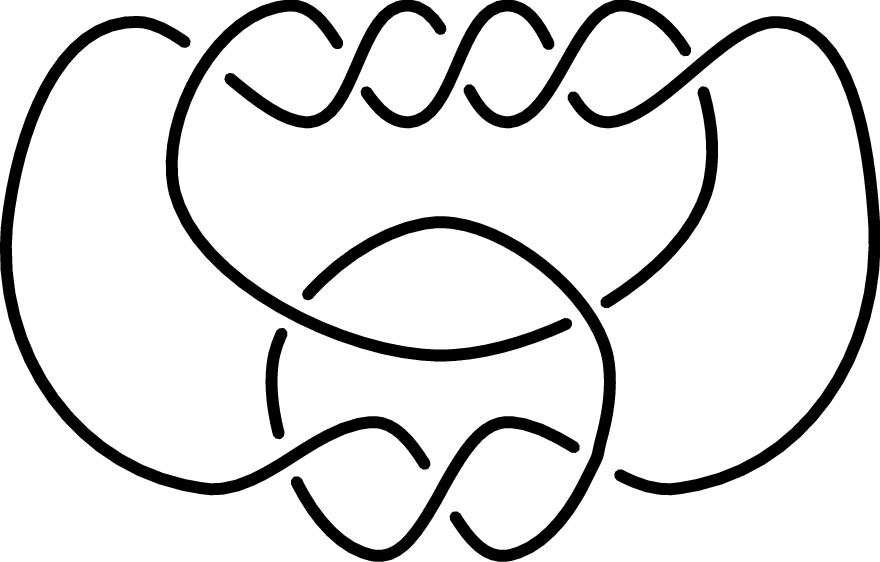}}
\caption{The reduced Khovanov homology of the trefoil (left) with $w=1$, and the knot $10_{124}$ (right) with $w=2$. The primary relative grading ($\delta$) is read horizontally, and the secondary relative grading ($q$) is read vertically. The values at a given bi-grading give the ranks of the abelian group (or $\bF$-vector space) at that location; trivial groups are left blank.}\label{fig:Kh-two-examples}\end{center}\end{figure}

\begin{proposition}\label{prp:homologically-thick} Any link $L$ with $\det(L)=0$ must have $w(L)>1$.\end{proposition}
\begin{proof}Since $V_L(t)$ is a non-zero polynomial \cite[Theorem 15]{Jones1985}, it follows that $\rk\Khred(L)>0$ for any link $L$. In particular there is at least one $b_\delta\ne0$. Therefore if $\det(L)=\chi\big(\Khred(L)\big)=0$ there must be at least two such gradings supporting non-trivial groups. \end{proof}

This work is principally concerned with the invariant $w(L)$ arising from the primary grading $\delta$. This grading counts the number of diagonals (of slope 1) supporting the reduced Khovanov homology when the gradings $(u,q)$ from Theorem \ref{thm:kh} are considered. As such, the $\delta$-grading used here corresponds to other occurrences of $\delta$ in the literature on homological width (for example, in work of Manolescu and Ozsv\'ath \cite{MO2007}). 

\subsection{Mapping cones}

The skein exact sequence for reduced Khovanov homology relates the homology of a link with a fixed crossing $\positive$ to the homology of the links obtained from the 0-resolution $\zero$ and the 1-resolution $\one$. For a link $L(\rightcross)$ with distinguished positive crossing we have that
\[\longto \Khred\left(L(\one)\right)[\textstyle-\frac{1}{2}c,\frac{1}{2}(3c+2)]\longto \Khred\left(L(\positive)\right) \longto \Khred\left(L(\zero)\right)[-\half,\half]\longto \] Here, $[\cdot,\cdot]$ is a shift in the bi-grading via $\Khred^\delta_q(L)[i,j]=\Khred^{\delta-i}_{q-j}(L)$, and $c=n_-(L(\one))-n_-(L(\positive))$, the difference in the number of negative crossings for some choice of orientation on the affected components of the resolution $L(\one)$ (these grading conventions are consistent with \cite{MO2007,Rasmussen2005}). Note that the connecting homomorphism raises the primary grading by 1. Similarly, for a link with distinguished negative crossing $L(\leftcross)$ we have 
\[\longto \Khred\left(L(\zero)\right)[\half,-\half]\longto \Khred\left(L(\negative)\right) \longto \Khred\left(L(\one)\right)[\textstyle-\frac{1}{2}(c+1),\frac{1}{2}(3c+1)]\longto \]

Omitting grading shifts for the moment, and simplifying with the notation $\positive$ for $L(\positive)$, these exact sequences are often represented by exact triangles of the form
\[\xymatrix@C=15pt@R=10pt{
	 & {\Khred(\positive)}\ar[dr] & \\
	{\Khred(\one)}\ar[ur] && {\Khred(\zero)}\ar@{-->}_{[1,0]}[ll] \\
}\]
Since we are working over a field, the homology $\Khred(L)$ is completely determined by the groups $\Khred(\zero)$ and $\Khred(\one)$, together with the connecting homomorphism. This leads directly to the notion of a mapping cone (see \cite[Section 4]{OSz2005-branch}, for example), which is a useful point of view in the present setting. That is, we have \[\Khred(\positive)\cong H_*\left(\Khred(\zero)\to\Khred(\one)\right)\] where the connecting homomorphism raises $\delta$-grading by one as above. 

Replacing the grading shifts, we have
\begin{align*}
\Khred(\rightcross) &\cong H_*\left( \Khred(\zero)[-\half,\half]\to\Khred(\one)[\textstyle-\frac{1}{2}c,\frac{1}{2}(3c+2)]\right) \\
\Khred(\leftcross) &\cong H_*\left( \Khred(\one)[\textstyle-\frac{1}{2}(c+1),\frac{1}{2}(3c+1)]\to\Khred(\zero)[\half,-\half]\right)
\end{align*} 

The singly $\delta$-graded group will be useful in many instances, and in this setting the mapping cones simplify to yield 
\begin{align*}
\Khred(\rightcross) &\cong H_*\left( \Khred(\zero)[-\half]\to\Khred(\one)[\textstyle-\frac{1}{2}c]\right) \\
\Khred(\leftcross) &\cong H_*\left( \Khred(\one)[\textstyle-\frac{1}{2}(c+1)]\to\Khred(\zero)[\half]\right)
\end{align*} 
where $[\cdot]$ shifts the $\delta$-grading.  

\subsection{Normalization and Support}\label{sub:support}

In calculations involving the skein exact sequence absolute gradings are essential. Therefore, we will generally need to fix an orientation, although the final result (as a relatively graded group) will not depend on this choice. 

In particular, $w(L)$ depends only on $\Khred(L)$ as a relatively graded group, however determining this quantity in practice will make use of absolute gradings. For this reason we introduce the notion of support, denoted by $\Supp(\Khred(L))$, as an absolutely $\bZ$-graded quantity. We define the support $\Supp(\Khred(L))$ as the gradings $\delta_1,\ldots,\delta_k$ having fixed an absolute $\bZ$-grading for $\Khred(L)\cong \bigoplus_{\delta=1}^k\bF^{b_\delta}$. In particular,  
the mapping cone \[\Khred(\rightcross) \cong H_*\left( \Khred(\zero)[-\half]\to\Khred(\one)[\textstyle-\frac{1}{2}c]\right)\] 
may be presented as
\[\Khred(\rightcross)\cong H_*\left(
\raisebox{15pt}{\xymatrix@C=15pt@R=12pt{
	\bF^{b_1}\ar[dr] & \bF^{b_2}\ar[dr] & {\cdots}\ar[dr] & \bF^{b_k}\\
	\bF^{b_1'}& \bF^{b_2'} & \cdots & \bF^{b_k'} 
}}\right)\] where $\textstyle\Khred(\zero)[-\half]\cong\bigoplus_{\delta=1}^k\bF^{b_\delta}$ and $\textstyle\Khred(\one)[\textstyle-\frac{1}{2}c]\cong\bigoplus_{\delta=1}^k\bF^{b'_\delta}$ for $b_i,b_i'\ge0$, since the connecting homomorphism raises $\delta$-grading by 1, provided \[\Supp\left(\Khred(\one)[\textstyle-\frac{1}{2}c]\right)\subseteq\Supp\left(\Khred(\zero)[-\half]\right).\] 

The following will be a useful absolutely $\bZ$-graded object:
\begin{definition} The $\si$-normalized Khovanov homology is an absolutely $\bZ$-graded theory defined by $\Khsig(L)=\Khred(L)[-\frac{\si(L)}{2}]$, where $\si(L)$ denotes the signature of the link $L$.\end{definition}

Note that for 2-component links, the $\si$-normalized Khovanov homology is an invariant of unoriented links. This turns out to be a natural absolute grading to consider, despite the fact that we are ultimately interested in the relative grading. Of course, $\Khsig(L)$ and $\Khred(L)$ coincide as relatively $\bZ$-graded groups. 

\subsection{The Manolescu-Ozsv\'ath exact sequence}

As a singly graded theory, there is a useful special case in which the skein exact sequence simplifies nicely in terms of the $\si$-normalization.

\begin{proposition}[{Manolescu-Ozsv\'ath \cite[Proposition 5]{MO2007}}]\label{prp:mo}
Let $L=L(\positive)$ be a link with some distinguished crossing, and set $L_0=L(\zero)$ and $L_1=L(\one)$. If $\det(L_0),\det(L_1)>0$ and $\det(L)=\det(L_0)+\det(L_1)$ then 
\[\textstyle \Khsig(L)=H_\ast\left(\Khsig(L_0)\to\Khsig(L_1)\right).\]
\end{proposition}

In the standard notation, this takes the form \[\textstyle \Khred(L)[-\frac{\si}{2}]=H_\ast\left(\Khred(L_0)[-\frac{\si_0}{2}]\to\Khred(L_1)[-\frac{\si_1}{2}]\right).\] where $\sigma=\sigma(L)$, $\sigma_0=\sigma(L_0)$ and $\sigma_1=\sigma(L_1)$ (see \cite{MO2007}). Notice that in this setting the orientation of the resolved crossing does not play a role so that a single expression replaces the pair of exact sequences.

\subsection{On the signature of a link}

We briefly review the work of Gordon and Litherland, constructing the signature of a link via the Goeritz matrix \cite{GL1978}. The conventions we adopt are those of Manolescu and Ozsv\'ath \cite{MO2007}, since our interest will be in proving a degenerate form of Proposition \ref{prp:mo}.

The complement of a projection of a link $L$ is divided into regions that may be coloured black and white in an alternating fashion to obtain the checkerboard colouring. Denote the white regions by $R_0,R_1,\ldots,R_n$. We may assume that every crossing $c$ of the diagram for $L$ is incident to distinct white regions, and assign an incidence number $\mu(c)$ and type by the conventions of Figure \ref{fig:Goeritz-conventions}.
\begin{figure}[ht!]
  \begin{center}
  \labellist \small \pinlabel {$\mu=+1$} at 305 310 \endlabellist
  \includegraphics[scale=0.25]{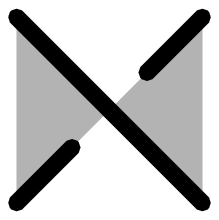}\qquad\qquad
  \labellist \small  \pinlabel {$\mu=-1$} at 305 310 \endlabellist
  \includegraphics[scale=0.25]{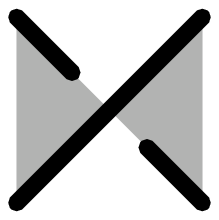}\qquad\qquad
  \labellist \small \pinlabel {Type I} at 305 310 \endlabellist
  \includegraphics[scale=0.25]{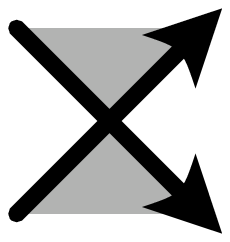}\qquad\qquad
  \labellist \small \pinlabel {Type II} at 305 310 \endlabellist
  \includegraphics[scale=0.25]{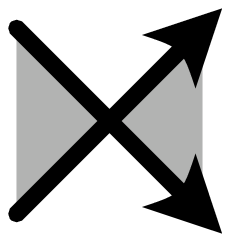}
        \caption{Incidence numbers and crossing types.}
    \label{fig:Goeritz-conventions}
  \end{center}
\end{figure}

The incidence number of the diagram for $L$ is obtained by taking the sum of incidences over crossings of type II. Setting \[\mu(L)=\sum_{c\, {\rm of\, type\, II}}\mu(c),\] the Goeritz matrix of $G$ for the diagram of $L$ is the $n\times n$ symmetric matrix 
\[g_{ij}=\begin{cases}
-\sum_{c\in R_{ij}}\mu(c) & i\ne j \\
-\sum_{i\ne k}g_{ik} & i=j
\end{cases}\]  
where the second sum starts at $k=0$ and $R_{ij}=\overline{R_i}\cap\overline{R_j}$ for $i,j\in\{1,\ldots,n\}$.

From the work of Gordon and Litherland \cite{GL1978}, the signature of the link $L$ is given by $\sigma(L)=\operatorname{signature}(G)-\mu(L)$ and $\det(L)=|\det(G)|$.

\subsection{Degenerations}
We now prove a degenerate version of Manolescu and Ozsv\'ath's exact sequence when one of determinants of the pair of resolutions vanishes. Once again, a single expression is obtained in each case.
\begin{proposition}\label{prp:mo-perturbed}
Using the same conventions as Proposition \ref{prp:mo}, if $\det(L_0)=0$ and $\det(L)=\det(L_1)\ne 0$ then \[\textstyle \Khsig(L)=H_\ast\left(\Khsig(L_0)[-\textstyle\frac{1}{2}]\to\Khsig(L_1)\right).\] Similarly, if $\det(L_1)=0$ and $\det(L)=\det(L_0)\ne0$ then
\[\textstyle \Khsig(L)=H_\ast\left(\Khsig(L_0)\to\Khsig(L_1)[\textstyle\frac{1}{2}]\right).\]
\end{proposition}

\begin{proof}
The proof closely follows the argument in \cite{MO2007} establishing Proposition \ref{prp:mo}, and as such we will adopt the same notation. Throughout, $\sigma=\sigma(L)$, $\sigma_0=\sigma(L_0)$ and $\sigma_1=\sigma(L_1)$. There are 2 orientations to consider in each case, hence 4 cases to consider in total. 

\begin{figure}[ht!]
  \begin{center}
    \mbox{
      \subfigure{\includegraphics[scale=0.25]{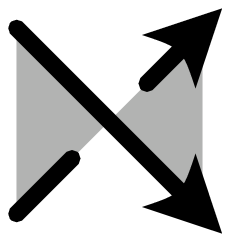}} \qquad\qquad\qquad
      \subfigure{\includegraphics[scale=0.25]{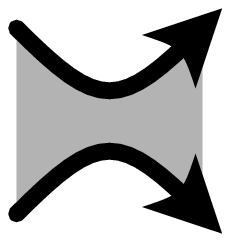}} \qquad\qquad\qquad
      \subfigure{\includegraphics[scale=0.25]{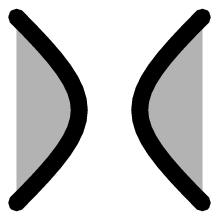}} 
      }
    \caption{Colouring conventions for case 1: $L$, $L_0$ (the oriented resolution) and $L_1$ (the unoriented resolution) at the resolved positive crossing. For case 2 the white and black regions are exchanged to yield the dual colouring.}
    \label{fig:positive-colouring}
  \end{center}
\end{figure}

Case 1: Suppose the distinguished crossing is positive, with $\det(L_0)=0$, and fix a checkerboard colouring of the diagram for $L$ as in Figure \ref{fig:positive-colouring} so that the distinguished crossing is of type II with incidence $\mu=+1$. Now writing $G_1$ for the Goeritz matrix of $L_1$, we have 
\[ G=\begin{pmatrix} a & v \\ v^T & G_1 \end{pmatrix}\quad\textrm{and}\quad G_0=\begin{pmatrix} a-1 & v \\ v^T & G_1 \end{pmatrix}\]
where $G$ and $G_0$ are the Goeritz matrices of $L$ and $L_0$ respectively. As in \cite{MO2007}, we assume without loss of generality that $G_1$ is diagonal (with diagonal entries $\alpha_1,\ldots,\alpha_n$) and write the bilinear form associated with $G$ as 
\[\bigg(a-\sum_{i=1}^n\frac{v_i^2}{\alpha_i}\bigg)x_0^2+\sum_{i=1}^n\alpha_i\left(x_i+\frac{v_i}{\alpha_i}x_0\right)^2.\] 
Similarly, the bilinear form associated with $G_0$ may be written as 
\[\bigg(a-1-\sum_{i=1}^n\frac{v_i^2}{\alpha_i}\bigg)x_0^2+\sum_{i=1}^n\alpha_i\left(x_i+\frac{v_i}{\alpha_i}x_0\right)^2\]
so that setting \[\beta= a-\sum_{i=1}^n\frac{v_i^2}{\alpha_i}\] we obtain 
\[\det(G)=\beta\det(G_1)\ \textrm{and}\ \det(G_0)=(\beta-1)\det(G_1).\]
Now since $0=\det(L_0)=|\det(G_0)|=|\beta-1|\det(L_1)$ and $\det(L_1)\ne0$, we have that $
\beta=+1$ and \[\operatorname{signature}(G)=\operatorname{signature}(G_0)+1=\operatorname{signature}(G_1)+1.\]
Using the Gordon-Litherland formula for the signature we have that
\begin{align*}
\sigma &= \operatorname{signature}(G)-\mu \\
&= \operatorname{signature}(G_0)+1 -(\mu_0+1) \\
&= \sigma_0
\end{align*}
where $\mu=\mu(L)$ and $\mu_0=\mu(L_0)$, while writing $\mu_1=\mu(L_1)$ gives
\begin{align*}
\sigma &= \operatorname{signature}(G)-\mu \\
&= \operatorname{signature}(G_1)+1 -(\mu_1+c+1) \\
&= \sigma_1-c
\end{align*}
as in \cite{MO2007}, noting that the incidence and type of a crossing determines its sign. Now since 
\[\textstyle\Khred\left(L\right)\cong H_*\left(\Khred\left(L_0\right)[-\frac{1}{2}]\to\Khred\left(L_1\right)[-\frac{c}{2}]\right)\]
we have $-1=\sigma-\sigma_0-1$ and $-c=\sigma-\sigma_1$ so that 
\[\textstyle\Khred\left(L\right)[-\frac{\sigma}{2}]\cong H_*\left(\Khred\left(L_0\right)[-\frac{\sigma_0+1}{2}]\to\Khred\left(L_1\right)[-\frac{\sigma_1}{2}]\right).\]
In terms of the $\si$-normalization, \[\textstyle \Khsig(L)=H_\ast\left(\Khsig(L_0)[-\textstyle\frac{1}{2}]\to\Khsig(L_1)\right)\] as claimed.

The remaining cases follow from a similar argument. \end{proof}


\section{Fillings, involutions and tangles}\label{sec:fillings}

We review the basic notions of Dehn surgery that will be required. The material of this section is for the most part standard, see for example Boyer \cite{Boyer2002} and Rolfsen \cite{Rolfsen1976}. 

Let $M$ be a compact, connected, orientable 3-manifold with torus boundary. A slope in $\partial M$ is a primitive class $\alpha\in H_1(\partial M;\bZ)/\pm1$, that is, the isotopy class of an essential simple closed curve in $\partial M$. Since $H_1(\partial M;\bZ)\cong \bZ\oplus \bZ$, the slopes  in $\partial M$ may be parameterized by reduced rational numbers $\{\pq\}\in\bQ\cup\{\frac{1}{0}\}$ once a basis $(\alpha_0,\beta_0)$ for $H_1(\partial M;\bZ)$ has been fixed. That is, any slope may be written in the form $p\alpha_0+q\beta_0$ for relatively prime integers $p$ and $q$, so that the slope  $\alpha_0$ is represented by $\frac{1}{0}$. There is some redundancy in this description that may be taken care of by fixing the convention $q\ge0$, say. Notice that, as a basis for $H_1(\partial M;\bZ)$, we have that $\alpha_0$ and $\beta_0$ intersect geometrically in a single point. More generally, it will be useful to measure the distance between any two slopes by their minimal geometric intersection number, denoted $\Delta(\alpha,\beta)=\left|\alpha\cdot\beta\right|,$ for any $\alpha,\beta\in H_1(M;\bZ)/\pm1$. 

For any slope $\alpha$, denote by $M(\alpha)$ the result of Dehn filling along $\alpha$. For example, given a knot  $K\into S^3$, denote the complement $M=S^3\smallsetminus\nu(K)$ where $\nu(K)$ is an open tubular neighbourhood of the knot $K\into S^3$. In this setting there is a preferred basis for surgery provided by the knot meridian $\mu$, and the longitude of the knot $\lambda$ resulting from the fact that $K$ bounds an oriented surface (a Seifert surface) in $S^3$.  We may choose orientations on $\mu$ and $\lambda$ so that $\mu\cdot\lambda=1$, and this convention will be assumed throughout. 

Now if $\alpha$ is a slope in the boundary of the knot complement, we may write $\alpha=\pm(p\mu+q\lambda)$ for $q\ge 0$. This gives rise to the notation $M(\alpha)=S^3_{p/q}(K)$ for Dehn filling, referred to as surgery on $K$, and fixes the convention $S^3_{1/0}(K)\cong S^3$ for the trivial surgery. By nature of this construction, we have that $\big|H_1(S^3_{p/q}(K);\bZ)\big|=\big|H_1(M(\alpha);\bZ)\big|=\Delta(\alpha,\lambda)$ (see, more generally, Lemma \ref{lem:cm} below).  

\subsection{The rational longitude}

For $M$ as above, suppose that $H_1(M;\bQ)\cong\bQ$. Such manifolds $M$ will be referred to as {\em knot manifolds}. Unless stated otherwise, we will generally make the additional assumption that a knot manifold $M$ is irreducible. However, this is not an essential hypothesis in the following discussion, or in the proof of Lemma \ref{lem:cm} below.

Let $i\co\partial M\into M$ be the inclusion map, inducing a homomorphism $i_*\co H_1(\partial M;\bQ)\to H_1(M;\bQ).$ Omitting the coefficients for brevity, consider the long exact sequence 
\[\xymatrix@C=15pt@R=5pt{
	{\cdots}\ar[r] &
	{H_2(M)}\ar[r] &
	{H_2(M,\partial M)}\ar[r] &
	{H_1(\partial M)}\ar[r]^-{i_*} &
	{H_1(M)}\ar[r] &
	{H_1(M,\partial M)}\ar[r] &
	{\cdots} 
}\] 
Since $\partial M$ is connected, the inclusion $i$ induces an isomorphism $H_0(\partial M)\cong H_0(M)$ so that
 \[\xymatrix@C=15pt@R=5pt{
	{0}\ar[r] &
	{H_2(M)}\ar[r] &
	{H_2(M,\partial M)}\ar[r] &
	{H_1(\partial M)}\ar[r]^-{i_*} &
	{H_1(M)}\ar[r] &
	{H_1(M,\partial M)}\ar[r] &
	{0} 
}\]
Since we are working over a field, by duality we have  $H_2(M)\cong H^1(M,\partial M)\cong H_1(M,\partial M)$ and $H_2(M,\partial M)\cong H^1(M)\cong H_1(M)$ hence
\[\xymatrix@C=15pt@R=5pt{
	{0}\ar[r] &
	{H_1(M,\partial M)}\ar[r] &
	{H_1(M)}\ar[r] &
	{H_1(\partial M)}\ar[r]^-{i_*} &
	{H_1(M)}\ar[r] &
	{H_1(M,\partial M)}\ar[r] &
	{0} 
}\]
Now we observe that $\rk(i_*)=1$, since $b_1(\partial M)=2$ so that $\rk(i_*)=2-\rk(i_*)$ by exactness.
As a result, $i_*\co H_1(\partial M;\bZ)\to H_1(M;\bZ)$ carries a free summand of $H_1(\partial M;\bZ)\cong\bZ\oplus\bZ$ injectively to  $H_1(M;\bZ)\cong \bZ\oplus H$  (for some finite abelian group $H$). Moreover, $\ker(i_*)$ must be generated generated by $k\lm$, for some primitive class $\lm\in H_1(\partial M;\bZ)$, and non-zero integer $k$. 

Note that this class is uniquely defined, up to sign, and hence determines a well-defined slope in $\partial M$. This gives a preferred slope in $\partial M$ for any knot manifold, and in turn motivates the following definition.

\begin{definition} For any knot manifold $M$, the rational longitude $\lm$ is the unique slope with the property that $i_*(\lm)$ is finite order in $H_1(M;\bZ)$.\end{definition}

More geometrically, the rational longitude $\lm$ is characterized among all slopes by the property that a non-empty, finite collection of like-oriented parallel copies of $\lm$ bounds an essential surface in $M$. As with the preferred longitude for a knot in $S^3$, the rational longitude controls the first homology of the manifold obtained by Dehn filling.  

\begin{lemma}\label{lem:cm} For every knot manifold $M$ there is a constant $c_M$ (depending only on $M$) such that \[|H_1(M(\alpha);\bZ)|=c_M\Delta(\alpha,\lm).\] 
\end{lemma}

\begin{proof}
Orient $\lm$ and fix a curve $\mu$ dual to $\lm$ so that $\mu\cdot\lm=1$. This provides a choice of basis $(\mu,\lm)$ for the group $H_1(\partial M;\bZ)\cong\bZ\oplus\bZ$. Under the homomorphism induced by inclusion we have $i_\ast(\mu)=(\ell,u)$ and $i_\ast(\lm)=(0,h)$ as elements of $H_1(M;\bZ)\cong\bZ\oplus H$. Note that for any other choice of class $\mu'$ such that $\mu'\cdot\lm=1$ we have $\mu'=\mu+n\lm$ so that $i_\ast(\mu')=(\ell,u+nh)$.

We claim that $|\ell| =\ord_Hi_\ast(\lm)$.     

Let $\zeta$ generate a free summand of $H_1(M;\bZ)$ so that (the free part of) the image of $\mu$ is $\ell\zeta$ where $i_\ast(\mu)=(\ell,u)\in\bZ\oplus H$. Noting that $H_2(M,\partial M;\bZ)\cong\bZ$ by duality and universal coefficients, let $\eta$ generate $H_2(M,\partial M;\bZ)$ so that $\eta\cdot\zeta=\pm1$ under the intersection pairing $H_2(M,\partial M;\bZ)\otimes H_1(M;\bZ)\to \bZ$.

Now suppose $k=\ord_Hi_\ast(\lm)$. The long exact sequence in homology gives
\[\xymatrix@C=15pt@R=5pt{
	{\cdots}\ar[r] &
	{H_2(M;\bZ)}\ar[r] &
	{H_2(M,\partial M;\bZ)}\ar[r]^-{\partial} &
	{H_1(\partial M;\bZ)}\ar[r]^-{i_*} &
	{H_1(M;\bZ)}\ar[r] &
	{\cdots} \\
	&& {\theta}\ar@{|->}[r] & {k\lm}\ar@{|->}[r] & 0 & 
}\] 
so there is a class $\theta\in H_2(M,\partial M;\bZ)$ with image $k\lm$, where $\theta=a\eta$ for some integer $a\ne 0$. Therefore $k\lm=a\partial\eta$, hence $\partial\eta=\frac{k}{a}\lm$. But since $i_*(\frac{k}{a}\lm)=i_*(\partial\eta)=0$, it must be that $|\frac{k}{a}|=|k|$ so that $|a|=1$. As a result, $\theta = \pm\eta$. In particular, up to a choice of sign $\partial\eta=k\lm$ as an element of $H_1(\partial M;\bZ)$. Now
\[|k|=|\mu\cdot k\lm|=|\mu\cdot\partial\eta|=|\ell\zeta\cdot\eta|=|\ell|\]
as claimed.

For a given slope $\alpha$ write $\alpha=a\mu+b\lm$ so that $i_\ast(\alpha)=(a\ell, au+bh)$. Then  \[H_1(M(\alpha);\bZ)\cong H_1(M;\bZ)/(a\ell,au+bh)\]
has presentation matrix of the form 
\[\begin{pmatrix} 
a\ell & 0  \\
au+bh & Ir
\end{pmatrix}\]
where $r=(r_1,\ldots, r_n)$ specifies the finite abelian group $H=\bZ/r_1\bZ\oplus\cdots\oplus\bZ/r_n\bZ$. Therefore $|H_1(M(\alpha);\bZ)|=a\ell r_1\cdots r_n$. Setting $c_M =\ell r_1\cdots r_n = (\ord_H i_\ast(\lm))|H|$ and noting that $a=\Delta(\alpha,\lm)$ proves the lemma.
\end{proof}

\subsection{Strong inversions and associated quotient tangles}

A knot manifold is called strongly invertible if there is an involution $f\co M\to M$ with 1-dimensional fixed point set intersecting the boundary torus transversely in exactly 4 points. More precisely, \[\fix(f)\cong I\amalg I \amalg \underbrace{S^1\amalg\cdots\amalg S^1}_k\] where $k\ge0$. A knot is called strongly invertible if its complement is strongly invertible.

\begin{definition}\label{def:simple-strong}
Given an irreducible knot manifold $M$ with $H_1(M;\bQ)=\bQ$, suppose that there is a strong inversion $f\in\End(M)$ with the property that $M/f$ is homeomorphic to a ball. Such $M$ will be called a simple, strongly invertible knot manifold. 
\end{definition}  

As a result, any simple strongly invertible knot manifold is naturally the two-fold branched cover of a tangle $\Br(B^3,\tau)$, where $\tau$ is a pair of properly embedded arcs (together with a possibly empty, finite collection of closed components) in $B^3$ given by the image of $\fix(f)$ in the quotient.

\begin{definition}\label{def:associated-tangle} To any simple, strongly invertible knot manifold $M$, the associated quotient tangle $T=(B^3,\tau)$ is obtained by taking $\tau=\operatorname{image}(\fix(f))$. \end{definition}

In this setting, equivalence of tangles is taken up to homeomorphism of the pair (in the sense of Lickorish \cite{Lickorish1981}), and need not fix the boundary in general. Note that the solid torus is a simple, strongly invertible knot manifold. Indeed, according to Lickorish, a tangle is rational if and only if the two-fold branched cover is a solid torus \cite{Lickorish1981}. Using this correspondence between rational tangles and solid tori, a strong inversion may always be extended across a surgery torus giving rise to an involution on any Dehn filling of a simple, strongly invertible knot manifold (moreover, the resulting branch set may be recovered, see Proposition \ref{prp:pq}). This fact is generally attributed to Montesinos \cite{Montesinos1975}. In particular, we have:

\begin{figure}[ht!]
\begin{center}
\raisebox{0pt}{\includegraphics[scale=0.50]{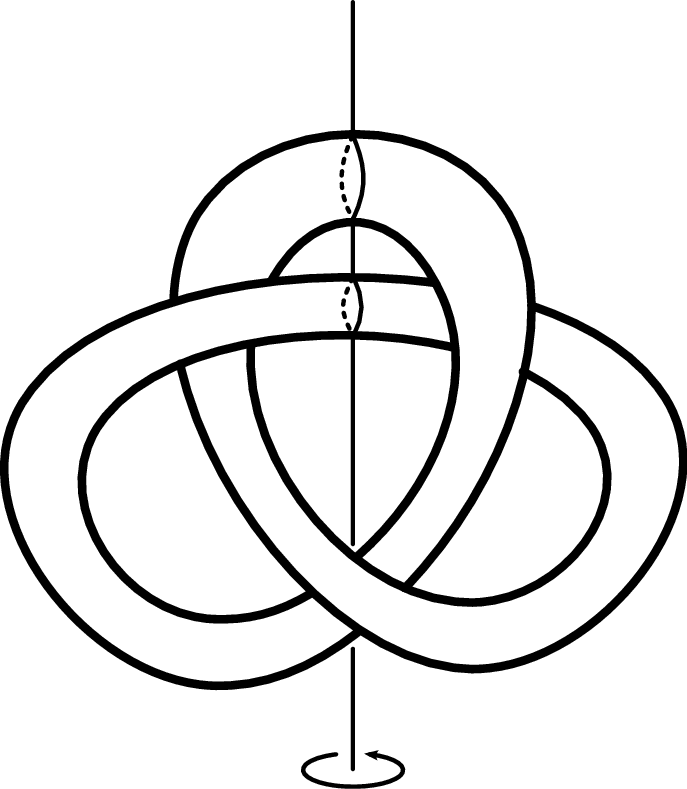}}\qquad
\raisebox{-50pt}{\includegraphics[scale=0.50]{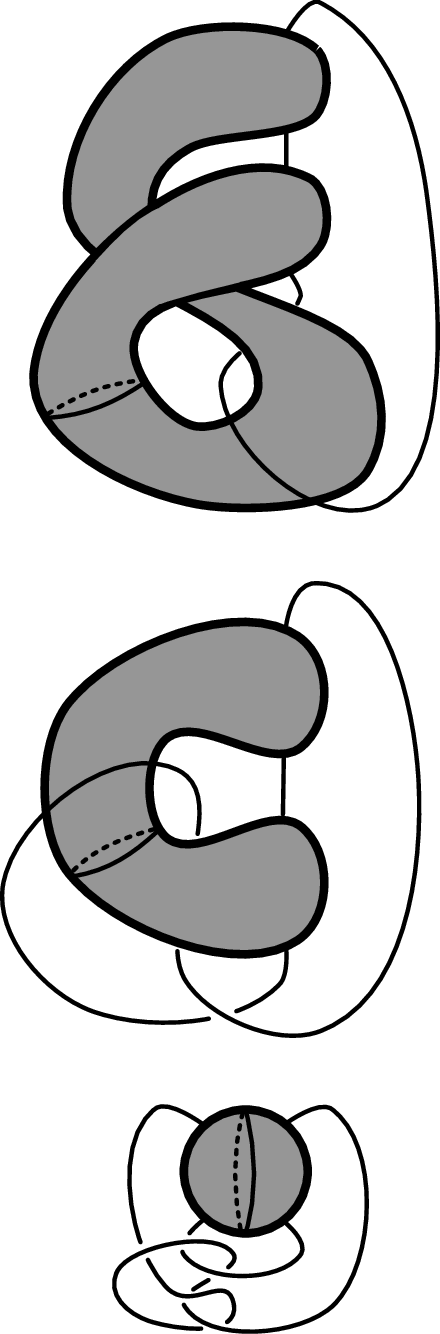}}\qquad
\labellist
	\pinlabel \rotatebox{-90}{$\cong$} at 190 486 
\endlabellist
\raisebox{-30pt}{\includegraphics[scale=0.50]{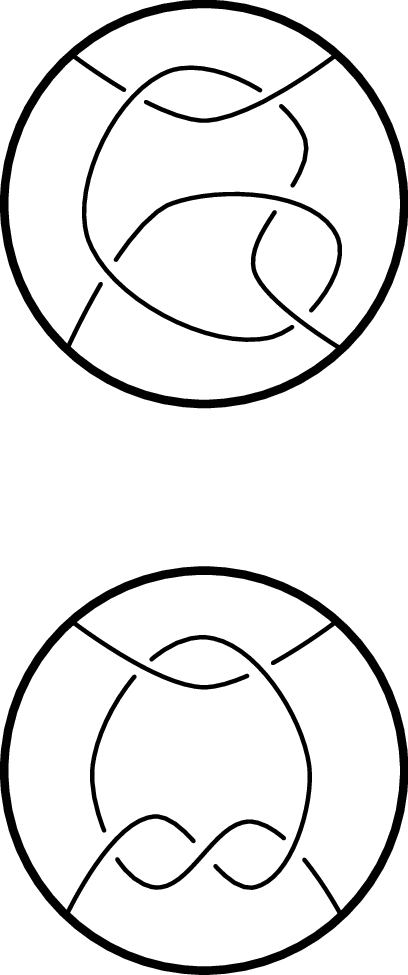}}
\end{center}
\caption{The trefoil with its strong inversion (left), an isotopy of a fundamental domain for the involution (centre), and two homeomorphic views of the tangle associated with the quotient (right). Notice that both representatives of the tangle have the property that $\tau(\overzero)$ is the trivial knot, giving a branch set for the trivial surgery. With a little more care, one may keep track of the image of the preferred longitude in the quotient; see Bleiler \cite{Bleiler1985}, for example. }
\label{fig:trefoil-inversion}\end{figure}

\begin{proposition}\label{prp:si} When $K\into S^3$ is strongly invertible, the complement $M\smallsetminus\nu(K)$ is simple.\end{proposition}
\begin{proof} 
Extending $f$ to $S^3$, across the surgery torus of the trivial surgery, gives the standard involution on $S^3$ by definition of strong invertibility (recall that any involution of $S^3$ is conjugate to the standard one as proved by Waldhausen \cite{Waldhausen1969}, see also \cite{HR1985}). The quotient of this involution is $S^3$, decomposed along a sphere obtained by the quotient of the torus $\partial M$. Since $S^3$ decomposes into a pair of 3-balls for any smooth embedding $S^2\into S^3$, $M/f$ must therefore be homeomorphic to $B^3$. \end{proof}

For example, the trefoil is a strongly invertible knot, and the associated quotient tangle is constructed in Figure \ref{fig:trefoil-inversion}. While complements of strongly invertible knots in $S^3$ provide the primary source of examples of simple, strongly invertible knot manifolds, we remark that the latter is certainly a much larger class. For example, the exterior of a generalized torus knot -- those manifolds Seifert fibred over the disk with two cone points -- always provides such a manifold. 

\begin{proposition}[Montesinos \cite{Montesinos1976}]\label{prp:Montesinos} Let $Y$ be a Seifert fibre space with base orbifold $S^2(p,q,r)$. Then $Y\cong M(\alpha)$ where $M$ is a simple strongly invertible knot manifold and $M$ has Seifert fibre structure with base orbifold $D^2(p,q)$.  
\end{proposition}
\begin{proof}
Let $M$ be a knot manifold endowed with a Seifert fibre structure and suppose that the base orbifold is $D^2(p,q)$, the disk with two cone points. We may assume that $D^2=\{z\in\bC:|z|\le1\}$, and that the cone points $p,q$ lie on the real axis in the interior of $D^2$. Note that such a Seifert fibre space is a union of solid tori along an essential annulus that corresponds to the lift of the imaginary axis in the interior of $D^2$. As we have noted previously, the solid torus admits a strong inversion, and such a strong inversion may be chosen so that it fixes the singular fibre of any Seifert fibre structure on the solid torus. In particular, the solid torus as a Seifert fibre space has base orbifold $D^2$ with a single cone point, and the strong inversion corresponds to a reflection in the real axis. Now the reflection $\rho(z)=\bar{z}$ in the real axis (fixing the cone points $p,q$) lifts to a strong inversion on $M$, and $\rho$ fixes the singular fibres. 

Choose a regular fibre $\fibre\subset\partial M$. By a theorem of Heil, the Dehn filling $M(\fibre)$ must be a connect sum of lens spaces \cite{Heil1974}. Further, extending the strong inversion across the surgery torus gives a strong inversion on $M(\fibre)$, the quotient of which is $S^3$ (with branch set a connect sum of 2-bridge links \cite{HR1985}). As a result, $M/f\cong B^3$ as in the proof of Proposition \ref{prp:si}. 

Now suppose that $Y$ is Seifert fibred, with base orbifold $S^2(p,q,r)$. Removing a tubular neighbourhood of a singular fibre yields a knot manifold $M$ that is Seifert fibred with base orbifold $D^2(p,q)$. Such an $M$ must be simple and strongly invertible.
\end{proof}

Two examples that will be particularly useful in the sequel are the twisted $I$-bundle over the Klein bottle (this is the unique manifold admitting a $D^2(2,2)$ Seifert structure) and the complement of the trefoil knot in $S^3$ (admitting a $D^2(2,3)$ Seifert structure; this is unique up to mirrors).

\subsection{Tangles and Dehn filling}
Consider a simple strongly invertible knot manifold $M$ and  associated quotient tangle $T$. 
Fix a representative for $T$ in which the four endpoints of the tangle lie on the equatorial great circle of $\partial B^3$, as in Figure \ref{fig:arcs}.
Having fixed such a representative, the associated quotient tangle has a pair of distinguished arcs $(\gamma_\overzero,\gamma_0)$ in the boundary of the tangle, as illustrated in Figure \ref{fig:arcs}, that meet in a single point. The hemisphere containing each arc lifts to an essential annulus in $\partial M=\Br(\partial B^3,\partial\tau)$, so that the pair $(\gamma_\overzero,\gamma_0)$ lifts to a (unoriented) basis for $H_1(\partial M;\bZ)$. By fixing an orientation so that $\widetilde{\gamma}_\overzero\cdot\widetilde{\gamma}_0=1$, we obtain a basis for Dehn fillings of $M$.

\begin{figure}[ht!] 
\labellist 
	\pinlabel $T$ at 245 503 
	\pinlabel {$\gamma_\overzero$} at 338 503
	\pinlabel {$\gamma_0$} at 245 417
\endlabellist \centering 
\includegraphics[scale=0.4]{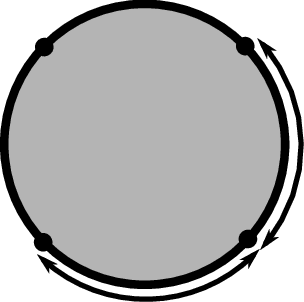} 
\caption{The arcs $\gamma_\overzero$ and $\gamma_0$ in the boundary of $T$.} \label{fig:arcs} 
\end{figure}

We will make use of an action of the 3-strand braid group $B_3=\langle \si_1,\si_2|\si_1\si_2\si_1=\si_2\si_1\si_2\rangle$ on the space of tangles. Braids in this setting are depicted horizontally, read from left to right, with standard generators \[\si_1=\raisebox{-8pt}{\includegraphics[scale=0.30]{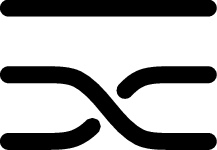}}\qquad\qquad\si_2=\raisebox{-8pt}{\includegraphics[scale=0.30]{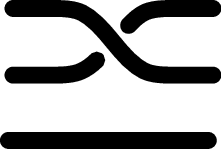}}\] For a given braid $\beta\in B_3$ the action 
is defined by taking $T^\beta$ as depicted in Figure \ref{fig:action}.
\begin{figure}[ht!] 
\labellist 
\pinlabel $T$ at 234 451
\pinlabel $\beta$ at 359 431
\endlabellist 
\centering 
\includegraphics[scale=0.25]{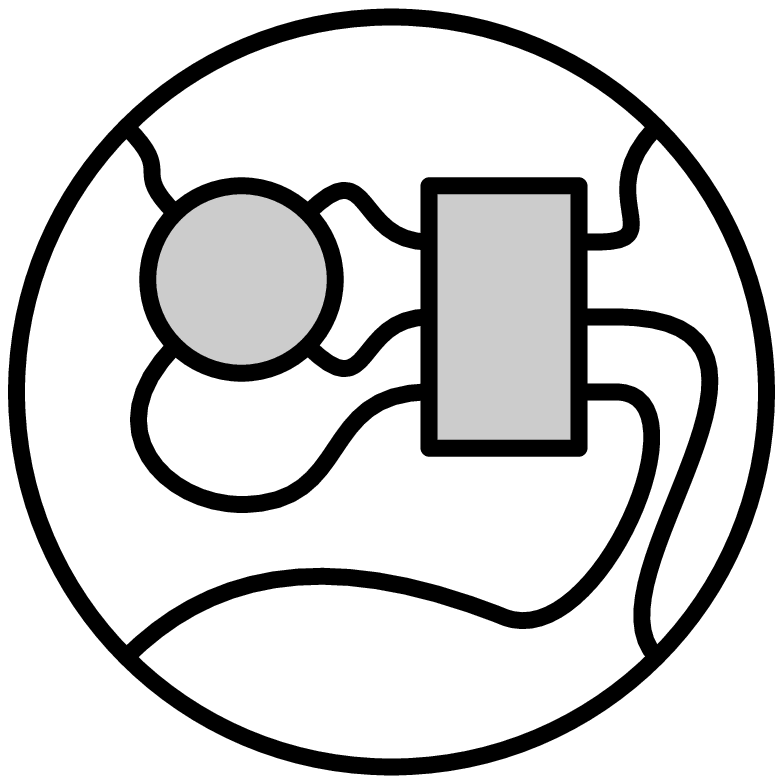} 
\caption{The tangle $T^\beta$ resulting from a tangle $T$ and a 3-braid $\beta$.} 
\label{fig:action} 
\end{figure}

It is straightforward to verify that this is a well defined action on tangles. Notice that this specifies a homeomorphism of the given tangle, and as such this action is trivial when considering tangles up to homeomorphism (though the choice of diagram for a fixed tangle may be altered dramatically). In fact, this may be viewed as a change of framing.

Let $\pq=[a_1,\ldots,a_r]$ be the continued fraction expansion with $a_1\ge 0$ and $a_i>0$ for $i>1$ when $\pq\ge0$ (when $\pq\le0$, $a_1\le 0$ and $a_i<0$ for $i>1$). To $\pq$ we associate the braid 
\begin{equation*} \beta =
\begin{cases}\si_2^{a_1}\si_1^{-a_2}\cdots\si_1^{-a_r} & r\ {\rm even} \\
\si_2^{a_1}\si_1^{-a_2}\cdots\si_2^{a_r} & r\ {\rm odd}\end{cases}\end{equation*} 
Notice that only positive powers of $\si_2$ and negative powers of $\si_1$ occur in the braid word for $\beta$. Define the {\em odd-} and {\em even-} closures of $T$ (also referred to as the {\em numerator} and {\em denominator} closures, respectively), as in Figure \ref{fig:closures}. Now observing that $0=[0]$, and fixing the convention $\overzero=[\ ]$ (with {\em length} $r=0$), denote $\tau(\pq)$ the link obtained by the even or odd closure of $T^\beta$ depending on whether $r$ is even or odd (a particular example is shown in Figure \ref{fig:continued-fraction-1}). 
\begin{figure}[ht!]
  \begin{center}
    \mbox{
      \subfigure{\labellist \pinlabel {$\tau(0)=$} at 150 397 \pinlabel $T$ at 305 397 \endlabellist 
				\raisebox{-9pt}{\includegraphics[scale=0.25]{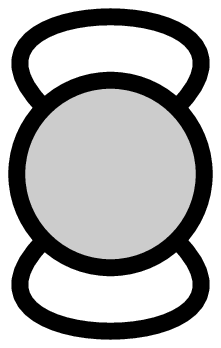}}} \qquad\qquad\qquad
      \subfigure{\labellist \pinlabel {$\tau(\overzero)=$} at 120 397 \pinlabel $T$ at 305 397 \endlabellist 
				\includegraphics[scale=0.25]{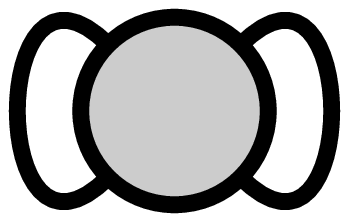}} 
      }
    \caption{The odd-closure $\tau(0)$ and the even-closure $\tau(\overzero)$ of the tangle $T$.}
    \label{fig:closures}
  \end{center}
\end{figure}

Now the strong inversion on $M$ extends to an involution on a Dehn filling of $M$, giving rise to a two-fold branched cover of $S^3$, branched over a link that we may now make explicit.

\begin{proposition}\label{prp:pq}
Let $M$ be a simple strongly invertible knot manifold. For a given slope $\alpha=p\widetilde{\gamma}_\overzero+q\widetilde{\gamma}_0$ we have that $\Br(S^3,\tau(\pq))\cong M(\alpha)$. 
\end{proposition}

\begin{proof}[Sketch of proof]
First observe that $\Br(S^3,\tau(0))\cong M(\widetilde{\gamma}_0)$ and $\Br(S^3,\tau(\overzero))\cong M(\widetilde{\gamma}_\overzero)$. 

Now consider the action of $\si_2$. We claim that this half twist (viewed as an action on the disk with 2 marked points) lifts to a Dehn twist along the curve $\widetilde{\gamma}_\overzero$. Indeed, the two-fold branched cover of the disk specified by a small neighbourhood of $\gamma_{\overzero}$ is an essential annulus in $\partial M$ (c.f \cite[Chapter 10]{Rolfsen1976}). In terms of the basis $(\widetilde{\gamma}_\overzero,\widetilde{\gamma}_0)$, this Dehn twist may be expressed $\begin{pmatrix}1&0\\1&1\end{pmatrix}$. Similarly, the action of $\si_1^{-1}$ lifts to a Dehn twist about $\widetilde{\gamma}_0$; this takes the form $\begin{pmatrix}1&1\\0&1\end{pmatrix}$. In particular, we have that $\Br(S^3,\tau(n))\cong M(n\widetilde{\gamma}_\overzero+\widetilde{\gamma}_0)$ and $\Br(S^3,\tau(\frac{1}{n}))\cong M(\widetilde{\gamma}_\overzero+n\widetilde{\gamma}_0)$.

In general, for $\pq=[a_1,\ldots,a_r]$, the action of the associated braid may be written (in the case $r$ is even) as \[\begin{pmatrix}1&1\\0&1\end{pmatrix}^{a_r}\cdots\begin{pmatrix}1&1\\0&1\end{pmatrix}^{a_2}\begin{pmatrix}1&0\\1&1\end{pmatrix}^{a_1}\]
(the case $r$ odd differs only in the first matrix of this product). We leave it to the reader to check that the first column of the resulting matrix is $\begin{pmatrix}q\\ p\end{pmatrix}$ so that we have specified the filling slope $\alpha=p\widetilde{\gamma}_\overzero+q\widetilde{\gamma}_0$ as desired. Details may be found in Rolfsen \cite[Chapter 10]{Rolfsen1976}, see also Montesinos \cite{Montesinos1975}.
\end{proof}

\begin{corollary}\label{crl:basis}
Given a basis $(\alpha,\beta)$ for surgery in $\partial M$ there is a choice of representative for $T$ so that $(\gamma_\overzero,\gamma_0)$ lifts to $(\alpha,\beta)$. \end{corollary}
\begin{proof}For any choice of representative of $T$, write $\alpha=p\widetilde{\gamma}'_\overzero+q\widetilde{\gamma}'_0$. In terms of this representative then, $M(\alpha)=\Br(S^3,\tau(\pq))$. However, by removing the arcs forming the closure as in Figure \ref{fig:closures}, the resulting tangle may be viewed as a reframing if $T$, and yields a representative compatible with $\alpha$. By twisting along $\alpha$ (i.e. by half twists in the quotient), this representative may be made compatible with $(\alpha,\beta)$ since $\alpha\cdot\beta=1$. \end{proof}

\subsection{Compatibility and preferred representatives} As a result of Corollary \ref{crl:basis}, for any choice of basis $(\alpha,\beta)$ for Dehn surgery on a simple strongly invertible knot manifold, a {\em compatible} representative for the associated quotient tangle exists so that $\alpha=\widetilde{\gamma}_\overzero$ and $\beta=\widetilde{\gamma}_0$. Notice that, as a result of Lemma \ref{lem:cm}, we have that \[\det(\tau(\textstyle\pq))=c_M\Delta(p\widetilde{\gamma}_\overzero+q\widetilde{\gamma}_0,\lm)=c_M\Delta(p\alpha+q\beta,\lm)\] once a basis for Dehn surgery, and compatible associated quotient tangle have been fixed. In particular, given a strongly invertible knot in $S^3$ there is always a choice of associated quotient tangle for which $S^3_{p/q}(K)=\Br(S^3,\tau(\textstyle\pq)).$
Such a representative will be referred to as the {\em preferred} representative for the associated quotient tangle.

Note that $S^3_{p/q}(K)\cong -S^3_{-p/q}(K^\star)$ where $-Y$ denotes the manifold $Y$ with orientation reversed, and $K^\star$ denotes the mirror image of $K$. More generally, \[\textstyle\Br(S^3,\tau(\pq))\cong-\Br(S^3,\tau(\pq)^\star)\cong-\Br(S^3,\tau^\star(-\pq)),\] and as a consequence we will only need to consider non-negative surgery coefficients and continued fractions (up to taking mirrors).

\subsection{Some properties of continued fractions}

There are three fundamental properties for continued fractions relating to Dehn filling that will be essential for the inductive arguments that follow. Since it will always be possible to restrict to non-negative surgery coefficients by passing to the mirror image, we will state these properties for non-negative continued fractions only. 

Therefore, assume that $\pq=[a_1,\ldots,a_r]$ is non-negative, with $a_1\ge0$ and $a_i>0$ for all $i>1$

\begin{property}\label{pro:floor} $\lfloor\pq\rfloor=a_1$ and $\lceil\pq\rceil=a_1+1$. \end{property}
 
\begin{proof} It is immediate from the definition of $\pq$ as a continued fraction that $a_1\le\pq<a_1+1$ for $\pq=[a_1,\ldots,a_r]$.\end{proof} 
 
\begin{property}\label{pro:end} $[a_1,\ldots,a_r,1]=[a_1,\ldots,a_r+1]$.\end{property}
  
\begin{proof} By partial evaluation $[a_1,\ldots,a_r,1]=[a_1,\ldots,a_r+\textstyle\frac{1}{1}]=[a_1,\ldots,a_r+1]$.\end{proof}  

\begin{figure}[ht!]
\begin{center}
\labellist \small \pinlabel $T$ at 100 592 \pinlabel $\simeq$ at 490 580 \endlabellist
\raisebox{0pt}{\includegraphics[scale=0.35]{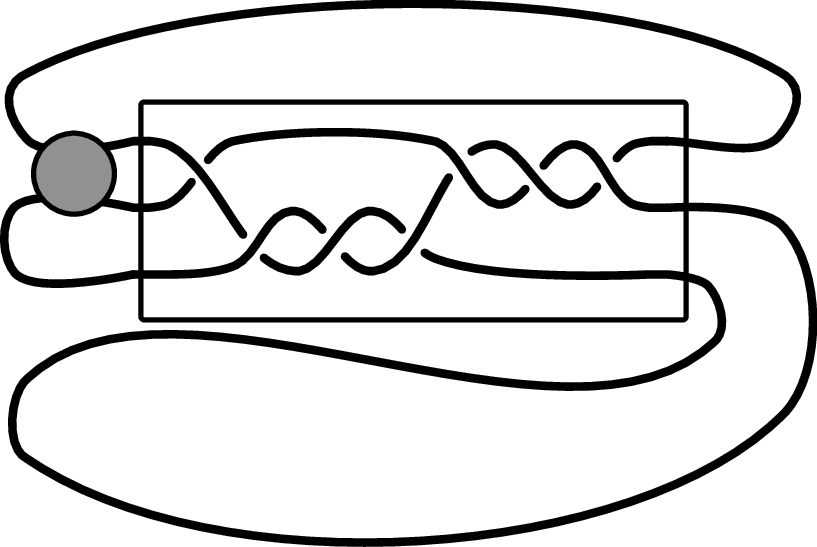}}\qquad
\labellist \small \pinlabel $T$ at 165 510 \endlabellist
\raisebox{18pt}{\includegraphics[scale=0.35]{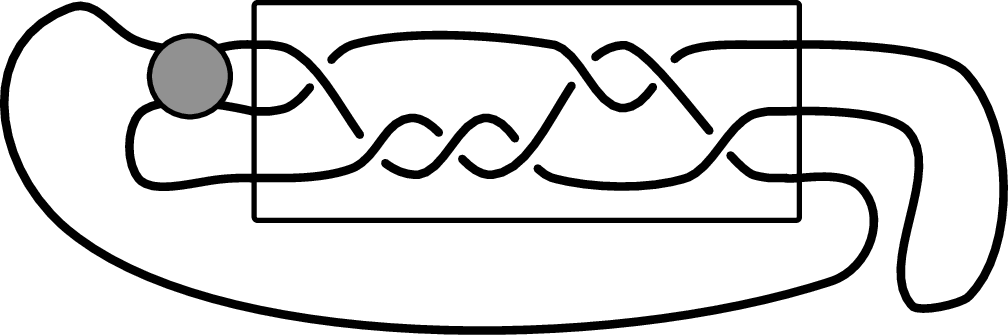}}
\end{center}
\caption{The link $\tau(\frac{13}{10})$ obtained from the odd-closure with the fraction $[1,3,3]$ (left), is isotopic to the link obtained from the even-closure with the fraction $[1,3,2+1]=[1,3,2,1]$ (right).}
\label{fig:continued-fraction-1}\end{figure}
  
It is important to note that this equality of continued fractions manifests itself as isotopic links when forming $\tau(\pq)$, for any tangle. This results from the fact that the even- and odd-closures replace one another, as is illustrated in a particular case in Figure \ref{fig:continued-fraction-1}.   

Finally, we turn to the behaviour of $\tau(\pq)$ under resolutions. 
\begin{definition}\label{def:terminal-crossing} The terminal crossing of $\tau(\pq)$ is the last crossing added by the action of $\beta\in B_3$ specified by the continued fraction. That is, the terminal crossing corresponds to the last generator in the braid word $\beta=\si_2^{a_1}\cdots\si_\epsilon^{a_r}$ (where $\si_\epsilon$ is either $\si_2$ or $\si_1^{-1}$, depending on the parity of $r$).\end{definition}
Our convention will be that the terminal crossing of $\tau(\pq)$ is resolved to obtain the 0-resolution $\tau(\pqzero)$ and the 1-resolution $\tau(\pqone)$. Notice that the 0-resolution is given by one of $[a_1,\ldots,a_{r-1}]$ or $[a_1,\ldots,a_{r-1},a_r-1]$ depending on the parity of $r$, and the 1-resolution is given by the other. By Property \ref{pro:end} we may assume without loss of generality that $a_r>1$ when $\pq=[a_1,\ldots,a_r]$.    

\begin{figure}[ht!]
\begin{center}
\[\xymatrix@C=15pt@R=0pt{
	&{\labellist \small 
	\pinlabel $T$ at 26 475 
\endlabellist
\raisebox{0pt}{\includegraphics[scale=0.32]{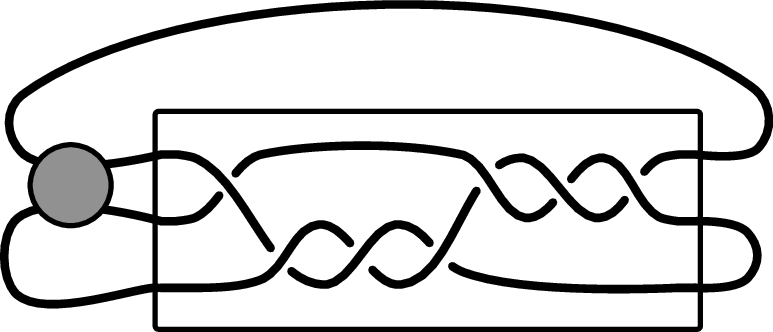}}}\ar@/_2pc/[dl]_{0}\ar@/^2pc/[dr]^{1} & \\
	{\labellist \small \pinlabel $T$ at 128 460 \endlabellist
\raisebox{0pt}{\includegraphics[scale=0.32]{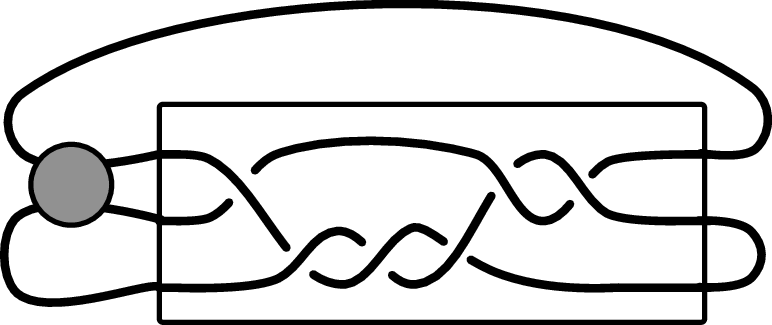}}} && {\labellist \small \pinlabel $T$ at 127 460 \endlabellist
\raisebox{-50pt}{\includegraphics[scale=0.32]{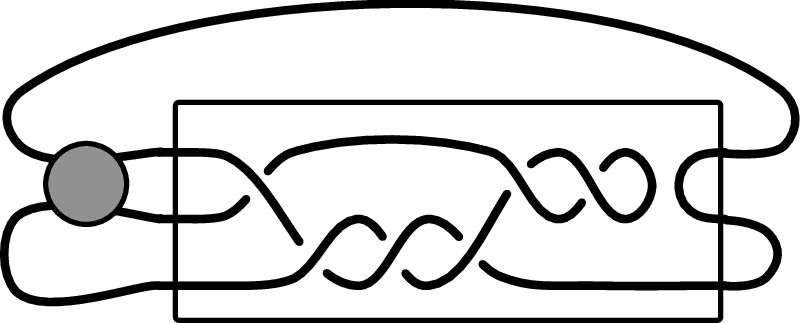}}} }\]
\end{center}
\caption{Resolving the terminal crossing of $\tau(\frac{13}{10})=\tau[1,3,3]$ gives 0-resolution with $\pqzero=[1,3,2]=\frac{9}{7}$ and 1-resolution with $\pqone=[1,3]=\frac{4}{3}$. }
\label{fig:continued-fraction-2}\end{figure}

\begin{property}\label{pro:resolve} $\pq=\frac{p_0+p_1}{q_0+q_1}$ where $\pqzero$ and $\pqone$ are the continued fractions associated with the 0- and 1-resolution of the terminal crossing, respectively.\end{property}
\begin{proof}
Recall that a continued fraction may be recursively defined by convergents $\frac{h_n}{k_n}$ where $h_{-1}=0,h_0=1$ and $h_n=a_nh_{n-1}+h_{n-2}$ for $n>1$, and  $k_{-1}=1,k_0=0$ and $k_n=a_nk_{n-1}+k_{n-2}$ for $n>1$. Now write $\frac{h_{r-1}}{k_{r-1}}=\pqzero$ and $\frac{h_{r}}{k_{r}}=\pqone$, then $\frac{p_0+p_1}{q_0+q_1}=\frac{h_{r}+h_{r-1}}{k_{r}+k_{r-1}}=[a_1,\ldots,a_r-1,1]$, so that applying Property \ref{pro:end} we have $\frac{p_0+p_1}{q_0+q_1}=[a_1,\ldots,a_r]=\pq$ as claimed.
\end{proof}

A particular example of Property \ref{pro:resolve} is illustrated in Figure \ref{fig:continued-fraction-2}. When $\pq=[a_1,\ldots,a_r]$ we will use the notation $\tau(\pq)=\tau[a_1,\ldots,a_r]$ for the closure where convenient.


\section{Branch sets and width bounds}\label{sec:width}

This section has two principle aims. First, we investigate the relationship between Dehn fillings of simple, strongly invertible knot manifolds and quasi-alternating links. The remainder of the section is devoted to establishing upper bounds on the homological width of branch sets associated with Dehn surgery on strongly invertible knots in $S^3$.  

\subsection{Quasi-alternating links}\label{sub:qa}

Given a strongly invertible knot $K\into S^3$, with fixed strong inversion, let $T=(B^3,\tau)$ be the associated quotient tangle. Choosing the preferred representative for $T$, we have that $\det(\tau(n))=|H_1(\Br(S^3,\tau(n));\bZ)|=|H_1(S^3_{n}(K);\bZ)|=n$ for $n>0$. Moreover, resolving the terminal crossing of $\tau(n+1)$ we have $\det(\tau(n+1))=\det(\tau(n))+\det(\tau(\overzero))$. This leads naturally to the notion of a quasi-alternating link.

\begin{definition}\label{def:qa}The set of quasi-alternating links $\sQ$ is the smallest set of links containing the trivial knot, and closed under the following relation: if $L$ admits a projection with distinguished crossing $L(\positive)$ so that \[\det(L(\positive))=\det(L(\zero))+\det(L(\one))\] for which $L(\zero),L(\one)\in\sQ$, then $L=L(\positive)\in\sQ$ as well. 
\end{definition} 

Ozsv\'ath and Szab\'o show that non-split, alternating links are quasi-alternating, and that $\Br(S^3,L)$ is an L-space whenever $L$ is quasi-alternating \cite{OSz2005-branch}. Recall that an L-space is a rational homology sphere $Y$ for which $\rk\HFhat(Y)=|H_1(Y;\bZ)|$ where $\HFhat$ denotes the Heegaard Floer homology of $Y$ (with $\bF$ coefficients). Manolescu and Ozsv\'ath have shown that quasi-alternating links are homologically thin \cite{MO2007}, generalizing Lee's result that non-split alternating links are thin \cite{Lee2005}. More generally:
\begin{proposition}\label{prp:branch-over-thin}The two-fold branched cover of a thin link is always an L-space.\end{proposition}
\begin{proof}When $L$ is a thin link, $\rk\Khred(L)=\det(L)$ as a consequence of Proposition \ref{prp:Kh-Euler}. However, as a result of the spectral sequence relating Khovanov homology and Heegaard Floer homology for two-fold branched covers, $\det(L)\le\rk\HFhat(\Br(S^3,L))\le\rk\Khred(L)$ \cite[Corollary 1.2]{OSz2005-branch}. Since $\det(L)=|H_1(\Br(S^3,L);\bZ)|$, the result follows.\end{proof}
Note however that the converse is false: the Poincar\'e homology sphere arises as the two-fold branched cover of $10_{124}$ (shown in Figure \ref{fig:Kh-two-examples}). This manifold is an L-space with thick branch set.

Let $M$ be a simple, strongly invertible, knot manifold. Suppose $\alpha$ and $\beta$ are a pair of slopes in $\partial M$ with $\alpha\cdot\beta=+1$. Fix a compatible representative for the associated quotient tangle $T=(B^3,\tau)$ with the property that $M(\alpha)=\Br(S^3,\tau(\frac{1}{0}))$ and $M(\beta)=\Br(S^3,\tau(0))$.

\begin{proposition}\label{prp:qa-base-case}If $\tau(\frac{1}{0})$ and $\tau(0)$ are quasi-alternating, and $\alpha\cdot\lm,\beta\cdot\lm>0$, then $\tau(1)$ is quasi-alternating as well.\end{proposition}

\begin{remark}Note that the quasi-alternating hypothesis ensures that $\det(\tau(\overzero))$ and $\det(\tau(0))$ must be non-zero, hence neither $\alpha$ nor $\beta$ coincides with the rational longitude.\end{remark}

\begin{proof}[Proof of Proposition \ref{prp:qa-base-case}]
We need to calculate $\det(\tau(1))$. To this end, by applying Lemma \ref{lem:cm} we have that
\begin{align*}
\det(\tau(1)) &= |H_1(M(\alpha+\beta);\bZ)| \\
&= c_M\Delta(\alpha+\beta,\lm) \\
&= c_M|(\alpha+\beta)\cdot\lm| \\
&= c_M|\alpha\cdot\lm+\beta\cdot\lm| \\
&= c_M|\alpha\cdot\lm|+c_M|\beta\cdot\lm| \\
&= c_M\Delta(\alpha,\lm)+c_M\Delta(\beta,\lm) \\
&= |H_1(M(\alpha);\bZ)|+|H_1(M(\beta);\bZ)| \\
&= \det(\tau(\textstyle\frac{1}{0})) + \det(\tau(0)), 
\end{align*}
which verifies that $\tau(1)$ is a quasi-alternating link, since both $\tau(\overzero)$ and $\tau(0)$ are quasi-alternating by hypothesis.
\end{proof}

\begin{remark}The condition on intersection with $\lm$ may be relaxed at the expense of taking mirrors. For any $M(\alpha)$ and $M(\beta)$ with quasi-alternating branch sets $\tau(\overzero)$ and $\tau(0)$ respectively, we can ensure positive intersection with $\lm$ at the expense of $\alpha\cdot\beta=\pm1$. In the case that $\alpha\cdot\beta=-1$, the same argument works by passing to mirrors. Any quasi-alternating link has quasi-alternating mirror image and as a result if $\tau(\overzero)$ and $\tau(0)$ are quasi-alternating then one of $\tau(-1)$ or $\tau(1)$ is quasi-alternating. \end{remark}

\begin{definition}A triad of links $(\tau(\frac{1}{0}),\tau(0),\tau(1))$ corresponds to a triple of slopes $(\alpha,\beta,\alpha+\beta)$ in the boundary of $M=\Br(B^3,\tau)$ where $\alpha\cdot\beta=1$, $\alpha\cdot\lm>0$, and $\beta\cdot\lm>0$.\end{definition}

The requirement that $\alpha$ and $\beta$ intersect positively with $\lm$ is stronger than necessary, since it is attainable up to taking mirrors. However, with this assumption we have: 

\begin{theorem}\label{thm:quasi-alternating}A triad of links, for which $\tau(\frac{1}{0})$ and $\tau(0)$ are quasi-alternating, gives rise to an infinite family of quasi-alternating links $\tau(\pq)$, for $\pq\ge0$.\end{theorem}
\begin{proof} First observe that $\tau(n)$ is quasi-alternating for every $n\ge 0$. This is immediate by induction in $n$, since $\tau(0)$ is quasi-alternating and 
\begin{align*}
\det(\tau(n)) &= |H_1(M(n\alpha+\beta);\bZ)| \\
&= c_M\Delta(n\alpha+\beta,\lm) \\
&= c_M|(n\alpha+\beta)\cdot\lm| \\
&= c_M|n\alpha\cdot\lm+\beta\cdot\lm| \\
&= c_M|\alpha\cdot\lm|+c_M|(n-1)\alpha+\beta\cdot\lm| \\
&= c_M\Delta(\alpha,\lm)+c_M\Delta((n-1)\alpha+\beta,\lm) \\
&= |H_1(M(\alpha);\bZ)|+|H_1(M((n-1)\alpha+\beta);\bZ)| \\
&= \det(\tau(\textstyle\frac{1}{0})) + \det(\tau(n-1)), 
\end{align*}
for $n>0$ as in Proposition \ref{prp:qa-base-case}.

For $\tau(\pq)$, we need a second induction in the length of the continued fraction $\pq=[a_1,\ldots,a_r]$. The base case $r=1$ is the observation above that $\tau(n)$ is quasi-alternating, applying Property \ref{pro:floor}.

Suppose then that $\tau(\pq)$ is quasi-alternating for all $\pq\ge0$ that may be represented by a continued fraction of length $r-1$. By resolving the terminal crossing and applying Property \ref{pro:resolve} for $\pq=[a_1,\ldots,a_r]$,
\begin{align*}
\det(\tau(\textstyle\pq)) 
&= |H_1(M(p\alpha+q\beta);\bZ)| \\
&= c_M\Delta(p\alpha+q\beta,\lm) \\
&= c_M|(p\alpha+q\beta)\cdot\lm| \\
&= c_M|(p_0+p_1)\alpha\cdot\lm+(q_0+q_1)\beta\cdot\lm| \\
&= c_M|(p_0\alpha+q_0\beta)\cdot\lm|+c_M|(p_1\alpha+q_1\beta)\cdot\lm| \\
&= c_M\Delta(p_0\alpha+q_0\beta,\lm)+c_M\Delta(p_1\alpha+q_1\beta,\lm) \\
&= |H_1(M(p_0\alpha+q_0\beta);\bZ)|+|H_1(M(p_1\alpha+q_1\beta);\bZ)| \\
&= \det(\tau(\textstyle\pqzero)) + \det(\tau(\textstyle\pqone))
\end{align*}
where $\pqzero$ and $\pqone$ are the continued fractions $[a_1,\ldots,a_{r-1}]$ and $[a_1,\ldots,a_{r}-1]$. This gives a continued fraction of length $r-1$ for which the corresponding link must be quasi-alternating by the induction hypothesis, and a continued fraction $[a_1,\ldots,a_{r}-1]$ with $r^{\rm th}$ entry reduced by one. 

Since $[a_1,\ldots,a_{r-1},1]=[a_1,\ldots,a_{r-1}+1]$ by Property \ref{pro:end}, repeating the above argument $a_r-1$ times (i.e. a second induction in $a_r$ as in the case $\tau(n)$) completes the induction.
\end{proof}

\subsection{Branch sets for L-spaces obtained from Berge knots}

Theorem \ref{thm:quasi-alternating} gives a tool with which to build large classes of quasi-alternating links and study the overlap between certain classes of L-spaces. In particular, we have that any quasi-alternating knot gives rise to an infinite family of quasi-alternating links, analogous to the behaviour of L-spaces. For example, it is well known that any sufficiently large surgery on a torus knot, or more generally a Berge knot, gives rise to an L-space \cite{OSz2005-lens}.  

\begin{proposition}\label{prp:Berge-quasi-alternating} For large enough integer surgery coefficient $N$, the branch set for $S^3_N(K)$ is quasi-alternating whenever $K$ is a Berge knot. Moreover, for every $\pq\ge N$ the branch set associated with $\pq$-surgery on $K$ must be quasi-alternating. \end{proposition}

\begin{proof} For any Berge knot $K$ there is some integer $N$, positive up to taking mirrors, with the property that $S^3_N(K)$ is a lens space \cite{Berge}. As a result, $(\mu,N\mu+\lambda,(N+1)\mu+\lambda)$ gives a triad of slopes, in terms of the preferred basis $(\mu,\lambda)$ for $K$.

Moreover, since Berge knots are strongly invertible \cite{Osborne1981}, there is an associated quotient tangle $T=(B^3,\tau)$ with representative so chosen so that $\tau(\frac{1}{0})$ is unknotted, and $S^3_N(K)=\Br(S^3,\tau(0))$. By construction, both branch sets are quasi-alternating: the trivial knot $\tau(\overzero)$ and some non-split 2-bridge link $\tau(0)$ since $\Br(S^3,\tau(0))$ is a lens space \cite{HR1985}. 

Now applying Theorem \ref{thm:quasi-alternating}, $\tau(\pq)$ must be quasi-alternating for every $\pq\ge0$, so that the L-space $S^3_{(Nq+p)/q}(K)$  is branched over $S^3$ with quasi-alternating branch set $\tau(\pq)$.
\end{proof}

As a result, many of the L-spaces arising as surgery on a Berge knot are also obtained as two-fold branched covers of quasi-alternating links. This implies in particular that the corresponding branch sets have thin Khovanov homology by work of Manolescu and Ozsv\'ath \cite{MO2007}. Although this cannot be the case for all possible fillings when $K$ is non-trivial (see Theorem \ref{thm:trivial-knot}), it turns out that in terms of homological width, the branch set corresponding to a filling of a Berge knot cannot be too much more complicated.

\begin{proposition}\label{prp:width-bound-Berge} Surgery on a Berge knot has branch set with width at most 2.\end{proposition}

The proof of this statement (given in Section \ref{sub:upper}) is a consequence of a particular stable behaviour for branch sets associated with Dehn surgery, which we develop in the following sections.

\subsection{A mapping cone for integer surgeries}

Given a strongly invertible knot $K\into S^3$, with fixed strong inversion, let $T=(B^3,\tau)$ be the preferred representative of the associated quotient tangle. Therefore, $\tau(\overzero)$ is the trivial knot, and $S^3_0(K)\cong\Br(S^3,\tau(0))$. As a result, $\Khred(\tau(\overzero))\cong \bF$, and $w(\tau(0))>1$ since $\det(\tau(0))=0$ (see Proposition \ref{prp:homologically-thick}). Notice that $\tau(0)$ is a two component link. 

\labellist 
\pinlabel $T$ at 305 397 \endlabellist 
\parpic[r]{$\begin{array}{c}\raisebox{0pt}{\includegraphics[scale=0.25]{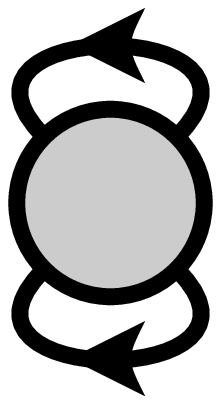}}\end{array}$}
In the interest of studying the Khovanov homology of the branch sets associated with integer surgery, we choose the orientation on $\tau(0)$ shown on the right. That this is possible follows from the fact that $\tau(\overzero)$ is the trivial knot; that such a choice is copacetic results from the fact that $\Khred(\tau(0))$, in the present context, is a relatively bi-graded group. Recall that only the absolute grading depends on orientation, as in Section \ref{sec:kh}, so we are free to fix any orientation. With this orientation on $\tau(0)$, there is a natural constant related to a fixed diagram for a representative of the associated quotient tangle $c_\tau=n_-(\tau(\textstyle\overzero))-n_-(\tau(0))$. Since $\tau(\overzero)$ has a single component, $c_\tau$ is independent of choice of orientation on $\tau(\overzero)$. 

For example, we may rewrite the mapping cones in Khovanov homology as
\[\textstyle\Khred(\tau(1))\cong 
H_*\left(\Khred(\tau(0))[-\frac{1}{2},\frac{1}{2}]\to\Khred(\tau(\overzero))[-\frac{c_\tau}{2},\frac{3c_\tau+2}{2}]\right)\] since $c=n_-(\tau(\overzero))-n_-(\tau(1))=n_-(\tau(\overzero))-n_-(\tau(0))=c_\tau$, and 
\[\textstyle\Khred(\tau(-1))\cong 
H_*\left(\Khred(\tau(\overzero))[-\frac{c_\tau}{2},\frac{3c_\tau-2}{2}]\to\Khred(\tau(0))[\frac{1}{2},-\frac{1}{2}]\right)\] since $c=n_-(\tau(\overzero))-n_-(\tau(-1))=n_-(\tau(\overzero))-n_-(\tau(0))-1=c_\tau-1$. Notice that in this second case there is an overall $[1,0]$ shift  (which may be ignored, as our interest is in the relative gradings and not the absolute gradings) so that 
\[\textstyle\Khred(\tau(-1))\cong 
H_*\left(\Khred(\tau(\overzero))[-\frac{c_\tau}{2},\frac{3c_\tau-2}{2}][-1,0]\to\Khred(\tau(0))[-\frac{1}{2},-\frac{1}{2}]\right)[1,0]\] which allows comparison of the homology of $\tau(\pm 1)$ in terms of $\Khred(\tau(0))$ and the new generator $\Khred(\tau(\overzero))\cong\bF$. More generally, we have:
\begin{lemma}\label{lem:general-stability} For any integer $m$, and positive integer $n$, \[\textstyle\Khred(\tau(m+n))\cong H_*\left(\Khred(\tau(m))\to\bigoplus_n\Khred(\tau(\overzero))\right)\] as a relatively $\bZ\oplus\bZ$-graded group, where the integer $m$ may be interpreted as a change of framing. More precisely, there exist explicit constants $x$ and $y$ and an identification \[\textstyle\bigoplus_{q=0}^{n-1}\Khred(\tau(\textstyle\overzero))[x,y][0,q]\cong\bF[\bZ/n\bZ]\] as graded $\bF$-vector spaces so that \[\textstyle\Khred(\tau(m+n))\cong H_*\left(\Khred(\tau(m))\to\bF[\bZ/n\bZ]\right).\]\end{lemma} 
\begin{proof}
This amounts to careful iterated application of the mapping cone for resolution of a positive crossing applied to the $n$ positive crossings in $\tau(m+n)$. When $n=1$ we have
\[\textstyle\Khred(\tau(m+1))\cong H_*\left(\Khred(\tau(m))[-\frac{1}{2},\frac{1}{2}]\to\Khred(\tau(\overzero))[-\frac{k_\tau}{2},\frac{3k_\tau+2}{2}]\right)\] where $k_\tau=c_\tau+m$. Set $[x,y]=[-\frac{k_\tau}{2},\frac{3k_\tau+2}{2}]$. Now when $n=2$ we obtain 
\begin{align*}
\textstyle\Khred(\tau(m+2)) &\cong \textstyle H_*\left(\Khred(\tau(m+1))[-\frac{1}{2},\frac{1}{2}]\to\Khred(\tau(\overzero))[-\frac{k_\tau+1}{2},\frac{3(k_\tau+1)+2}{2}]\right) \\ 
&\cong \textstyle H_*\left(\Khred(\tau(m+1))[-\frac{1}{2},\frac{1}{2}]\to\Khred(\tau(\overzero))[x,y][-\half,\half][0,1]\right) 
\end{align*}
or, by unpacking the group $\Khred(\tau(m+1))$ as in the previous case, 
\begin{align*}
&\textstyle\Khred(\tau(m+2)) \\& \cong\textstyle H_*\left(H_*\left(\Khred(\tau(m))[-\frac{1}{2},\frac{1}{2}]\to\Khred(\tau(\overzero))[x,y]\right)[-\frac{1}{2},\frac{1}{2}]\to\Khred(\tau(\overzero))[x,y][-\half,\half][0,1]\right)\end{align*} as an iterated mapping cone. Said another way, this expression is simply the repeated application of the long exact sequence. This simplifies considerably however, since the two occurrences of the group $\Khred(\tau(\overzero))\cong\bF$ appear in the same $\delta$-grading. Since the differential of the mapping cone (or, the connecting homomorphism of the long exact sequence) raises $\delta$-grading by one, there cannot be a differential between the copies of $\Khred(\tau(\overzero))$. As result, 
\begin{align*}
&\textstyle\Khred(\tau(m+2)) \\
&\textstyle \cong H_*\left(\Khred(\tau(m))[-1,1]\to\Khred(\tau(\overzero))[x,y][-\half,\half]\oplus\Khred(\tau(\overzero))[x,y][-\half,\half][0,1]\right) \\
&\textstyle \cong H_*\left(\Khred(\tau(m))[-1,1]\to\bigoplus_{q=0}^1\Khred(\tau(\overzero))[x,y][-\half,\half][0,q]\right) \\
&\textstyle \cong H_*\left(\Khred(\tau(m))[-\half,\half]\to\bigoplus_{q=0}^1\Khred(\tau(\overzero))[x,y][0,q]\right)[-\half,\half]
\end{align*}
Now suppose for induction that 
\[\textstyle \Khred(\tau(m+n-1)) \cong H_*\left(\Khred(\tau(m))[-\half,\half]\to\bigoplus_{q=0}^{n-2}\Khred(\tau(\overzero))[x,y][0,q]\right)[-\frac{n-2}{2},\frac{n-2}{2}]\]
and consider the group
\[\textstyle\Khred(\tau(m+n)) \cong H_*\left( \Khred(\tau(m+n-1))[-\half,\half] \to \Khred(\tau(0))[-\frac{c}{2},\frac{3c+2}{2}]\right)\]
where $c=n_-(\tau(\overzero))+n-1-n_-(\tau(m+n-1))=n_-(\tau(\overzero))-n_-(\tau(m))+n-1=k_\tau+n-1$. Then
\begin{align*}
&\textstyle\Khred(\tau(m+n)) \\
& \cong\textstyle H_*\left(\Khred(\tau(m+n-1))[-\half,\half] \to \Khred(\tau(0))[-\frac{k_\tau+n-1}{2},\frac{3(k_\tau+n-1)+2}{2}]\right) \\
& \cong\textstyle H_*\left( \Khred(\tau(m+n-1))[-\half,\half] \to \Khred(\tau(0))[-\frac{k_\tau}{2},\frac{3k_\tau+2}{2}][0,n-1][-\frac{n-1}{2},\frac{n-1}{2}]\right) \\
& \cong\textstyle H_*\left(H_*\left(\Khred(\tau(m))[-\half,\half]\to\bigoplus_{q=0}^{n-2}\Khred(\tau(\overzero))[x,y][0,q]\right)[-\frac{n-2}{2},\frac{n-2}{2}][-\half,\half] \right.\\
&\qquad\qquad\qquad\qquad\textstyle\left.\to\Khred(\tau(\overzero))[x,y][0,n-1][-\frac{n-1}{2},\frac{n-1}{2}]\right) \\
&\textstyle\cong H_*\left(\Khred(\tau(m))[-\half,\half]\to\bigoplus_{q=0}^{n-1}\Khred(\tau(\overzero))[x,y][0,q]\right)[-\frac{n-1}{2},\frac{n-1}{2}]
\end{align*}
noting once again that each of the occurrences of $\Khred(\tau(\overzero))$ differs only in the secondary grading. 

Now as a relatively graded group, we are free to ignore the overall grading shift $[-\frac{n-1}{2},\frac{n-1}{2}]$. Moreover, since $\Khred(\tau(\overzero))\cong\bF$, fixing an identification \[\textstyle\bigoplus_{q=0}^{n-1}\Khred(\tau(\textstyle\overzero))[x+\half,y-\half][0,q]\cong\bF[q]/q^n\cong\bF[\bZ/n\bZ]\] we have that \[\textstyle\Khred(\tau(m+n))\cong H_*\left(\Khred(\tau(m))\to\bF[\bZ/n\bZ]\right)\] as a relatively $\bZ\oplus\bZ$-graded group.
\end{proof}

\begin{remark}\label{rmk:integer-surgery} As stated, this lemma might be viewed in terms of Heegaard Floer homology. In particular, the long exact sequence for integer surgeries may be stated \[\cdots\longto\HFhat(S^3_m(K))\longto\HFhat(S^3_{m+n}(K))\longto\textstyle\bigoplus_n\HFhat(S^3)\longto\cdots\] where \[\textstyle\bigoplus_n\HFhat(S^3)\cong\bF[\bZ/n\bZ]\] when viewed with twisted coefficients (c.f. \cite[Theorem 3.1]{OSz2008}). We have given an analogous statement in terms of the Khovanov homology of the associated branch sets in the case when $K$ is strongly invertible, a fact that is particularly interesting in light of \cite[Theorem 1.1]{OSz2005-branch} relating Khovanov homology and Heegaard Floer homology for two-fold branched covers by a spectral sequence.\end{remark}

Before turning to consequences of Lemma \ref{lem:general-stability}, we note that a similar statement is forced to exist for negative surgeries. Indeed, consider $\Khred(\tau(m-n))$ for any integer $m$, and positive integer $n$. Setting $m'=m-n$ we have that $\Khred(\tau(m'))\cong\Khred(\tau(m-n))$ and \[\Khred(\tau(m))\cong \Khred(\tau(m'+n))\cong H_*\left(\Khred(\tau(m-n))\to\bF[\bZ/n\bZ] \right).\]
It follows that:
\begin{lemma}\label{lem:general-stability-negative} For any integer $m$, and positive integer $n$, \[\textstyle\Khred(\tau(m-n))\cong H_*\left(\bigoplus_n\Khred(\tau(\overzero))\to\Khred(\tau(m))\right)\] as a relatively $\bZ\oplus\bZ$-graded group, where the integer $m$ may be interpreted as a change of framing. More precisely, there exist an explicit constants $x'$ and $y'$ (different than above) and an identification \[\textstyle\bigoplus_{q=0}^{n-1}\Khred(\tau(\textstyle\overzero))[x',y'][0,q]\cong\bF[\bZ/n\bZ]\] so that \[\textstyle\Khred(\tau(m-n))\cong H_*\left(\bF[\bZ/n\bZ]\to\Khred(\tau(m))\right)\]\end{lemma} 
\begin{remark}
In fact, it should be immediately clear that in this case the group \[\textstyle\bigoplus_{q=0}^{n-1}\Khred(\tau(\textstyle\overzero))[x',y'][0,q]\cong\bF[q^{-1}]/q^{-n}\cong\bF[\bZ/n\bZ]\] must lie in grading $\delta-1$ relative to the group \[\textstyle\bigoplus_{q=0}^{n-1}\Khred(\tau(\textstyle\overzero))[x+\half,y-\half][0,q]\cong\bF[q]/q^n\cong\bF[\bZ/n\bZ]\] of Lemma \ref{lem:general-stability} in grading $\delta$. Alternatively, Lemma \ref{lem:general-stability-negative} may be proved directly by an argument nearly identical to the argument of Lemma \ref{lem:general-stability}, up to renaming constants. 
\end{remark}

\subsection{Width stability.}

There are two essential consequences that we derive from Lemma \ref{lem:general-stability}. Similar properties exist for branch sets associated with negative surgeries, and these may be easily inferred by the reader. We will not state these, opting instead to pass to positive surgeries on the mirror to avoid negative coefficients. 

\begin{lemma}\label{lem:split} For $N\gg0$ the exact sequence for $\Khred(\tau(N+1))$ splits so that, ignoring gradings, \[\Khred(\tau(N+1))\cong\Khred(\tau(N))\oplus\bF.\]\end{lemma}
\begin{proof}
Let $m=N$ and $n=1$ in the notation of Lemma \ref{lem:general-stability}, so that \[\Khred(\tau(N+1))\cong H_*\left(\Khred(\tau(N))\to\Khred(\tau(\textstyle\overzero))\right)\cong H_*\left(\Khred(\tau(N))\to\bF\right).\]
On the other hand, with $m=0$ and $n=N+1$ we have that \[\Khred(\tau(N+1))\cong H_*\left(\Khred(\tau(0))\to\bF[q]/q^{N+1}\right).\] Comparing these two expressions, the generator represented by $q^N$ in the second corresponds to $\bF$ in the first.   Since the differential preserves the secondary $q$-grading, for $N\gg0$ the generator represented by $q^N$ cannot be in the image of the differential, hence \[\Khred(\tau(N+1))\cong\Khred(\tau(N))\oplus\bF\] as claimed. 
\end{proof}

\begin{lemma}\label{lem:diagonal} Up to overall shift the generators $\Khred(\tau(\overzero))\cong \bF$, when they survive in homology, are all supported in a single relative $\delta$-grading. \end{lemma}
\begin{proof} Immediate from the identification with the truncated polynomial ring in Lemma \ref{lem:general-stability}.\end{proof} 	

As a result of Lemma \ref{lem:split}, the width of the $\tau(n)$ may be calculated for all $n$ once some finite collection of the values is known. Moreover, these quantities must be bounded, in light of Lemma \ref{lem:diagonal}.

\begin{definition}\label{def:max-min}For a given strongly invertible knot and preferred associated quotient tangle, define $\wmax=\max_{n\in\bZ}\left\{w(\tau(n))\right\}$ and $\wmin=\min_{n\in\bZ}\left\{w(\tau(n))\right\}$.\end{definition}

\subsection{A characterization of the trivial knot}

\begin{lemma}\label{lem:trivial-knot-width} Let $T=(B^3,\tau)$ be the preferred associated quotient tangle for some strongly invertible knot in $S^3$. If $w(\tau(N))=1$ for $|N|\gg0$ then $T=(B^3,\tau)$ is the tangle associated with the trivial knot.  \end{lemma}

\begin{proof} Suppose that $\wmin=1$  with $w(\tau(N))=1$ for all $|N|$ sufficiently large. Then by Proposition \ref{prp:branch-over-thin}, $S^3_{\pm N}(K)=\Br(S^3,\tau(\pm N))$ must be an L-space for all $N$ sufficiently large. However, if $S^3_N(K)$ is an L-space, for $N$ large enough in absolute value, then $K$ is the trivial knot. 

To see this, note that since $S^3_N(K)$ is an L-space for $N\gg0$ we have that $g(K)=\boldtau(K)$, where $\boldtau(K)$ is the Ozsv\'ath-Szab\'o concordance invariant of $K$, by \cite[Proposition 3.3]{OSz2005-lens}. On the other hand, $S^3_{-N}(K)\cong-S^3_{N}(K^\star)$ is an L-space as well, so that $g(K^\star)=\boldtau(K^\star)$. However, it is a standard property of $\boldtau$ that $\boldtau(K^\star)=-\boldtau(K)$ \cite[Lemma 3.3]{OSz2003-four-ball}. Therefore, since $g(K)=g(K^\star)$ we have shown that $\boldtau(K)=g(K)=-\boldtau(K)$ hence $g(K)=0$ and $K$ must be the trivial knot (thus $\tau(0)\simeq\unknot\sqcup\unknot$ with width 2).
\end{proof}

We remark that this leads to a characterization of the trivial knot, among strongly invertible knots, by way of Khovanov homology. 

\begin{theorem}\label{thm:trivial-knot} Let $K$ be a strongly invertible knot in $S^3$ with associated quotient tangle $(B^3,\tau)$. Then $w(\tau(N))=1$ for all $|N|$ sufficiently large if and only if  $K$ is the trivial knot.  \end{theorem}

\begin{proof} Non-zero surgery on the trivial knot results in a lens space, and lens spaces arise as two-fold branched covers of a non-split two-bridge links  \cite{HR1985}. These branch sets are alternating, hence thin, by a result due to Lee \cite{Lee2005}. We return to this point, and expand on it, in Section \ref{sec:obstructions}.

The converse follows immediately from Lemma \ref{lem:trivial-knot-width}.\end{proof}

Note that it is a finite check to decide $w(\tau(n))$ for all $n$, so this is a reasonable detection of the trivial knot in the presence of a strong inversion. Of course, this result (as an application of Lemma \ref{lem:trivial-knot-width}) depends on a characterization from Heegaard Floer homology (compare \cite{HW2010}). Taken together with Remark \ref{rmk:integer-surgery}, this observation gives a further strengthening of the relationship between the two theories in the context of Dehn surgery and two-fold branched covers.

\subsection{Changes in width for integer surgeries}

To begin, recall that there are two equivalent definitions of homological width, under the additional assumption that $b_\delta>0$ for every $\delta$ in the expression $\Khred(L)\cong\bigoplus_{\delta=1}^k\bF^{b_\delta}$ (see Remark \ref{rmk:width-and-positive-betti-numbers}). In this setting, we may define $w(L)=k$ as the number of $\delta$-gradings supporting non-trivial groups in reduced Khovanov homology.   

\begin{remark}\label{rmk:blank-diagonals}Allowing for blank diagonals (i.e. for $b_i=0$ where $0<i<w+1$) gives rise to the possibility for the number of diagonals supporting non-trivial groups to be strictly lower than the homological width. This phenomenon has been conjectured not to occur, and indeed, in every example we consider (see Section \ref{sec:examples}) the two definitions for width coincide. Since it will simplify some (though not all) arguments to assume that this is the case for $\Khred(\tau(n))$, we will do so in the sequel and opt to include this assumption in our notion of generic tangle introduced in Definition \ref{def:generic}. In practice, we will see that this assumption is not restrictive. \end{remark}

\begin{lemma}\label{lem:max-min}  Let $T=(B^3,\tau)$ be the preferred associated quotient tangle for some strongly invertible knot is $S^3$. Under the assumption that there are no blank diagonals in $\Khred(\tau(n))$ for all $n\in\bZ$, the maximum and minimum widths differ by at most 1: either $\wmax=\wmin$ or $\wmax=\wmin+1$. Moreover, whenever $\wmax=\wmin+1$ for the tangle associated with a non-trivial knot there is a unique integer $\ell$ where the width changes so that $w(\tau(n))$ is constant for integers $n\le\ell$ and for integers $n\ge\ell+1$.\end{lemma}
\begin{proof}
First notice that the statement holds for the tangle associated with the quotient of the trivial knot by Lemma \ref{lem:trivial-knot-width}, since $w(\tau(0))=w(\unknot\sqcup\unknot)=2$ and $w(\tau(n))=1$ for all $n\ne 0$.

Now suppose the tangle is associated with a non-trivial knot. By stability of width, we may choose $m\ll 0$ so that $\Khred(\tau(m))\cong\bF^{b_1}\oplus\cdots\oplus\bF^{b_k}$ and $w(\tau(m'))=k$ is constant for all $m'\le m$. Now in the notation of Lemma \ref{lem:general-stability} we have that 
\[\Khred(\tau(m+n))\cong H_*\left(\Khred(\tau(m))\to\bF[\bZ/n\bZ]\right)\] 
for every $n>0$. Notice that since we assume that there are no blank diagonals in $\Khred(\tau(n))$, for any $n$, it must be that $\bF^n\cong\bF[\bZ/n\bZ]$ is supported on one of the $k$ $\delta$-gradings of $\Khred(\tau(m))$, or on a $\delta$-grading immediately adjacent to those of $\Khred(\tau(m))$. 

The case where \[\Khred(\tau(m+n))\cong H_*\left(
\raisebox{15pt}{\xymatrix@C=15pt@R=12pt{
	&\bF^{b_1} & \bF^{b_2} & \cdots & \bF^{b_{k}} \\
	 \bF^{n} &  &&& 
}}\right)\] may be immediately ruled out, however. Recall that $w(\tau(m'))=k$ for all $m'\le m$, so for $m'=m-1$ we would have that $\Khred(\tau(m))\cong\Khred(\tau(m-1+1)\cong \bF\oplus\bF^{b_1'}\oplus\cdots\oplus\bF^{b_k'}$, contradicting $w(\tau(m))=k$.

The case where  \[\Khred(\tau(m+n))\cong H_*\left(
\raisebox{15pt}{\xymatrix@C=15pt@R=12pt{
	\bF^{b_1}& \bF^{b_2} & \cdots & \bF^{b_{k}} \\
	 \bF^{n}   &&& 
}}\right)\] shows that $w(\tau(n))$ is constant for all $n$, since the connecting homomorphism is necessarily zero. Similarly, if $\bF^n\cong\bF[\bZ/n\bZ]$ is supported on $\delta$-gradings $3$ through $k$ in $\Khred(\tau(m))\cong\bF^{b_1}\oplus\cdots\oplus\bF^{b_k}$, then  $w(\tau(n))$ is constant for all $n$. In each of these cases, though the connecting homomorphism need not be zero, the width cannot change for grading reasons (in particular, since the connecting homomorphism raises $\delta$-grading by 1). 

There are two remaining cases. First consider the case where  \[\Khred(\tau(m+n))\cong H_*\left(
\raisebox{15pt}{\xymatrix@C=15pt@R=12pt{
	\bF^{b_1}\ar[dr] & \bF^{b_2} & \cdots & \bF^{b_{k}} \\
	& \bF^{n} &  & 
}}\right)\] It is possible for the width to decrease in this case. We claim that the $w(\tau(m+n))$ is no less than $k-1$. To see this, notice that a decrease in width by 2 would require that  $b_1=n$ and \[\Khred(\tau(m+n))\cong H_*\left(
\raisebox{15pt}{\xymatrix@C=15pt@R=12pt{
	\bF^{n}\ar[dr] & \{0\}& \cdots & \bF^{b_{k}} \\
	& \bF^{n} &  & 
}}\right)\] However, this shows that $\Khred(\tau(m))$ contains a blank diagonal ($b_2=0$), a contradiction.  

The final case is similar. If  \[\Khred(\tau(m+n))\cong H_*\left(
\raisebox{15pt}{\xymatrix@C=15pt@R=12pt{
	\bF^{b_1}& \bF^{b_2} & \cdots & \bF^{b_{k}}\ar[dr] & \\
	&  &  & &\bF^{n}
}}\right)\] the the possibility for width increase arises. Notice however that the possibility that the width drops by one before it increases is once again ruled out by the no blank diagonals hypothesis. Indeed if $b_k=n$ and \[\Khred(\tau(m+n))\cong H_*\left(
\raisebox{15pt}{\xymatrix@C=15pt@R=12pt{
	\bF^{b_1}& \bF^{b_2} & \cdots & \bF^{n}\ar[dr] & \\
	&  &  & &\bF^{n}
}}\right)\] then the case $w(\tau(m+n))=k-1$ implies that $\Khred(\tau(m+n+1))$ contains a blank diagonal. In particular, the width can only increase from $k$ to $k+1$ in this setting. 

We conclude from this case study that either $\wmax=\wmin$ or $\wmax=\wmin+1$, as claimed. Whenever $\wmax=\wmin+1$, the width either {\em expands} or {\em decays}. It follows from the argument above that when this is the case for a tangle associated with a non-trivial knot in $S^3$, there is a unique $\ell$ for which $w(\tau(\ell))$ and $w(\tau(\ell+1))$ differ. 

 More precisely, in the notation of Lemma \ref{lem:general-stability}, the width expands whenever 
\[\Khred(\tau(\ell+1))\cong H_*\left(
\raisebox{15pt}{\xymatrix@C=15pt@R=12pt{
	\bF^{b_1} & \bF^{b_2} & \cdots & \bF^{b_{\wmin}}\ar[dr]^{0} & \\
	& &  & & \bF
}}\right)\] with $b_{\wmin}=1$, and the  width decays whenever \[
\Khred(\tau(\ell+1))\cong H_*\left(
\raisebox{15pt}{\xymatrix@C=15pt@R=12pt{
	\bF\ar[dr]_{\cong} & \bF^{b_2} & \cdots & \bF^{b_{\wmax}} \\
	& \bF &  & 
}}\right)\] 
for some $\ell\in\bZ$. Notice that, since \[\Khred(\tau(\ell+n))\cong H_*\left(\Khred(\tau(\ell))\to\bF[\bZ/n\bZ]\right),\] once a new diagonal appears it is forced to persist, and when an old diagonal vanishes it cannot reappear, due to the identification of the $q$-gradings in $\bF[\bZ/n\bZ]$. As a result, the value $\ell\in\bZ$ in each setting  is unique. \end{proof}

For example, Berge knots (chosen so that the lens space surgeries are positive) give rise to a family of tangles for which the width decays (c.f. Proposition \ref{prp:width-bound-Berge}). 

\subsection{On determinants and resolutions}\label{sec:det-and-res}

In the arguments that follow, we will rely heavily on resolutions of terminal crossings (see Definition \ref{def:terminal-crossing}) in branch sets $\tau(\pq)$ for which $S^3_{p/q}(K)=\Br(B^3,\tau(\pq))$. As such, we remark that $\textstyle\det(\tau(\pq))=|H_1(S^3_{p/q}(K);\bZ)|=p$ for any $\pq>0$ (in all cases, we deal with negative surgeries by passing to the mirror image). Moreover if $\tau(\pqzero)$ and $\tau(\pqone)$ are the links obtained by resolving the terminal crossing, then $\textstyle\det(\tau(\pq))=\det(\tau(\pqzero))+\det(\tau(\pqone))$ by applying Property \ref{pro:end}.

As a result, $\Khred(\tau(\pq))$ may be studied by applying Proposition \ref{prp:mo} to the resolutions  $\tau(\pqzero)$ and $\tau(\pqone)$ whenever $\pq>1$. In the case $\pq\in(0,1)$ the same arguments work by using Proposition \ref{prp:mo-perturbed} when treating continued fractions of length $r=2$: here $\det(\tau(\pqzero))=\det(\tau(0))=0$.

By Lemma \ref{lem:general-stability} we have that \[\Khred(\tau(n))\cong H_*\left(\Khred(\tau(0))\to\bF[\bZ/n\bZ]\right)\] for a specific identification of $\bigoplus_n\Khred(\tau(\overzero))\cong \bF[\bZ/n\bZ]$ as a graded group. As a result, \[\textstyle\Khsig(\tau(n))\cong H_*\left(\Khsig(\tau(n-1))\to\Khsig(\tau(\overzero))\right)\]
whenever $n>1$, and 
\[\textstyle\Khsig(\tau(1))\cong H_*\left(\Khred(\tau(0))[-\half]\to\Khsig(\tau(\overzero))\right)\]
where $\Khsig(\tau(\overzero))=\Khred(\tau(\overzero))$ since the signature of the trivial knot is 0. In either case, \[\textstyle\Supp\left(\Khsig(\tau(n))\right)\subseteq\Supp\left(\Khsig(\tau(n+1))\right)\] or \[\textstyle\Supp\left(\Khsig(\tau(n+1))\right)\subseteq\Supp\left(\Khsig(\tau(n))\right)\] as absolutely $\bZ$-graded groups (where the fixed shifts are adjusted accordingly by $[-\half]$ in the case $n=0$). Notice that these inclusions are equalities whenever $w(\tau(n))=w(\tau(n+1))$, so that the inclusions are only relevant in the case when the width changes by one.

\subsection{An upper bound for width}\label{sub:upper}

\begin{proposition}\label{prp:width-upper-bound}
Let $K$ be a strongly invertible knot in $S^3$, with preferred associated quotient tangle $T=(B^3,\tau)$. Then $w(\tau(\pq))$ is bounded above by $\wmax$ for all $\pq\in\bQ$.
\end{proposition}

\begin{proof}
By taking mirrors, we suppose without loss of generality that $\pq\ge0$ and proceed in 2 cases. 

{\bf Case 1:} $1\le\pq$

By its definition, $\wmax$ provides the upper bound for $w(\tau(n))$ for any $n$. This provides a base for induction in $r$, the length of the continued fraction representation $\pq=[a_1,a_2,\ldots,a_r]$. 

First consider the case $\pq=[a_1,2]$. Here we have $\textstyle\det(\pq)=p=p_0+p_1=a_1+a_1+1=\det(\pqzero)+\det(\pqone)$ where $\pqzero=[a_1]$ and $\pqone=[a_1,1]=[a_1+1]$ by resolving the terminal crossing. In either case $w(\pqzero),w(\pqone)\le\wmax$, and by applying Proposition \ref{prp:mo} we have \[\textstyle\Khsig(\tau(\pq))\cong H_*\left(\Khsig(\tau(\pqzero))\to\Khsig(\tau(\pqone))\right).\] Moreover, according to Section \ref{sec:det-and-res} we have that either  \[\textstyle\Supp\left(\Khsig(\tau(\pqone))\right)\subseteq\Supp\left(\Khsig(\tau(\pqzero))\right)\] or \[\textstyle\Supp\left(\Khsig(\tau(\pqzero))\right)\subseteq\Supp\left(\Khsig(\tau(\pqone))\right)\] (depending on expansion or decay) as a consequence of Lemma \ref{lem:general-stability}. Therefore, 
\[w(\tau[a_1,2]) = \textstyle w(\tau(\frac{2a_1+1}{2})) \textstyle\le \max\{w(\tau\lfloor\frac{2a_1+1}{2}\rfloor),w(\tau\lceil\frac{2a_1+1}{2}\rceil)\} \le\wmax.\]
The same statement holds for $\pq=[a_1,a_2]$. By iterating Proposition \ref{prp:mo} we have \[\xymatrix@C=15pt@R=12pt{
	 {\textstyle\Khsig(\tau(\pq))}\ar[r]  &{\textstyle\Khsig(\tau(\pqzero))}\\
	 {\textstyle\Khsig(\tau(\frac{p_1'}{q_1'}))}\ar[u]\ar[r] & {\textstyle\Khsig(\tau(\pqzero))}\\
	 {\vdots}\ar[u] &\\
	{\textstyle\Khsig(\tau(\pqone))}\ar[u] &
}\]
where the connecting homomorphisms have been omitted (of course, there is no danger in doing so since these can only decrease the homological width). Once again, as a consequence of supports we conclude that $w(\tau[a_1,a_2])\le\max\{w(\tau(a_1)),w(\tau(a_1+1))\}\le\wmax$.

Now for induction on $r$: given $\pq=[a_1,a_2,\ldots,a_{r-1}]$ the inductive hypothesis is that $w(\tau(\pq))\le\wmax$ and one of \[\textstyle\Supp\left(\Khsig(\tau[a_1,a_2,\ldots,a_{r-1}])\right)\subseteq\Supp\left(\Khsig(\tau[a_1,a_2,\ldots,a_{r-1}+1])\right)\] or \[\textstyle\Supp\left(\Khsig(\tau[a_1,a_2,\ldots,a_{r-1}+1])\right)\subseteq\Supp\left(\Khsig(\tau[a_1,a_2,\ldots,a_{r-1}])\right)\] holds.

This being the case, we claim that 
\[w(\tau[a_1,a_2,\ldots,a_{r-1},a_{r}])\le\max\left\{w(\tau[a_1,a_2,\ldots,a_{r-1}]),w(\tau[a_1,a_2,\ldots,a_{r-1}+1])\right\}.\]

By resolving the terminal crossing of $\tau(\pq)$ and applying Proposition \ref{prp:mo} 
\[\textstyle\Khsig(\tau(\pq))\cong H_*\left(
\textstyle\Khsig(\tau(\pqzero))\to\textstyle\Khsig(\tau(\pqone))\right) \]
so that  $w(\tau(\pq))\le\max\{w(\tau(\pqzero)),w(\tau(\pqone))\}$ if $a_r=2$. By induction in $a_r$ we have that \[w(\tau[a_1,a_2,\ldots,a_{r-1},a_{r}])\le\max\{w(\tau[a_1,a_2,\ldots,a_{r+1}]),w(\tau[a_1,a_2,\ldots,a_{r-1}+1])\}\] by applying Property \ref{pro:end} together with the induction hypothesis on supports. 

Notice that, by construction of the mapping cone, we have that $\Supp(\Khred(\tau(\pq)))$ either contains or is contained by each of  $\Supp(\Khred(\tau(\pqzero)))$ and $\Supp(\Khred(\tau(\pqone)))$. Recall that which of the latter is relevant to the inductive step depends on the parity of $r$. 

As a result, by induction in length we have that $w(\tau(\textstyle\pq))\le\{w(\tau\lfloor\pq\rfloor),w(\tau\lceil\pq\rceil)\}\le\wmax$, concluding the proof in this case.

{\bf Case 2:} $0<\pq<1$
 
The proof in this case follows the same lines as the previous case, and differs only in passing from the case $r=2$ to $r=1$. Indeed, the argument here is identical, once we replace the use of Proposition \ref{prp:mo} with that of its degenerative counterpart, Proposition \ref{prp:mo-perturbed}. This is due to the fact that, while the determinants remain additive under resolution, $\det(\tau\lfloor\pq\rfloor)=0$ in this case.  
 
To see this, consider once again the case $\pq=[a_1,2]=[0,2]$. By applying Proposition \ref{prp:mo-perturbed} we have \[\textstyle\Khsig(\tau(\pq))\cong H_*\left(\Khsig(\tau(\pqzero))[-\half]\to\Khsig(\tau(\pqone))\right).\] Moreover, according to Section \ref{sec:det-and-res} we have that either  \[\textstyle\Supp\left(\Khsig(\tau(\pqone))[-\half]\right)\subseteq\Supp\left(\Khsig(\tau(\pqzero))\right)\] or \[\textstyle\Supp\left(\Khsig(\tau(\pqzero))\right)\subseteq\Supp\left(\Khsig(\tau(\pqone))[-\half]\right)\] as a consequence of Lemma \ref{lem:general-stability}. Therefore, 
\[w(\tau[a_1,2]) = \textstyle w(\tau(\frac{2a_1+1}{2})) \textstyle\le \max\{w(\tau\lfloor\frac{2a_1+1}{2}\rfloor),w(\tau\lceil\frac{2a_0+1}{2}\rceil)\} \le\wmax.\]
The same statement holds for $\pq=[a_1,a_2]$. By iterating Proposition \ref{prp:mo-perturbed} as in the previous case so that $w(\tau[a_1,a_2]) \le \max\{w(\tau(a_1)),w(\tau(a_1+1))\}\le\wmax$. 
\end{proof}

\begin{remark}\label{rmk:WLOG} Case 2, when $\pq\in(0,1)$, will be present in many of the arguments that follow. However, in every setting this case simply amounts to replacing Proposition \ref{prp:mo} with Proposition \ref{prp:mo-perturbed} in passing from half-integer (continued fractions of length 2) to integer surgeries, as in the above proof. Thus we will restrict, without loss of generality, to the case $\pq\ge1$ in the arguments below.\end{remark}

With the upper bound of Proposition \ref{prp:width-upper-bound} in hand, we conclude this section with the proof of Proposition \ref{prp:width-bound-Berge}. 

\begin{proof}[Proof of Proposition \ref{prp:width-bound-Berge}] For any Berge knot $K$, there is some $N$, positive up to taking mirrors, for which $S^3_N(K)$ is a lens space. By a result of Osborne \cite{Osborne1981}, K is a strongly invertible knot, so let $T=(B^3,\tau)$ be a representative for the associated quotient tangle compatible with the basis for surgery $(\mu,N\mu+\lambda)$, where $\lambda$ is the preferred longitude. Note that $S^3_N(K)\cong\Br(S^3,\tau(0))$ in this setting.

Now we have already seen that $w(\tau(n))=1$ for all $n\ge0$ as a result of Proposition \ref{prp:Berge-quasi-alternating}, since quasi-alternating knots are thin by a result of Manolescu and Ozsv\'ath \cite{MO2007}. On the other hand, $w(\tau(n))$ is at most 2 when $n<0$ in application of Lemma \ref{lem:max-min}. Therefore, for any Berge knot, the associated quotient tangle has $\wmin=1$ and $\wmax=2$.

The result now follows from an application of Proposition \ref{prp:width-upper-bound}: $w(\tau(\pq))$ is bounded above by $\wmax=2$. \end{proof}  


\section{Lower bounds and generic tangles}\label{sec:estimates}

\subsection{A lower bound for width}

\begin{proposition}\label{prp:width-lower-bound}
Let $K$ be a strongly invertible knot in $S^3$, with preferred associated quotient tangle $T=(B^3,\tau)$. If $\wmax=\wmin$ then $w(\tau(\pq))$ is bounded below by $\wmin$ for all $\pq\in\bQ$.
\end{proposition}

\begin{proof} Without loss of generality, assume that $\pq\ge1$.
Since $\wmax=\wmin=w$, we have that $w=w(\tau(n))$ for every $n\in\bZ$. In particular, 
\[\textstyle\Supp\left(\Khsig(\tau(n+1))\right)=\Supp\left(\Khsig(\tau(n))\right)\] as a consequence of Lemma \ref{lem:general-stability}. 
Thus, applying Proposition \ref{prp:mo} \[\textstyle\Khsig(\tau[a_1,2])\cong H_*\left(\Khsig(\tau(a_1))\to\Khsig(\tau(a_1+1))\right)\] so that
if $\Khred(\tau(a_1))\cong\bF^{b_1}\oplus\cdots\oplus\bF^{b_w}$ and $\Khred(\tau(a_1+1))\cong\bF^{b_1'}\oplus\cdots\oplus\bF^{b_w'}$ (note that $b_i\ne b_i'$ for precisely one value $1\le i\le w$) then  
\[\Khred(\tau[a_1,2])\cong H_*\left(
\raisebox{15pt}{\xymatrix@C=15pt@R=12pt{
	\bF^{b_1}\ar[dr] & \bF^{b_2}\ar[dr] & {\cdots}\ar[dr] & \bF^{b_w}\\
	\bF^{b_1'}& \bF^{b_2'} & \cdots & \bF^{b_w'} 
}}\right)\] as a relatively graded group, since the differential of the mapping cone raises $\delta$-grading by 1. Notice in particular that $b_1^*\ge b_1'$ and $b_w^*\ge b_w$  for $\Khred(\tau[a_1,2])=\bF^{b_1^*}\oplus\cdots\oplus\bF^{b_w^*}$, so that $w(\tau[a_1,2])=w$.

Similarly, for $\pq=[a_1,a_2]$ in general, we may iteratively apply Proposition \ref{prp:mo} $a_2-1$ times to the same end: 
\[\Khred([a_1,a_2])\cong H_*\left(\raisebox{-35pt}{$H_*\Bigg($}
\raisebox{45pt}{\xymatrix@C=15pt@R=12pt{
	\bF^{b_1} & \bF^{b_2} & {\cdots} & \bF^{b_w}\\
	\vdots & \vdots &  &\vdots \\
	\bF^{b_1}\ar[dr] & \bF^{b_2}\ar[dr] & {\cdots}\ar[dr] & \bF^{b_w}\\
	\bF^{b_1'}& \bF^{b_2'} & \cdots & \bF^{b_w'} 
}}\raisebox{-35pt}{$\Bigg)$}\right)\]
so that $b_1^*\ge b_1'$ and $b_w^*\ge b_w$ for $\Khred(\tau[a_1,a_2])=\bF^{b_1^*}\oplus\cdots\oplus\bF^{b_w^*}$, and once again $w(\tau[a_1,a_2])=w$.

To complete the proof then, we induct in $r$ with the assumption that $w(\tau(\pq))=w$ for all $\pq=[a_1,\ldots,a_{r-1}]$, and that   
\[\textstyle\Supp\left(\Khsig(\tau[a_1,a_2,\ldots,a_{r-1}])\right)=\Supp\left(\Khsig(\tau[a_1,a_2,\ldots,a_{r-1}+1]))\right)\] holds.
This being the case, we claim that 
\[w(\tau[a_1,a_2,\ldots,a_{r-1},a_{r}])\ge\min\left\{w(\tau[a_1,a_2,\ldots,a_{r-1}]),w(\tau[a_1,a_2,\ldots,a_{r-1}+1])\right\}.\]
Indeed, when $a_r=2$ we have that 
\[\textstyle\Khsig(\tau(\pq))\cong H_*\left(
\textstyle\Khsig(\tau(\pqzero))\to\textstyle\Khsig(\tau(\pqone))\right) \]
by applying Proposition \ref{prp:mo} so that in either case $w(\tau(\pq))= w(\tau(\pqzero)),w(\tau(\pqone))$ since the corresponding groups have the same support. By induction in $a_r$ we have that \[w(\tau[a_1,a_2,\ldots,a_{r-1},a_{r}])=w(\tau[a_1,a_2,\ldots,a_{r+1}]),w(\tau[a_1,a_2,\ldots,a_{r-1}+1])\] as before, by applying the induction hypothesis on supports. Notice that the new supports coincide by construction of the mapping cone. 

As a result, by induction in the length of the continued fraction expansion of $\pq$, we have that $w(\tau(\textstyle\pq))=w$ concluding the proof.
\end{proof}

Combining Proposition \ref{prp:width-lower-bound} with Proposition \ref{prp:width-upper-bound} we have immediately that $w(\tau(-)):\bQ\to\bN$ takes a single value $w\in\bN$ when $w=\wmax=\wmin$, where $T=(B^3,\tau)$ is the preferred representative for the quotient tangle associated with a strongly invertible knot in $S^3$.

Before turning attention to the case of changes in width, we remark that implicit in the above argument is the following:

\begin{proposition}\label{prp:lower-bound-on-interval}Let $K$ be a strongly invertible knot in $S^3$, with preferred associated quotient tangle $T=(B^3,\tau)$. If $w=w(\tau(n))=w(\tau(n+1))$ then $w(\tau(\pq))=w$ for all $\pq\in[n,n+1]$. \end{proposition}

It follows that, given any interval $[a,b]$ for integers $a,b$ on which the width $w=w(\tau(n))$ is constant for $n\in[a,b]\cap\bZ$, then $w=w(\tau(\pq))$ for all $\pq\in[a,b]$. 

\subsection{Expansion and genericity} The interesting case now corresponds to the setting wherein $w(\tau(n))$ and $w(\tau(n+1))$ differ. Under the assumption that there are no blank diagonals in the group $\Khred(\tau(n))$ for all $n$ (see Remark \ref{rmk:blank-diagonals}), we have that there is a unique integer $\ell$ where this possibility arrises (as shown in Lemma \ref{lem:max-min}). Since we deal with negative surgeries by passing to the mirror, and  the case $\ell=0$ by Remark \ref{rmk:WLOG}, we will treat the case $\ell>0$ in the arguments that follow, without loss of generality. 

First consider the case of width expansion. Then $w(\tau(\ell))+1=w(\tau(\ell+1))$ and 
\[\Khred(\tau(\ell)) \cong \bF^{b_1}\oplus\cdots\oplus\bF^{b_k} \] while
\[\Khred(\tau(\ell+1)) \cong \bF^{b_1}\oplus\cdots\oplus\bF^{b_k} \oplus \bF\] by Lemma \ref{lem:max-min}. Further, $\Supp(\Khsig(\tau(\ell))\subset\Supp(\Khsig(\tau(\ell+1))$ and 
\[\Khred(\tau(\ell+n)) \cong H_\ast \left(\raisebox{15pt}{\xymatrix@C=15pt@R=12pt{
	\bF^{b_1} & \bF^{b_2} & \cdots & \bF^{b_k}\ar[dr] & \\
	& &  & & \bF^n 
}}\right)\]
by Lemma \ref{lem:general-stability}. Consulting the proof of Proposition \ref{prp:width-upper-bound}, we see that resolving the terminal crossing of $\tau(\pq)$ to obtain a mapping cone in terms of the homologies of $\tau(\pqzero)$ and $\tau(\pqone)$, the support of the {\em shorter} of the two groups is always contained in the support of the {\em longer}. Moreover, the supports coincide on all but the highest $\delta$-grading. 

A potential for considerable decrease in width arises then whenever a mapping cone $H_\ast(A\to B)$ is considered wherein $A$ is supported on all but the highest $\delta$-grading of $B$. For example, if 
\[H_\ast \left(\raisebox{15pt}{\xymatrix@C=15pt@R=12pt{
	\bF^{b_1}\ar[dr]  & \bF^{b_2} \ar[dr] & \cdots\ar[dr]  & \bF^{b_k}\ar[dr] & \\
	\bF^{b_1'} & \bF^{b_2'} & \cdots & \bF^{b_k'} & \bF^{b_{k+1}'} 
}}\right)\]
the resulting width could be less than $k$. For this reason, we take into consideration the additional structure from the quantum grading in this setting. 

\begin{definition}\label{def:short-long}A relatively $\bZ\oplus\bZ$-graded group $A$ has expansion short form if there is a primary grading $\delta_{\max}$ and a secondary grading $q_{\max}$ with the properties that \begin{itemize}\item[(1)] $A_q=0$ for all $q>q_{\max}$ and \item[(2)]  $A^{\delta}_{q_{\max}}\ne 0$ if and only if $\delta= \delta_{\max}$. 
\end{itemize}
A relatively $\bZ\oplus\bZ$-graded group $B$ has expansion long form if, in addition to the conditions above, there is a grading $\delta_{\max}+1$ supporting non-trivial groups in secondary gradings $q<q_{\max}$ (but $B^{\delta_{\max}+1}_{q_{\max}}=0$).
\end{definition}

A schematic representation of this definition is shown in Figure \ref{fig:short-and-long}.

\begin{figure}[ht!]\begin{center}
\labellist\small
\endlabellist
\includegraphics[scale=0.5]{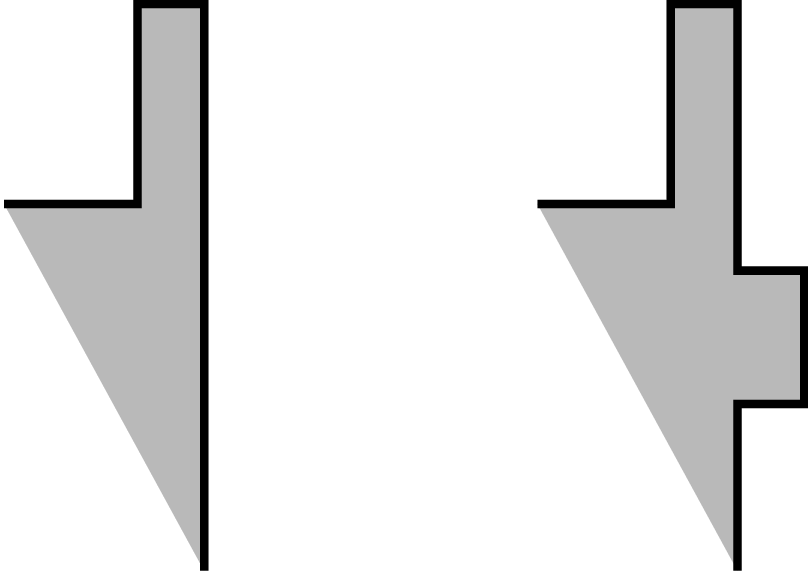}\caption{A schematic description of the bigraded support in the top gradings for expansion in short (left) and long (right) form.}\label{fig:short-and-long}\end{center}\end{figure}

\begin{proposition}\label{prp:short-to-long}Given a mapping cone of the form $H_\ast(A\to B)$ in which \begin{itemize} \item[(1)]$A$ is of width $k$ and  has expansion short form, 
\item[(2)]  $B$ is of width $k+1$ and has expansion long form and
\item[(3)]  $A$ is supported on the first $k$ $\delta$-gradings of $B$ in the expression $H_\ast(A\to B)$,
\end{itemize}
 the homology of the mapping cone has width at least $k$. Moreover, if the width is $k$ then the result has expansion short form and otherwise the width is $k+1$ and the result has expansion long form. 
\end{proposition}

\begin{proof}First, by ignoring the secondary grading we have that by hypothesis the mapping cone $H_\ast(A\to B)$ is of the form \[H_\ast \left(\raisebox{15pt}{\xymatrix@C=15pt@R=12pt{
	\bF^{b_1}\ar[dr]  & \bF^{b_2} \ar[dr] & \cdots\ar[dr]  & \bF^{b_k}\ar[dr] & \\
	\bF^{b_1'} & \bF^{b_2'} & \cdots & \bF^{b_k'} & \bF^{b_{k+1}'} 
}}\right)\] As a result, it is immediate that the homology will have non-trivial groups supported in the first $\delta$-grading. To prove the claim then, we focus on the top ($k^{\rm th}$ and $(k+1)^{\rm st}$) $\delta$-gradings. This will require consideration of the secondary grading.

Since the relevant mapping cone is constructed without reference to the secondary grading, we proceed in cases. Denote the secondary grading on $A$ by $q^A$ and the secondary grading on $B$ by $q^B$. Let $q^B_{\rm top}$ be the largest grading in $B^{\delta_{\max}+1}$ supporting a non-trivial group. Note that  $q^B_{\rm top}<q^B_{\max}$. Similarly, let $q^B_{\rm bottom}$ be the smallest grading in $B^{\delta_{\max}+1}$ supporting a non-trivial group. There are 3 cases to consider; these are schematically illustrated in Figure \ref{fig:mapping-cone-cases}. 

Case 1: $q^A_{\max} < q^B_{\rm bottom}$. Notice that the connecting homomorphism $\bF^{b_k}\to\bF^{b'_{k+1}}$ is necessarily zero for grading reasons. As a result, the homology has expansion long form (hence width $k+1$), since groups in gradings $q^B_{\max}$ survive in homology.  

Case 2: $q^A_{\max} < q^B_{\max}$. In this case it may be that there are non-trivial connecting homomorphisms $\bF^{b_k}\to\bF^{b'_{k+1}}$, so groups supported in grading $\delta_{\max}+1$ will not necessarily survive in homology. In the case that all groups in this top grading collapse, notice that the group $B^{\delta_{\max}}_{q_{\max}}$ always survives in homology. As a result, the width may be no smaller than $k$ and the resulting homology has expansion short form. On the other hand, if a single group survives in grading $\delta_{\max}+1$, the width remains $k+1$ and the surviving group in  $B^{\delta_{\max}}_{q_{\max}}$ ensures that the result is of expansion long form.

Case 3: $q^B_{\max} < q^A_{\max}$. In this case, the group $A^{\delta_{\max}}_{q_{\max}}$ is the top-most group surviving in homology. This ensures that the resulting homology has expansion short form (in the case that the groups in grading $\delta_{\max}+1$ do not survive in homology) or expansion long form (in the case that some group in grading $\delta_{\max}+1$ survives). Note that for $q^B_{\max} \ll q^A_{\max}$ the latter is the expected behavior since the connecting homomorphism $\bF^{b_k}\to\bF^{b'_{k+1}}$ will be zero for grading reasons. In particular, this is the case whenever the $A^{\delta_{\max}}$ is supported in secondary gradings larger than $q^B_{\rm top}$. 
 \end{proof}

\begin{figure}[ht!]\begin{center}
\labellist\small
\endlabellist
\includegraphics[scale=0.5]{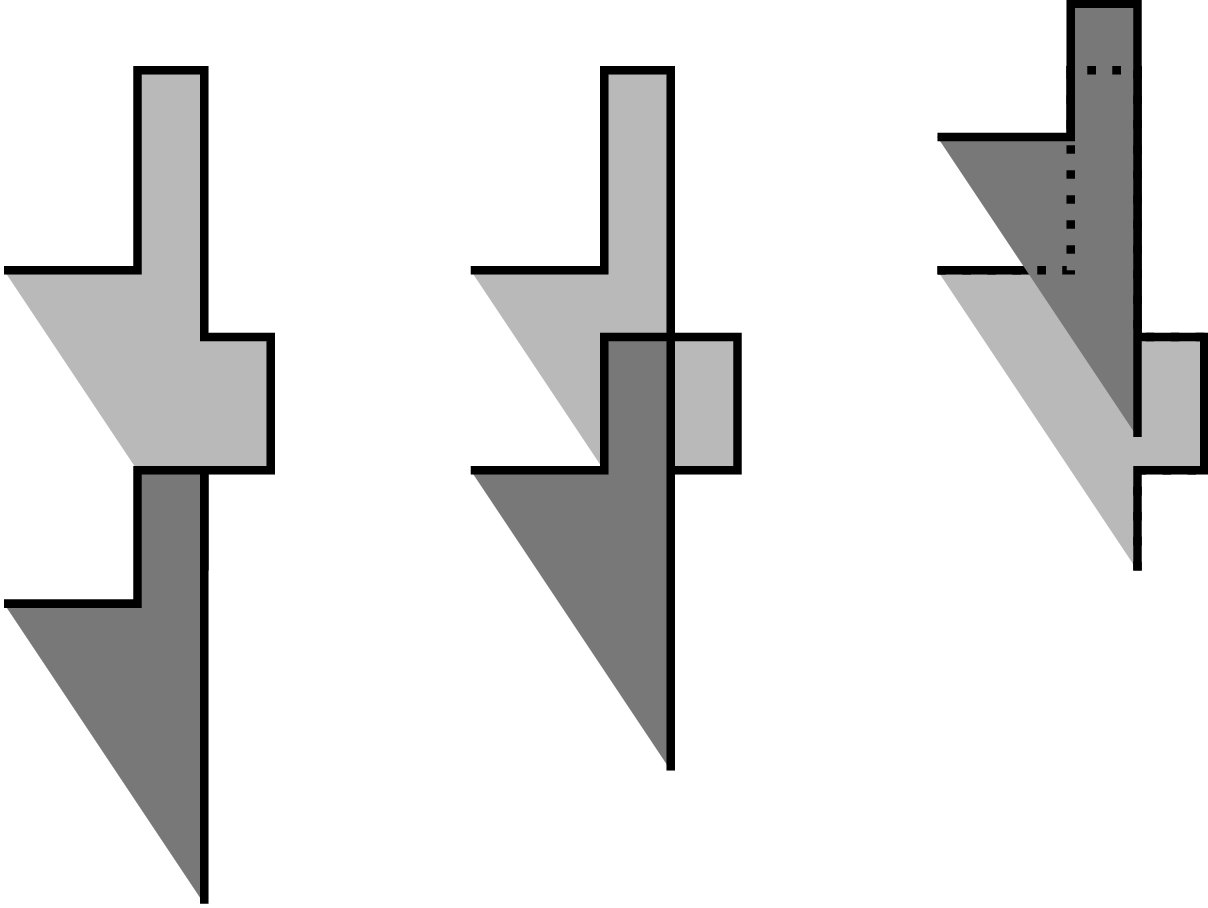}\caption{Various cases that arise in the proof of Proposition \ref{prp:short-to-long}. The top bi-gradings of the groups $A$ (darker) and $B$ (lighter) are shown schematically as combined in the expression $H_*(A\to B)$.}\label{fig:mapping-cone-cases}\end{center}\end{figure}

Similarly, we have:

\begin{proposition}\label{prp:reversed-possibility}
With hypothesis (1), (2) and (3) of Proposition \ref{prp:short-to-long}, the mapping cone $H_\ast(B\to A)$ has homology of width $k+1$ and is in expansion long form. 
\end{proposition}

\begin{proof}This time ignoring the secondary grading the hypothesis for the mapping cone $H_\ast(A\to B)$ gives \[H_\ast \left(\raisebox{15pt}{\xymatrix@C=15pt@R=12pt{
	\bF^{b_1}\ar[dr]  & \bF^{b_2} \ar[dr] & \cdots\ar[dr]  & \bF^{b_k}&  \bF^{b_{k+1}} \\
	\bF^{b_1'} & \bF^{b_2'} & \cdots & \bF^{b_k'} & 
}}\right)\] 
As a result it is immediate that the homology has width $k+1$. To see that the homology has expansion long form notice first that either $q^A_{\max} \le q^B_{\max}$ or $q^B_{\max} < q^A_{\max}$. In the first case we have that $q^B_{\max}$ assumes the role of $q_{\max}$ in homology; in the second case $q^A_{\max}$ assumes the role of $q_{\max}$.

Regardless of which case occurs, the groups supported in the $(k+1)^{\rm st}$ $\delta$-grading satisfy the additional hypothesis of Definition \ref{def:short-long} for expansion long form.
 \end{proof}

\begin{proposition} \label{prp:remaining-possibilities}
A mapping cone $H_\ast(A\to A')$ where the supports of $A$ and $A'$ coincide and each is of expansion short form has homology of expansion short form. Similarly, a mapping cone $H_\ast(B\to B')$ where the supports of $B$ and $B'$ coincide and each is of expansion long form has homology of expansion long form.
\end{proposition}

\begin{proof}The argument follows along similar lines to that of Proposition \ref{prp:short-to-long} and Proposition \ref{prp:reversed-possibility}. 

Consider the case $H_\ast(A\to A')$. Since by hypothesis this mapping cone takes the form \[H_\ast \left(\raisebox{15pt}{\xymatrix@C=15pt@R=12pt{
	\bF^{b_1}\ar[dr]  & \bF^{b_2} \ar[dr] & \cdots\ar[dr]  & \bF^{b_k} \\
	\bF^{b_1'} & \bF^{b_2'} & \cdots & \bF^{b_k'} 
}}\right)\] the resulting homology must have width $k$. The maximal $q$-grading in the top $\delta$-grading is the larger of these gradings from $A$ and $A'$ as they appear in the mapping cone. 

The case $H_*(B\to B')$ is similar. \end{proof}

\begin{definition}\label{def:expansion}Let $K$ be a strongly invertible knot with associated quotient tangle $(B^3,\tau)$. This tangle is expansion generic if there are no blank diagonals in $\Khred(\tau(n))$ (see Remark \ref{rmk:blank-diagonals}) and if  $w(\tau(\ell))+1=w(\tau(\ell+1))$ then $\Khred(\tau(\ell))$ has expansion short form and $\Khred(\tau(\ell+1))$ has expansion long form (as in Definition \ref{def:short-long}). \end{definition}

\begin{proposition}\label{prp:expand}Given a strongly invertible knot with expansion generic associated quotient tangle, $\wmin$ provides a lower bound for $w(\tau(\pq))$.\end{proposition}

\begin{proof}Let $\pq\ge0$ (negative surgeries are dealt with by passing to the mirror). By Lemma \ref{lem:max-min} there is a unique integer $\ell$ such that $w(\tau(\ell))+1 = w(\tau(\ell+1))$. If $\ell$ is negative there is nothing to show: the width is constant for tangles associated with positive integer surgeries and the result follows from Proposition \ref{prp:lower-bound-on-interval}. Suppose then that $\ell$ is non-negative, and note also that the case $\pq\notin[\ell,\ell+1]$ follows from Proposition \ref{prp:lower-bound-on-interval}.

We need only consider the case $\pq\in[\ell,\ell+1]$, and we may assume without loss of generality that $\ell>0$ by Remark \ref{rmk:WLOG}, as is now familiar. 

First notice that $\Khsig(\tau(\ell+1))\cong H_*(\Khsig(\tau(\ell))\to\bF)$ where $\Khsig(\tau(\ell))$ is supported on $k$ $\delta$-gradings by assumption, and term $\bF$ is supported on a $(k+1)^{\rm st}$ diagonal so that  $\tau(\ell+1)$ has width $k+1$. In particular, we observe that  $\Khsig(\tau(\ell))$ is supported on the first $k$ diagonals of the group $\Khsig(\tau(\ell+1))$. 

This provides a base case for induction on $r$, the length of the continued fraction expansion $\pq=[a_1,a_2,\ldots,a_r]$ with $a_1=\ell$ (so that $\lfloor\pq\rfloor = \ell$). Suppose for induction that \begin{itemize}
\item[(1)] $\Khsig(\tau(\pq))$ is of either expansion short form (with width $k$) or expansion long form (with width $k+1$) for any $\pq$ of length at most $r$, and
\item[(2)] any pair of branch sets $\tau[a_1,\ldots,a_s]$ and  $\tau[a_1,\ldots,a_s+1]$ have support of their respective $\sigma$-normalized Khovanov homologies coinciding on the first $k$ $\delta$-gradings (for any positive integer $a_s$ and any $1\le s\le r$). 
\end{itemize}

Now consider the branch set $\tau(\pq)=\tau[a_1,a_2,\ldots,a_r,a_{r+1}]$ in the case where where $a_1=\ell$ and $a_{r+1}=2$. By resolving the terminal crossing, we obtain $\tau(\pqzero)$ and $\tau(\pqone)$, where $\pqzero$ and $\pqone$ are the continued fractions $[a_1,a_2,\ldots,a_r]$ and  $[a_1,a_2,\ldots,a_r+1]$ (recall that which is $\pqzero$ depends on the parity of $r$). As a result, \[\textstyle\Khsig(\tau(\pq))\cong H_*\left(\Khsig(\tau(\pqzero))\to\Khsig(\tau(\pqone))\right)\] so that by our induction hypothesis  $\Khsig(\tau(\pq))$ is computed via a mapping cone involving groups that are each either expansion short form or expansion long form. Further, by the hypothesis (2) on supports, the groups involved in the mapping cone have support that coincides on the first $k$ $\delta$-gradings. 

We claim that the homology $\Khsig(\tau(\pq))$ is either supported on $k$ diagonals and has expansion short form, or it is supported on $k+1$ diagonals and has expansion long form. To see this, we proceed in cases and consider the mapping cone calculating this group. 

If $\Khsig(\tau(\pqzero))$ has expansion short form and $\Khsig(\tau(\pqone))$ has expansion long form then applying Proposition \ref{prp:short-to-long} the result follows. In particular, the result has either expansion short form (and width $k$) or expansion long form (and width $k+1$).

Similarly, if $\Khsig(\tau(\pqzero))$ has expansion long form and $\Khsig(\tau(\pqone))$ has expansion short form then applying Proposition \ref{prp:reversed-possibility} the result has expansion long form. The case where both groups have the same width is handled by Proposition \ref{prp:remaining-possibilities}. In both cases the result has the same form (expansion short or expansion long) as the groups involved in the mapping cone. 

Finally, we observe that by construction of the mapping cone  \[\textstyle\Khsig(\tau(\pq))\cong H_*\left(\Khsig(\tau(\pqzero))\to\Khsig(\tau(\pqone))\right),\] the branch sets $\tau[a_1,a_2,\ldots,a_r,1]=\tau[a_1,a_2,\ldots,a_r+1]$ and $\tau[a_1,a_2,\ldots,a_r,2]$ have $\sigma$-normalized Khovanov homologies coinciding on the first $k$ $\delta$-gradings.

To complete the inductive step in $r$, we appeal to a second induction in $a_{r+1}$. That is, suppose that for $1\le a_{r+1} \le n$ we have verified the induction in $r$ as in the case $a_{r+1}=2$ described above. That is, we suppose that (1) and (2) hold (additionally) for $\tau(\pq)$ whenever $\pq=[\ell,a_2,\ldots, a_r,a_{r+1}]$ where $1\le a_{r+1} \le n$. Now consider the branch set $\tau(\pq)$ with $\pq=[\ell,a_2,\ldots,a_r,n+1]$. Resolving the terminal crossing we have that \[\textstyle\Khsig(\tau(\pq))\cong H_*\left(\Khsig(\tau(\pqzero))\to\Khsig(\tau(\pqone))\right)\] where $\pqzero$ and $\pqone$ are described by continued fractions $[\ell,a_2,\ldots,a_r]$ and $[\ell,a_2,\ldots,a_r,n]$, with identification depending on the parity of $r$ as usual. Since each of $\Khsig(\tau(\pqzero))$ and $\Khsig(\tau(\pqone))$ has either expansion short form or expansion long form, and since the support of these groups coincides on the first $k$ $\delta$-gradings, we can apply Proposition \ref{prp:short-to-long}, Proposition \ref{prp:reversed-possibility} or Proposition \ref{prp:remaining-possibilities} as required by the various possibilities to conclude that $\Khsig(\tau(\pq))$ has expansion short form (and width $k$) or expansion long form (and width $k+1$). 

To conclude the proof then, we need only observe that by construction of the mapping cone  \[\textstyle\Khsig(\tau(\pq))\cong H_*\left(\Khsig(\tau(\pqzero))\to\Khsig(\tau(\pqone))\right),\]the branch sets $\tau[a_1,a_2,\ldots,a_r,n]$ and $\tau[a_1,a_2,\ldots,a_r,n+1]$ have $\sigma$-normalized Khovanov homologies coinciding on the first $k$ $\delta$-gradings.
\end{proof}

\subsection{Decay and genericity} As in the previous section, we consider changes in width this time in the decay setting. That is, for some integer $\ell$ we have $w(\tau(\ell))-1=w(\tau(\ell+1))$. As before, we will consider the case $\ell>0$. What follows is nearly identical to the last section, so we will give definitions and statements of the relevant propositions only, and leave the proofs to the reader. 

\begin{definition}\label{def:long-short}A relatively $\bZ\oplus\bZ$-graded group $B$ has decay short form if there is a secondary grading $q_{\min}$ with the property that \begin{itemize}\item[(1)] $B_q=0$ for all $q<q_{\min}$ and \item[(2)]  $B^{\delta}_{q_{\min}}\ne0$ if and only if $\delta= \delta_{\min}$. 
\end{itemize}
A relatively $\bZ\oplus\bZ$-graded group $A$ has decay long form if, in addition to the conditions above, there is a grading $\delta_{\min}-1$ supporting non-trivial groups in secondary gradings $q>q_{\min}$ (but $A^{\delta_{\min}-1}_{q_{\min}}=0)$.
\end{definition}

In analogy with the properties established in the previous section, we have:

\begin{proposition}\label{prp:long-to-short}Given an mapping cone of the form $H_\ast(A\to B)$ in which \begin{itemize} \item[(1)]$A$ is of width $k+1$ and  has decay long form, 
\item[(2)]  $B$ is of width $k$ and has decay short form and
\item[(3)]  $B$ is supported on the last $k$ $\delta$-gradings of $A$ in the expression $H_\ast(A\to B)$,
\end{itemize}
 the homology of the mapping cone has width at least $k$. Moreover, if the width is $k$ then the result has decay short form and otherwise the width is $k+1$ and the result has decay long form. 
\end{proposition}

\begin{proposition}\label{prp:reversed-possibility-b}
With hypothesis (1), (2) and (3) of Proposition \ref{prp:long-to-short}, the mapping cone $H_\ast(B\to A)$ has homology of width $k+1$ and is in decay long form. 
\end{proposition}

\begin{proposition} \label{prp:remaining-possibilities-b}
A mapping cone $H_\ast(A\to A')$ where the supports of $A$ and $A'$ coincide and each is of decay long form has homology of decay long form. Similarly, a mapping cone $H_\ast(B\to B')$ where the supports of $B$ and $B'$ coincide and each is of decay short form has homology of decay short form.
\end{proposition}

This leads to a similar notion of genericity in the decay setting.

\begin{definition}\label{def:decay}Let $K$ be a strongly invertible knot with associated quotient tangle $(B_3,\tau)$. This tangle is decay generic if there are no blank diagonals in $\Khred(\tau(n))$ (see Remark \ref{rmk:blank-diagonals}) and if  $w(\tau(\ell))-1=w(\tau(\ell+1))$ then $\Khred(\tau(\ell))$ has decay long form and $\Khred(\tau(\ell+1))$ has decay short form (as in Definition \ref{def:long-short}). \end{definition}

\begin{proposition}\label{prp:decay}Given a strongly invertible knot with decay generic associated quotient tangle, $\wmin$ provides a lower bound for $w(\tau(\pq))$.\end{proposition}

\subsection{Generic tangles} We now summarize the  results of this section by fixing the notion of generic tangle, as follows:

\begin{definition}\label{def:generic}Let $T=(B^3,\tau)$ be the preferred representative for the tangle associated with a strongly invertible knot in $S^3$. $T$ is generic if either \begin{itemize}\item[(1)] $\wmin=\wmax$, \item[(2)] $T$ is expansion generic (as in Definition \ref{def:expansion}) or \item[(3)] $T$ is decay generic (as in Definition \ref{def:decay}). \end{itemize} Recall that the case of expansion or decay assumes that for generic $T$, $\Khred(\tau(n))$ contains no blank diagonals for any $n\in\bZ$. \end{definition}

\begin{theorem}\label{thm:generic} For generic associated quotient tangles, $\wmin$ and $\wmax$ provide lower and upper bounds respectively for $w(\tau(\pq))$ for any reduced $\pq\in\bQ$. \end{theorem}

\begin{proof} The upper bound as claimed is established in Section \ref{sec:width}; the lower bound follows from a summary of the results of this section.

First notice that if $\wmin=\wmax$ we are done, by applying Proposition \ref{prp:width-lower-bound}. The cases of expansion and decay follow from Proposition \ref{prp:expand} and Proposition \ref{prp:decay}, respectively. 
\end{proof}


\section{Surgery obstructions}\label{sec:obstructions}

\subsection{Width bounds for lens spaces}

Combining work of Hodgson and Rubinstein \cite{HR1985} with work of Lee \cite{Lee2005}, we have the following statement:

\begin{theorem}\label{thm:lens-bound}  If $Y$ is a lens space, then $Y$ is a two-fold branched cover of $S^3$ with branch set of width 1.\end{theorem}

Note that this excludes the manifold $S^2\times S^1$ since it is branched over the 2-component trivial link having width 2. 

\begin{proof}[Proof of Theorem \ref{thm:lens-bound}]
By work of Hodgson and Rubinstein, only non-split two-bridge links arise as the branch sets of lens spaces \cite{HR1985}. As a result, to generate this collection of branch sets we need to consider surgery on the trivial knot in $S^3$; the associated quotient tangle is rational, and the preferred representative  is $(B^3,\zero)$ since $\det(\tau(0))=\det(\unknot\sqcup\unknot)=0$ (equivalently, $S^2\times S^1=\Br(S^3,\unknot\sqcup\unknot)$).

We have the material in place to show that this class of branch sets has thin Khovanov homology, a result that is due to Lee by virtue of the fact that non-split two-bridge links are alternating \cite{Lee2005}. Since both $\tau(\overzero)$ and $\tau(1)$ are the trivial knot, applying Lemma \ref{lem:general-stability} we have that $\bF\cong H_*(\Khred(\tau(0))\to\bF).$ Recall that $\Khred(\tau(0))\cong\bF\oplus\bF$ as a relatively $\bZ$-graded group. Now it follows that the branch sets corresponding to positive integer surgery have Khovanov homology \[
\Khred(\tau(n))\cong H_*(\Khred(\tau(0))\to\bF[\bZ/n\bZ])\cong \bF^n,\] hence $w(\tau(n))$ is thin for all $n\ne0$ (compare Theorem \ref{thm:trivial-knot}).

Without loss of generality, we consider $\Khred(\tau(\pq))$ for $\pq>0$. In fact, $\tau[0,a_2,a_3,\ldots,a_r]\simeq\tau[a_3,\ldots,a_r]$, so we need only consider $\pq\ge1$. Now it is a quick application of Proposition \ref{prp:width-upper-bound} to see that $\tau(\pq)$ is a thin link, for all $\pq\ne0$, since $\tau(n)$ is thin for all $n\ne 0$.  
\end{proof}

In constructing two-bridge links in this way, we  recover Schubert's normal form for this class \cite{Schubert1956}. Note that this proof that two-bridge knots are thin may be viewed as an adaptation of the proof that quasi-alternating knots are thin, due to Manolescu and Ozsv\'ath \cite{MO2007}, applied to two-bridge knots -- a class of knots whose members are alternating, hence quasi-alternating.  

\subsection{Width bounds for finite fillings}
Our main goal of this section is to prove an analogous statement in the case of manifolds with finite fundamental group. This will require the following result, implicit in work of Boileau and Otal \cite{BO1991}. 

\begin{theorem}\label{thm:Boileau-and-Otal} If $\Br(S^3,L)$ has finite fundamental group then $\Br(S^3,L)\cong \Br(S^3,L')$ implies $L\simeq L'$. In particular, manifolds with finite fundamental group arise uniquely as a two-fold branched covers of $S^3$. \end{theorem}

\begin{proof}[Sketch of proof] Since geometrization has been established by Thurston when the manifold in question admits an involution with non-trivial fixed point set \cite{Thurston1982},  we have that $|\pi_1(\Br(S^3,L))|<\infty$ is equivalent to $\Br(S^3,L)$ admitting elliptic geometry (the work of Boileau and Porti \cite{BP2001} is devoted to a detailed proof of Thurston's Orbifold Theorem from which this fact follows).  

The case of lens spaces is treated by work of Hodgson and Rubinstein \cite{HR1985} as applied in the proof of Theorem \ref{thm:lens-bound}, so we assume without loss of generality that  $Y=\Br(S^3,L)$ is not a lens space. According to Scott, the remaining manifolds with elliptic geometry are Seifert fibred with 3 singular fibres and base orbifold $S^2$, and fall into two classes: either $S^2(2,2,n)$ for $n>1$ or $S^2(2,3,n)$ for $n=3,4,5$ \cite{Scott1983}. Further, the Seifert structure on $Y$ is unique (up to isotopy), and by Thurston's Orbifold Theorem we may assume that the strong inversion on $Y$ is an isometry. In particular, the involution preserves the Seifert structure. 

In this setting, there are two cases to consider: either the involution reverses orientation on the fibres or it preserves orientation on the fibres. In the case of the former, the branch set is a Montesinos link, while the latter give rise to a two-fold branched cover of a Seifert link (that is, a union of fibres in some Seifert fibration of $S^3$) \cite{BO1991,Montesinos1976,Montesinos1987}. 

However, for a given $Y$ (with finite fundamental group), the possible branch sets must be isotopic, since the pair of involutions on $Y$  are conjugate by work of Boileau and Otal \cite{BO1991} (see also \cite{Montesinos1987}).
\end{proof}

Particular instances of this observation may be established by hand. For example, the Poincar\'e homology sphere arises as two-fold branched cover of both the $(3,5)$-torus knot (a Seifert link) and the $(-2,3,5)$-pretzel knot (a Montesinos link). We leave as an exercise the task of establishing the equivalence between this pair of knots. It seems likely that the above result was known to Montesinos \cite{Montesinos1987}, though we learned the  proof (as outlined) from M. Boileau. With this in hand, we establish the following:

\begin{theorem}\label{thm:finite-bound} If $\Br(S^3,L)$ has finite fundamental group then $w(L)$ is at most 2. \end{theorem}

\begin{proof}
Note that by Theorem \ref{thm:lens-bound} any lens space $\Br(S^3,L)$ satisfies the bound of Theorem \ref{thm:finite-bound} since $w(L)=1$ in this case. As in the proof of Theorem \ref{thm:Boileau-and-Otal}, the remaining manifolds with elliptic geometry are Seifert fibred with 3 singular fibres and base orbifold $S^2$, and fall into two classes: either $S^2(2,2,n)$ for $n>1$ or $S^2(2,3,n)$ for $n=3,4,5$ \cite{Scott1983}. The manifolds in each class may be constructed by considering fillings of Seifert fibred knot manifolds with base orbifold $D^2(2,2)$ (the twisted $I$-bundle over the Klein bottle) and $D^2(2,3)$ (the trefoil complement), respectively (see Heil \cite{Heil1974}, Montesinos \cite{Montesinos1976}).

To any Seifert fibred space $Y$ with base orbifold $S^3(p,q,r)$, Montesinos constructs a strong inversion so that $Y\cong\Br(S^3,L)$ (see Proposition \ref{prp:Montesinos}). Since $Y$ has finite fundamental group, the branch set $L$ is unique (up to isotopy). As a consequence, it suffices to construct the family of manifolds in each class in such a way that the branch set is made explicit. 

When filling the complement of the trefoil we appeal to Proposition \ref{prp:width-bound-Berge}: the branch set associated with filling any torus knot in $S^3$ has width at most 2. To complete the proof then, we are left to consider the case of filling the twisted $I$-bundle over the Klein bottle, $M$.   

When considered with Seifert structure $D^2(2,2)$, this manifold has the property that $\Delta(\fibre,\lm)=1$, where $\fibre$ is a regular fibre in the boundary. Note that $M(\lm)$ must be $S^2\times S^1$, and $M(n\fibre+\lm)$ is a lens space for all $n\ne0$ by work of Heil \cite{Heil1974}. By fixing a representative for the associated quotient tangle compatible with the basis for surgery $(\fibre,\lm)$ it follows that $w(\tau(n))=1$ for all $n\ne0$, and $w(\tau(0))=2$. Now resolving the terminal crossing in $\tau(\pq)$ we have, for $\pq\ge0$,
\begin{align*}
\det(\tau(\textstyle\pq)) 
&= |H_1(M(p\fibre+q\lm);\bZ)| \\
&= c_M\Delta(p\fibre+q\lm,\lm) \\
&= c_M|p\fibre\cdot\lm| \\
&= c_M|(p_0+p_1)\fibre\cdot\lm+(q_0+q_1)\lm\cdot\lm| \\
&= c_M|(p_0\fibre+q_0\lm)\cdot\lm|+c_M|(p_1\fibre+q_1\lm)\cdot\lm| \\
&= c_M\Delta(p_0\fibre+q_0\lm,\lm)+c_M\Delta(p_1\fibre+q_1\lm,\lm) \\
&= |H_1(M(p_0\fibre+q_0\lm);\bZ)|+|H_1(M(p_1\fibre+q_1\lm);\bZ)| \\
&= \det(\tau(\textstyle\pqzero)) + \det(\tau(\textstyle\pqone))
\end{align*}
in terms of $(\fibre,\lm)$.
This is enough to obtain the result, proceeding as in the proof of Proposition \ref{prp:width-bound-Berge} (by way of Proposition \ref{prp:width-upper-bound}), working with a tangle compatible with the basis $(\fibre,\lm)$ in place of the preferred basis $(\mu,\lambda)$.
\end{proof}

This result should be compared with \cite[Proposition 2.3]{OSz2005-lens}: Ozsv\'ath and Szab\'o show that manifolds with elliptic geometry are all L-spaces.

\subsection{Width obstructions}
We are now in a position to assemble the material developed to this point into obstructions to exceptional surgeries.

\begin{theorem}\label{thm:general} Let $M$ be a simple, strongly invertible knot manifold with associated quotient tangle $T=(B^3,\tau)$ compatible with some basis $(\alpha, \beta)$ in $\partial M$. Then $w(\tau(\pq))>1$ implies that $M(p\alpha+q\beta)$ is not a lens space, and $w(\tau(\pq))>2$ implies that $M(p\alpha+q\beta)$ has infinite fundamental group.  \end{theorem}
\begin{proof}For $w>1$ the statement follows from Theorem \ref{thm:lens-bound}; for $w>2$ the statement follows from Theorem \ref{thm:finite-bound}.\end{proof}

This is an effective obstruction when combined with the results of Section \ref{sec:width} (in particular the stability of Lemma \ref{lem:general-stability}) and Section \ref{sec:estimates}. In particular, restricting to generic associated quotient tangles in the sense of Definition \ref{def:generic} yields the following:

\begin{theorem}\label{thm:lens} Let $K\into S^3$ be strongly invertible with generic preferred associated quotient tangle. Then $\wmin>1$ implies that $K$ does not admit lens space surgeries. Moreover, determining $\wmin$ is a finite check by stability.\end{theorem}
\begin{proof} This is an application of Theorem \ref{thm:lens-bound}, together with the observation that $\wmin$ is determined on some finite collection of integers as a result of Lemma \ref{lem:general-stability} and provides a lower bound for $w(\tau(\pq))$ according to Theorem \ref{thm:generic}.\end{proof}

\begin{theorem}\label{thm:finite} Let $K\into S^3$ be strongly invertible with generic preferred associated quotient tangle. Then $\wmin>2$ implies that $K$ does not admit finite fillings. Moreover, determining $\wmin$ is a finite check by stability.\end{theorem}
\begin{proof}Similarly, this is an application of Theorem \ref{thm:finite-bound}, together with the observation that $\wmin$ is determined on some finite collection of integers as a result of Lemma \ref{lem:general-stability} and provides a lower bound for $w(\tau(\pq))$ according to Theorem \ref{thm:generic}.\end{proof}

In the absence of the genericity hypothesis, homological width is still a useful obstruction. In light of the cyclic surgery theorem \cite{CGLS1987}, it is enough to check the integer fillings of $K$ when the question of lens space surgeries is of interest. Similarly, in the case of finite fillings only the integer and half-integer surgeries need to be considered in light of work of Boyer and Zhang, whenever the complement admits a hyperbolic structure \cite{BZ1996}. In practice however, genericity is easy to check and seems to be the rule and rather than the exception. In the generic setting (see examples given below), it is particularly interesting that Khovanov homology is able to give useful surgery obstructions, without relying on these powerful theorems. 

\section{Examples and applications}\label{sec:examples}

\subsection{A first example: the figure eight}

It is well known that the figure eight knot $K=4_1$ does not admit lens space surgeries. In fact, Thurston \cite{Thurston1980} classified the non-hyperbolic fillings of $S^3\smallsetminus\nu(K)$ and showed that they all have infinite fundamental group. That $K$ does not admit lens space surgeries has been reproved using the machinery of $SU(2)$-representation spaces \cite{KK1990, Klassen1991}, essential laminations \cite{Delman1995}, character varieties \cite{Tanguay1996} and most recently, Heegaard Floer homology \cite{OSz2005-lens}. As a first example of the width obstructions developed here, we show that Khovanov homology detects that $K$ does not admit finite fillings. 

\begin{figure}[ht!]
\begin{center}
\raisebox{20pt}{\includegraphics[scale=0.45]{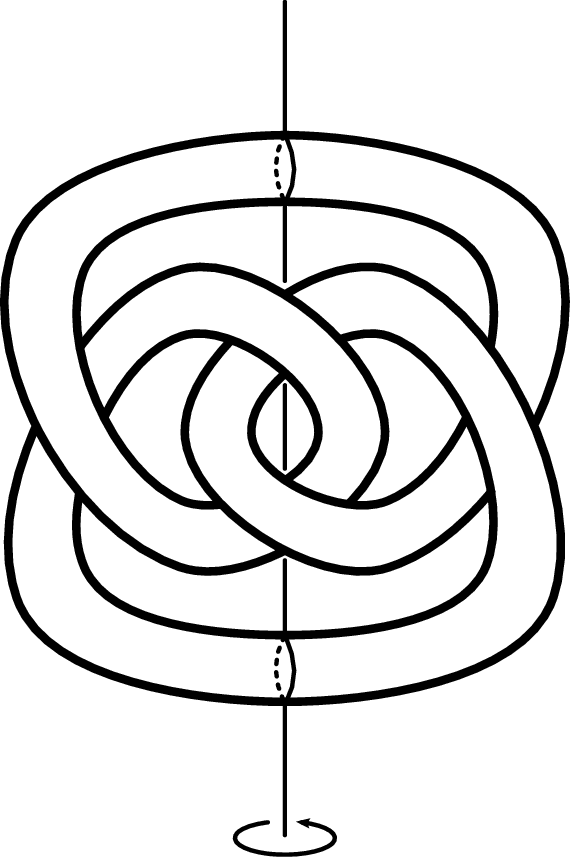}}\qquad
\raisebox{-40pt}{\includegraphics[scale=0.45]{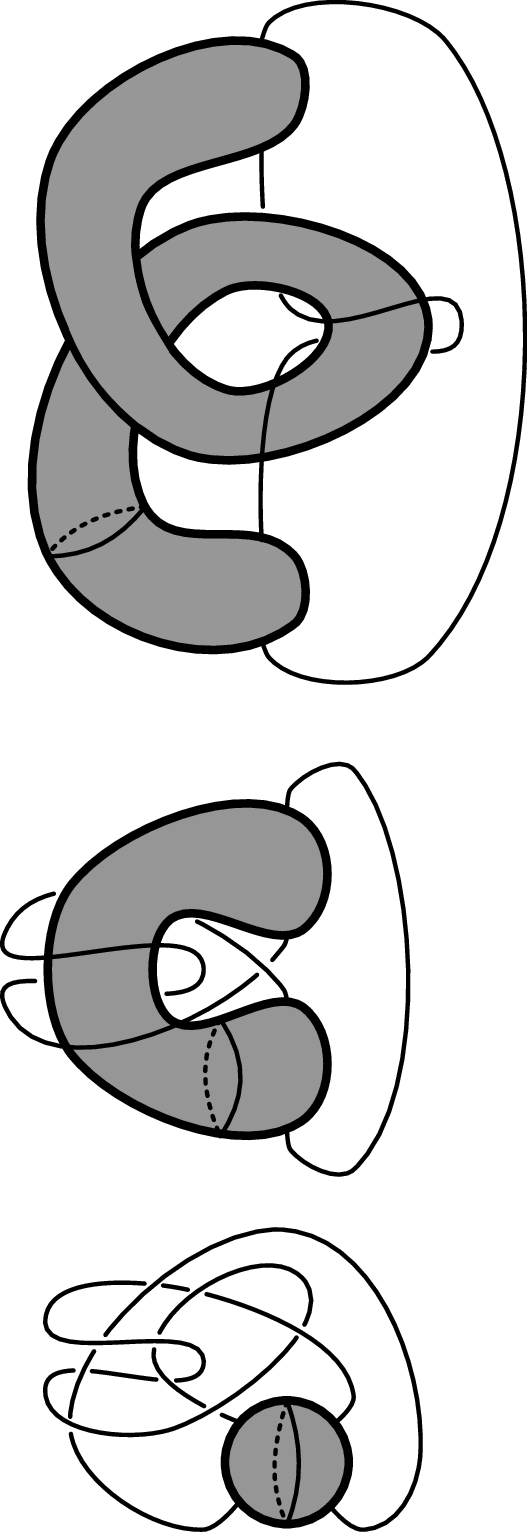}}\qquad
\labellist
	\pinlabel \rotatebox{-90}{$\cong$} at 210 425
\endlabellist
\raisebox{20pt}{\includegraphics[scale=0.45]{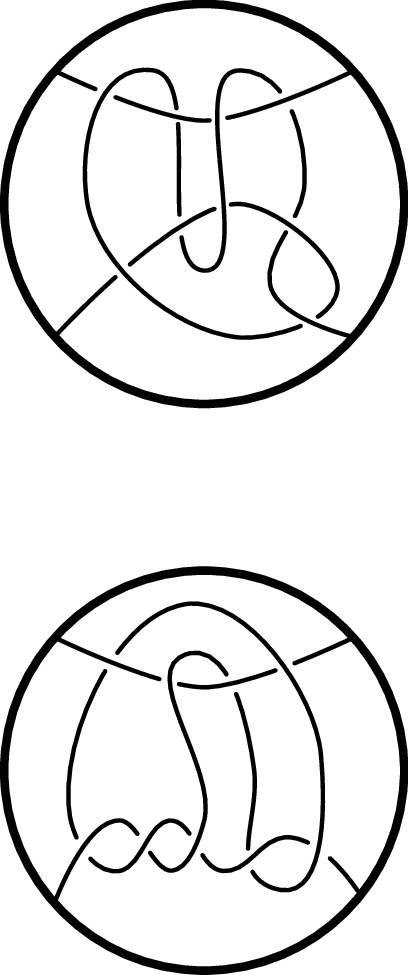}}
\end{center}
\caption{The strong inversion on the figure eight (left); isotopy of a fundamental domain (centre); and two representatives of the associated quotient tangle (right).}
\label{fig:figure8-inversion}\end{figure}

$K$ is a strongly invertible knot, and this symmetry is shown in Figure \ref{fig:figure8-inversion} together with the associated quotient tangle. We have given two equivalent views of the associated quotient tangle. The first of these shows that the branch sets for integer surgeries may be expressed as closed 3-braids. For \[\beta_n=\si_1^{-1}\si_2^{-2}\si_1^{-2}\si_2^{-2}\si_1^{-2}\si_2^{-4+n}\] we have that $\tau(n)\simeq\overline{\beta_n}$, the closure of $\beta_n$. The Khovanov homology groups $\Khred(\tau(-1))$,  $\Khred(\tau(0))$ and $\Khred(\tau(+1))$ are given in Figure \ref{fig:figure8-homology} (note in particular that $\chi(\Khred(\tau(0)))=\det(\tau(0))=0$). Notice that $\wmin=2$ and that the tangle is decay generic. It follows at once that $K$ does not admit lens space surgeries applying Theorem \ref{thm:lens}, and it seems worth pointing out that this result could have been inferred simply by inspection of the single Khovanov homology group $\Khred(\tau(0))$.

\begin{figure}[ht!]\begin{center}
\labellist\small
	\pinlabel $1$ at 161 306
	
	\pinlabel $1$ at 196 236 
	\pinlabel $1$ at 196 306
	\pinlabel $1$ at 196 342
	\pinlabel $1$ at 196 414
	
	\pinlabel $1$ at 233 342
	\pinlabel $1$ at 233 414
	\pinlabel $1$ at 233 450
	\pinlabel $1$ at 233 520
\endlabellist
\includegraphics[scale=0.35]{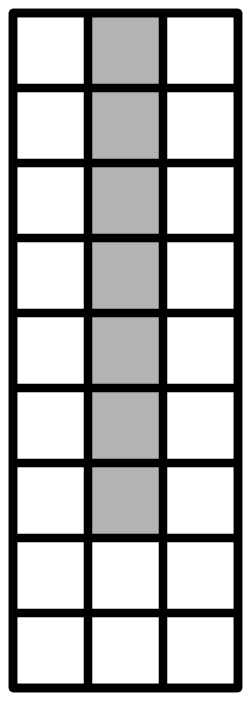} \quad
\labellist\small
	\pinlabel $1$ at 161 306
	
	\pinlabel $1$ at 196 236 
	\pinlabel $1$ at 196 269
	\pinlabel $1$ at 196 306
	\pinlabel $1$ at 196 342
	\pinlabel $1$ at 196 414
	
	\pinlabel $1$ at 233 342
	\pinlabel $1$ at 233 414
	\pinlabel $1$ at 233 450
	\pinlabel $1$ at 233 520
\endlabellist
\includegraphics[scale=0.35]{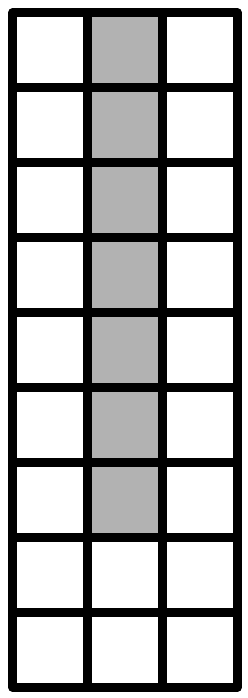} \quad
\labellist\small
	
	\pinlabel $1$ at 196 236 
	\pinlabel $1$ at 196 269
	\pinlabel $1$ at 196 306
	\pinlabel $1$ at 196 342
	\pinlabel $1$ at 196 414
	
	\pinlabel $1$ at 233 342
	\pinlabel $1$ at 233 414
	\pinlabel $1$ at 233 450
	\pinlabel $1$ at 233 520
\endlabellist
\includegraphics[scale=0.35]{figures/grid-4_1}
\qquad 
\includegraphics[scale=0.45]{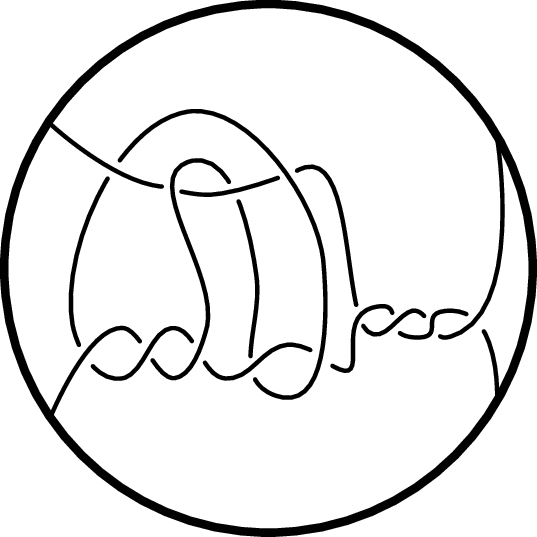}\caption{The preferred representative for the associated quotient tangle $T=(B^3,\tau)$ of the figure eight, and the reduced Khovanov homology groups $\Khred(\tau(-1))$, $\Khred(\tau(0))$ and $\Khred(\tau(1))$ (from left to right). The distinguished $\delta$-grading has been highlighted, in accordance with Lemma \ref{lem:general-stability} setting $m=0$.}\label{fig:figure8-homology}\end{center}\end{figure}

More generally, we may use Lemma \ref{lem:general-stability} to calculate:

\begin{proposition}\[\Khred(\tau(n))\cong
\begin{cases}
\phantom{\bF\oplus}\ \ \bF^{4+n}\!\oplus\bF^4 & n>0 \\
\ \ \ \! \bF\oplus\bF^5\oplus\bF^4 & n=0 \\
\bF^{|n|}\oplus\bF^4\oplus\bF^4 & n<0 
\end{cases}\]\end{proposition}
\begin{proof} The distinguished $\delta$-grading from Lemma \ref{lem:general-stability} is identified in Figure \ref{fig:figure8-homology}. By calculating that $\Khred(\tau(-2))\cong\bF^2\oplus\bF^4\oplus\bF^4$, Lemma \ref{lem:general-stability}, together with the groups \begin{align*}\Khred(\tau(-1)) &\cong \bF\oplus\bF^4\oplus\bF^4 \\ \Khred(\tau(0)) &\cong \bF\oplus\bF^5\oplus\bF^4 \\ \Khred(\tau(1)) &\cong \phantom{\bF\oplus\ }\bF^5\oplus\bF^4 \\ \end{align*} forces the result.\end{proof}

In fact, we have enough to recover Thurston's result:
\begin{theorem}Khovanov homology detects that the figure eight admits no finite fillings.\end{theorem}
\begin{proof} First notice that $w(\tau(n))=3$ for $n\le0$. As a result, a finite filling cannot arise by negative surgery on the figure eight (applying Theorem \ref{thm:finite-bound}), since $w(\tau(\pq))=3$ for any $\pq\le 0$ by Proposition \ref{prp:lower-bound-on-interval}. However, since the figure eight knot is amphicheiral, the same must be true for positive surgeries.\end{proof}

\subsection{Some pretzel knots that do not admit finite fillings}\label{sec:255} According to Mattman \cite{Mattman2000}, it is unknown if the $(-2,p,q)$-pretzel knots admit fillings with finite fundamental group for $q\ge p \ge 5$. When $p=q=5$ we have the following.

\begin{theorem}\label{thm:pretzel-finite-fillings} The $(-2,5,5)$-pretzel knot does not admit finite fillings.\end{theorem}
\begin{proof}
We begin by noting that the $(-2,5,5)$-pretzel knot, $K_5$, is strongly invertible in two ways as indicated in Figure \ref{fig:pretzel-inversion}. 
\begin{figure}[ht!]\begin{center}  
\includegraphics[scale=0.35]{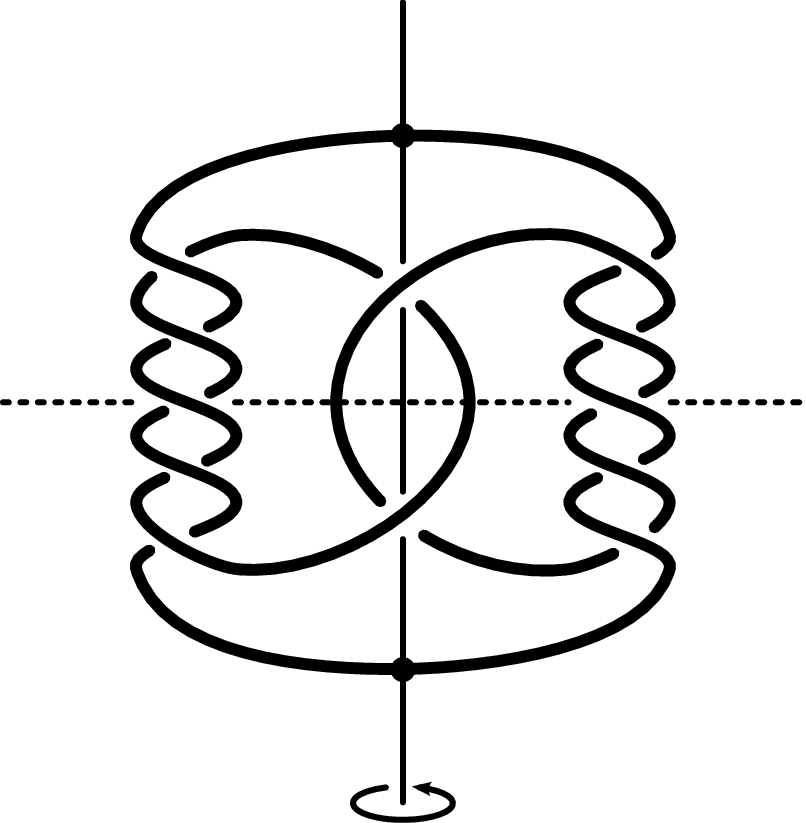}\caption{Two strong inversions on the $(-2,5,5)$-pretzel knot.}\label{fig:pretzel-inversion}\end{center}\end{figure}
We will make use of the inversion indicated by the solid vertical line; the associated quotient tangle is calculated in Figure \ref{fig:pretzel-isotopy}.
\begin{figure}[ht!]\begin{center}  
\includegraphics[scale=0.40]{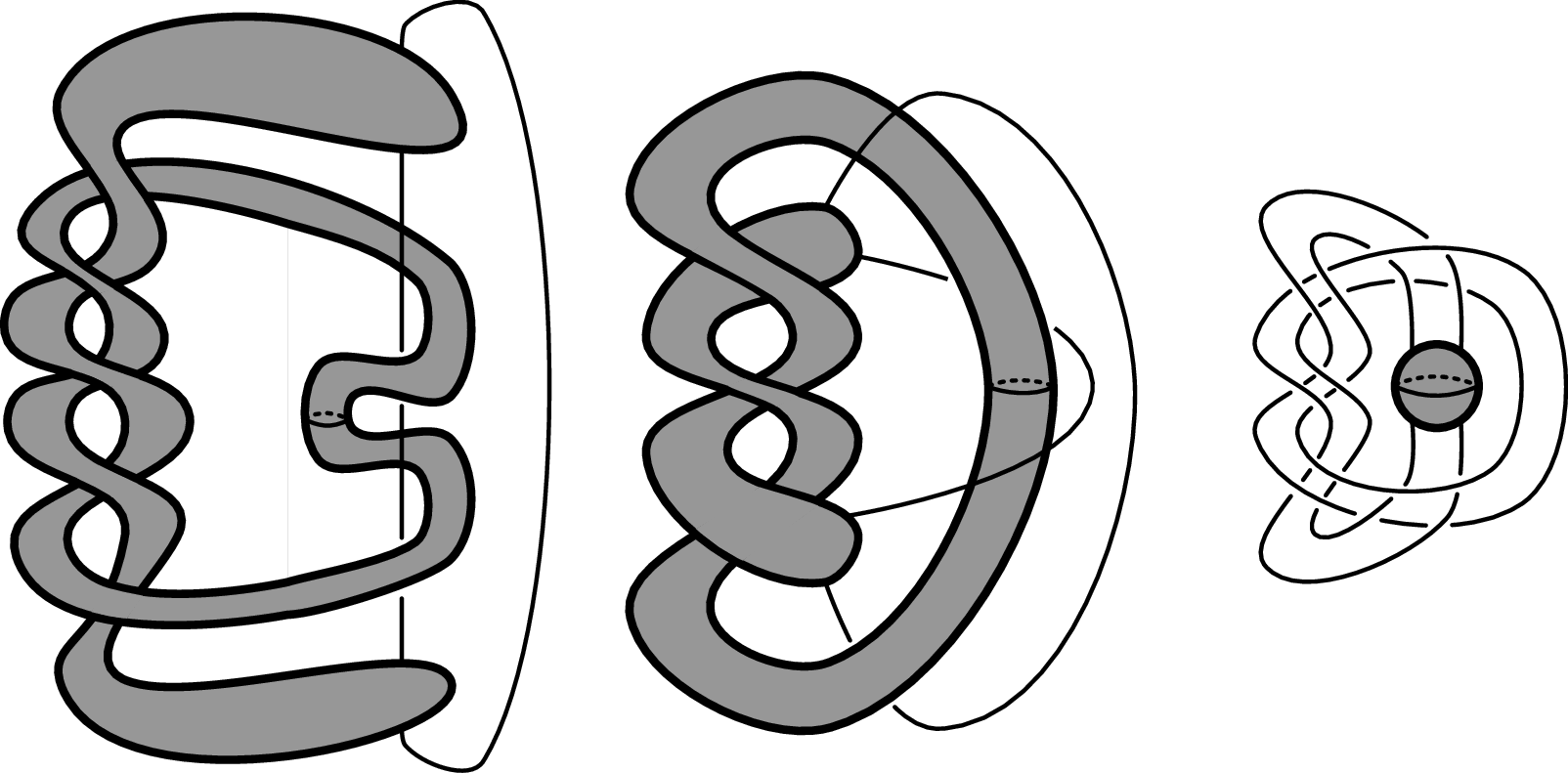}\caption{Isotopy of the fundamental domain for a strong inversion on the $(-2,5,5)$-pretzel knot. Notice that the resulting tangle has the property that integer closures are representable by closed 4-braids.}\label{fig:pretzel-isotopy}\end{center}\end{figure}
Notice that the associated quotient tangle in this case gives rise to an obvious collection of 4-braids, the closures of which give the branch sets for integer fillings. Setting \[\beta_n=\si_2^{-1}\si_3^{-1}\si_1\si_2\si_1^{14+n}\si_2\si_1\si_3^{-1}\si_2^{-1}(\si_2^{-1}\si_3^{-1}\si_1^{-1}\si_2^{-1})^3\] we have $\tau(n)=\overline{\beta_n}$ by verifying that $\Khred(\tau(0))\cong\bF^{16}\oplus\bF^{20}\oplus\bF^{4}$ so that $\det(\tau(0))=0$. The homologies of $\tau(n)$ for $n=-18,-17,-16,-15,-14$ are given in Figure \ref{fig:pretzel-homologies}.
\begin{figure}[ht!]\begin{center}
\labellist\small
	\pinlabel $1$ at 161 236 
	\pinlabel $2$ at 161 272
	\pinlabel $2$ at 161 307 
	\pinlabel $4$ at 161 342
	\pinlabel $3$ at 161 377
	\pinlabel $3$ at 161 414
	\pinlabel $2$ at 161 450
	\pinlabel $1$ at 161 486		
	
	\pinlabel $1$ at 196 342
	\pinlabel $1$ at 196 414
	\pinlabel $1$ at 196 450
	\pinlabel $1$ at 196 520 
	
	\pinlabel $1$ at 233 450
	\pinlabel $1$ at 233 520
	\pinlabel $1$ at 233 559
	\pinlabel $1$ at 233 632
\endlabellist
\includegraphics[scale=0.35]{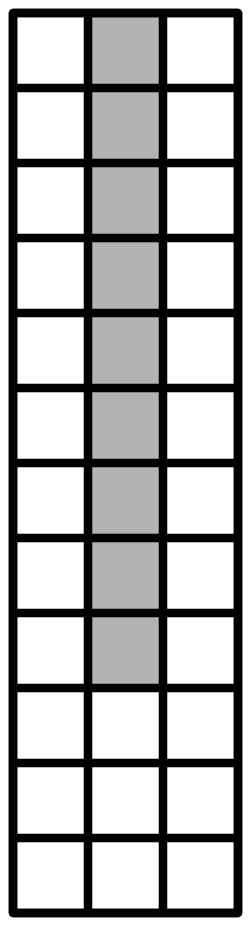} 
\qquad
\labellist\small
	\pinlabel $1$ at 161 236 
	\pinlabel $2$ at 161 272
	\pinlabel $2$ at 161 307 
	\pinlabel $3$ at 161 342
	\pinlabel $3$ at 161 377
	\pinlabel $3$ at 161 414
	\pinlabel $2$ at 161 450
	\pinlabel $1$ at 161 486		
	
	\pinlabel $1$ at 196 342
	\pinlabel $1$ at 196 414
	\pinlabel $1$ at 196 450
	\pinlabel $1$ at 196 520 
	
	\pinlabel $1$ at 233 450
	\pinlabel $1$ at 233 520
	\pinlabel $1$ at 233 559
	\pinlabel $1$ at 233 632
\endlabellist
\includegraphics[scale=0.35]{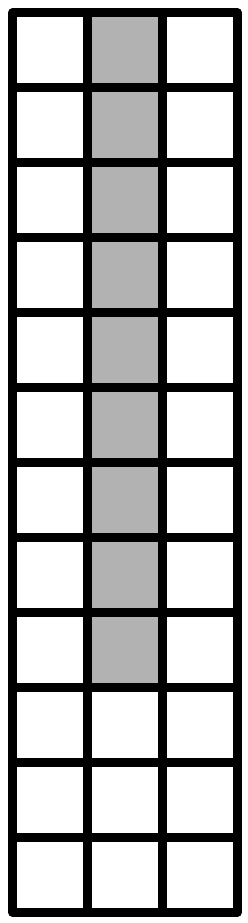} 
\qquad
\labellist\small
	\pinlabel $1$ at 161 236 
	\pinlabel $2$ at 161 272
	\pinlabel $2$ at 161 307 
	\pinlabel $3$ at 161 342
	\pinlabel $3$ at 161 377
	\pinlabel $3$ at 161 414
	\pinlabel $2$ at 161 450
	\pinlabel $1$ at 161 486		
	
	\pinlabel $1$ at 196 342
	\pinlabel $1$ at 196 377
	\pinlabel $1$ at 196 414
	\pinlabel $1$ at 196 450
	\pinlabel $1$ at 196 520 
	
	\pinlabel $1$ at 233 450
	\pinlabel $1$ at 233 520
	\pinlabel $1$ at 233 559
	\pinlabel $1$ at 233 632
\endlabellist
\includegraphics[scale=0.35]{figures/grid-pretzel-thick} 
\qquad
\labellist\small
	\pinlabel $1$ at 161 236 
	\pinlabel $2$ at 161 272
	\pinlabel $2$ at 161 307 
	\pinlabel $3$ at 161 342
	\pinlabel $3$ at 161 377
	\pinlabel $2$ at 161 414
	\pinlabel $2$ at 161 450
	\pinlabel $1$ at 161 486		
	
	\pinlabel $1$ at 196 342
	\pinlabel $1$ at 196 377
	\pinlabel $1$ at 196 414
	\pinlabel $1$ at 196 450
	\pinlabel $1$ at 196 520 
	
	\pinlabel $1$ at 233 450
	\pinlabel $1$ at 233 520
	\pinlabel $1$ at 233 559
	\pinlabel $1$ at 233 632
\endlabellist
\includegraphics[scale=0.35]{figures/grid-pretzel-thick} 
\qquad
\labellist\small
	\pinlabel $1$ at 161 236 
	\pinlabel $2$ at 161 272
	\pinlabel $2$ at 161 307 
	\pinlabel $3$ at 161 342
	\pinlabel $3$ at 161 377
	\pinlabel $2$ at 161 414
	\pinlabel $2$ at 161 450
	\pinlabel $1$ at 161 486		
	
	\pinlabel $1$ at 196 342
	\pinlabel $1$ at 196 377
	\pinlabel $1$ at 196 414
	\pinlabel $2$ at 196 450
	\pinlabel $1$ at 196 520 
	
	\pinlabel $1$ at 233 450
	\pinlabel $1$ at 233 520
	\pinlabel $1$ at 233 559
	\pinlabel $1$ at 233 632
\endlabellist
\includegraphics[scale=0.35]{figures/grid-pretzel-thick}\qquad
\labellist\small
	\pinlabel $1$ at 161 236 
	\pinlabel $2$ at 161 272
	\pinlabel $2$ at 161 307 
	\pinlabel $3$ at 161 342
	\pinlabel $3$ at 161 377
	\pinlabel $2$ at 161 414
	\pinlabel $2$ at 161 450
	\pinlabel $1$ at 161 486		
	
	\pinlabel $1$ at 196 342
	\pinlabel $1$ at 196 377
	\pinlabel $1$ at 196 414
	\pinlabel $2$ at 196 450
	\pinlabel $1$ at 196 486
	\pinlabel $1$ at 196 520 
	
	\pinlabel $1$ at 233 450
	\pinlabel $1$ at 233 520
	\pinlabel $1$ at 233 559
	\pinlabel $1$ at 233 632
\endlabellist
\includegraphics[scale=0.35]{figures/grid-pretzel-thick}
\caption{$\Khred(\tau(n))$ for $n=-18,-17,-16,-15,-14$ (from left to right).}\label{fig:pretzel-homologies}\end{center}\end{figure}
This data is enough to infer that \[\Khred(\tau(n))\cong\begin{cases} \bF^{-n}\oplus\bF^4\oplus\bF^4 & n<-16 \\ \bF^{17}\oplus\bF^5\oplus\bF^4 & n<-16 \\ 
\bF^{16}\oplus\bF^{20+n}\oplus\bF^4 & n>-16\end{cases}\] as relatively $\bZ$-graded groups. In particular, $\wmin=\wmax=3$ so the associated quotient tangle is generic. 
The result now follows from Theorem \ref{thm:finite}.
\end{proof}

Considering the same involution on the $(-2,p,p)$-pretzel knot $K_p$ for all $p\ge5$ we have that, in terms of the preferred associated quotient tangle, $\tau_p(n)=\overline{\beta_{n,p}}$ where \[\beta_{n,p}=
\si_2^{-1}\si_3^{-1}\si_1\si_2\si_1^{24+2p+n}\si_2\si_1\si_3^{-1}\si_2^{-1}(\si_2^{-1}\si_3^{-1}\si_1^{-1}\si_2^{-1})^{p-2}\] so that $S^3_n(K_p)\cong\Br(S^3,\tau_p(n))$. Notice that, aside from the exponent $24+2p+n$ corresponding to the surgery coefficient, this expression changes only the number of double-strand full-twists in the associated quotient tangle (see Figure \ref{fig:pretzel-isotopy}). From this expression, we calculate $\Khred(\tau_p(24+2p))$ for $p$ odd in the range $5\le p\le 31$.
 These calculations yield generic tangles in each case, with $w(\tau_p(24+2p))=p-2$, from which we conclude:
\begin{theorem}The $(-2,p,p)$-pretzel knots do not admit finite fillings for $5\le p\le 31$.\end{theorem} 
It seems reasonable to conjecture that $w=p-2$ for the branch sets associated with surgery on $K_p$, for any $p\ge 5$, so that Khovanov homology obstructs finite fillings on this class of knots. We do not pursue this here, since the result may be shown by other means. Indeed, it is possible to use obstructions from Heegaard Floer homology to rule out L-space surgeries by considering the Alexander polynomials of the $(-2,p,p)$-pretzel knots, as pointed out to the author by M. Hedden. This has been carried out very recently by Ichihara and Jong completing Mattman's classification of Montesinos knots admitting finite fillings \cite{IJ2008}. Since then the result has received a different treatment by Futer, Ishikawa, Kabaya, Mattman and Shimokawa \cite{FIKMS2009}. 

We remark that Mattman's classification \cite{Mattman2000} using character variety methods illustrates some subtleties. Indeed, the $(-2,3,q)$-pretzel knots admit L-space surgeries for all $q\ge3$ (see \cite{OSz2005-lens}). Despite this fact however, Mattman shows that for $q>9$ none of these manifolds can have finite fundamental group. On the other hand, for the $(-2,p,p)$-pretzel knots the character variety methods of Mattman were inconclusive, but this is precisely the setting in which Heegaard Floer homology -- and, as seen here, Khovanov homology -- obstructs finite fillings. 

\subsection{Khovanov homology obstructions and Heegaard Floer homology}\label{sec:11n1493} In light of the discussion above, it is natural to put the obstructions from Khovanov homology in contrast with those coming from Heegaard Floer homology. The latter theory gives very stringent restrictions for the knot Floer homology of a knot admitting an L-space surgery \cite{OSz2005-lens}. As a consequence, Ozsv\'ath and Szab\'o give the following quickly implemented obstruction from the Alexander polynomial. 
\begin{theorem}[{Ozsv\'ath-Szab\'o \cite[Corollary 1.3]{OSz2005-lens}}]\label{thm:os} 
A knot $K\into S^3$ for which $S^3_n(K)$ is an L-space (for some $n\in\bZ$) has Alexander polynomial of the form \[\Delta_K(t)=(-1)^k+\sum_{j=1}^k(-1)^{k-j}(t^{-n_j}+t^{n_j})\] for some increasing sequence of integers $0<n_1<n_2<\cdots<n_k$. 
\end{theorem} 
Manifolds elliptic geometry are known to be L-spaces \cite[Proposition 2.3]{OSz2005-lens}, so this gives a useful obstruction to finite fillings in the present context. However, the criteria given in Theorem \ref{thm:os} can fail. For example, \[\Delta_K(t)=t^{-3}-t^{-2}+t^{-1}-1+t-t^2+t^3\] where $K$ is the 14 crossing, non-alternating knot shown in Figure \ref{fig:14n11893-inversion}. 
\begin{figure}[ht!]\begin{center}  
\includegraphics[scale=0.35]{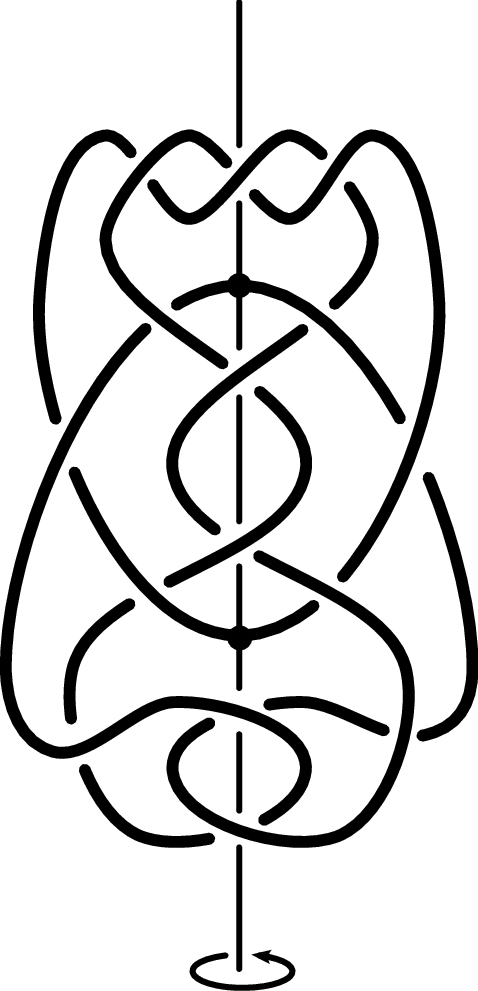}\caption{The strongly invertible knot $K=14^n_{11893}$ has Alexander polynomial $\Delta_K(t)=t^{-3}-t^{-2}+t^{-1}-1+t-t^2+t^3$.}\label{fig:14n11893-inversion}\end{center}\end{figure}
Since this is a strongly invertible knot, we are in a position to apply width obstructions from Khovanov homology. 
\begin{figure}[ht!]\begin{center}  
\includegraphics[scale=0.40]{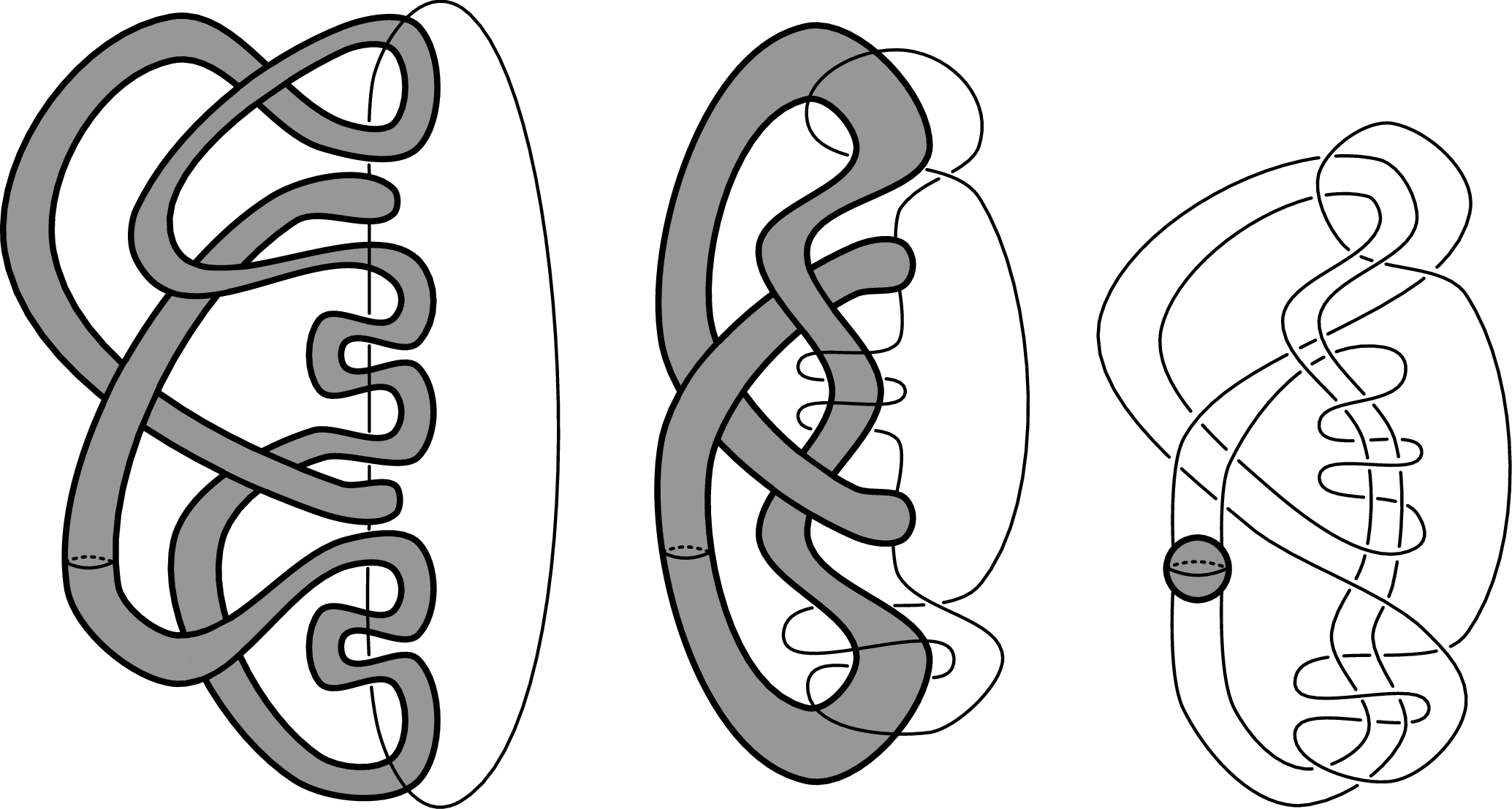}\caption{Isotopy of a fundamental domain for the involution on the complement of $14^n_{11893}$.}\label{fig:14n11893-isotopy}\end{center}\end{figure}
The associated quotient tangle is determined in Figure \ref{fig:14n11893-isotopy}: notice that by construction the trivial knot $\tau(\overzero)$ is obtained by connecting the endpoints of the arcs of $\tau$ with two horizontal arcs inside the small sphere shown. Therefore, without knowing the framing, we can be sure that the branch sets for integer surgeries result from adding vertical half-twists inside the sphere, as shown in Figure \ref{fig:14n11893-homology}.    
\begin{figure}[ht!]\begin{center}
\labellist\small
	\pinlabel $2$ at 161 236 
	\pinlabel $2$ at 161 271 	
	\pinlabel $3$ at 161 307 
	\pinlabel $4$ at 161 342
	\pinlabel $3$ at 161 377	
	\pinlabel $3$ at 161 414
	\pinlabel $2$ at 161 450
	\pinlabel $1$ at 161 486
		
	\pinlabel $3$ at 196 342
	\pinlabel $4$ at 196 377
	\pinlabel $5$ at 196 414
	\pinlabel $7$ at 196 450
	\pinlabel $6$ at 196 486
	\pinlabel $5$ at 196 520 
	\pinlabel $4$ at 196 557
	\pinlabel $2$ at 196 596
	
	\pinlabel $3$ at 233 450
	\pinlabel $5$ at 233 486
	\pinlabel $6$ at 233 520
	\pinlabel $7$ at 233 557	
	\pinlabel $7$ at 233 596
	\pinlabel $5$ at 233 632
	\pinlabel $4$ at 233 665
	\pinlabel $2$ at 233 701

	\pinlabel $1$ at 268 557
	\pinlabel $2$ at 268 596
	\pinlabel $2$ at 268 632
	\pinlabel $3$ at 268 665
	\pinlabel $3$ at 268 701
	\pinlabel $2$ at 268 739
	\pinlabel $2$ at 268 776
	\pinlabel $1$ at 268 810
\endlabellist
\includegraphics[scale=0.35]{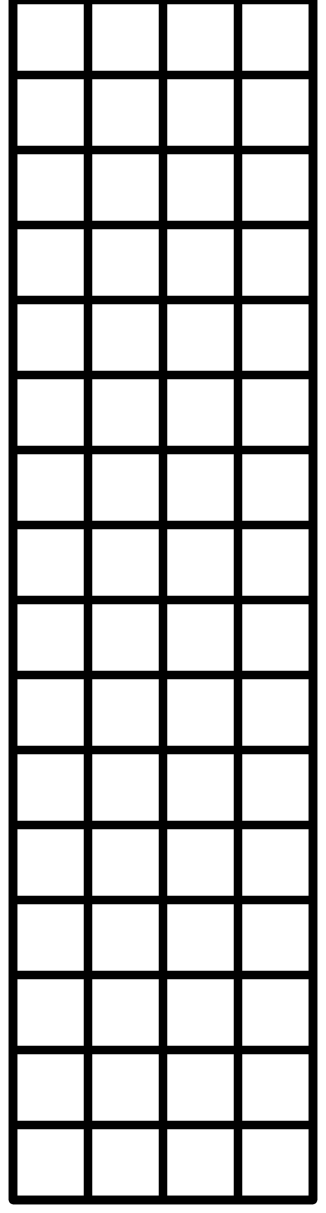}\quad 
\includegraphics[scale=0.35]{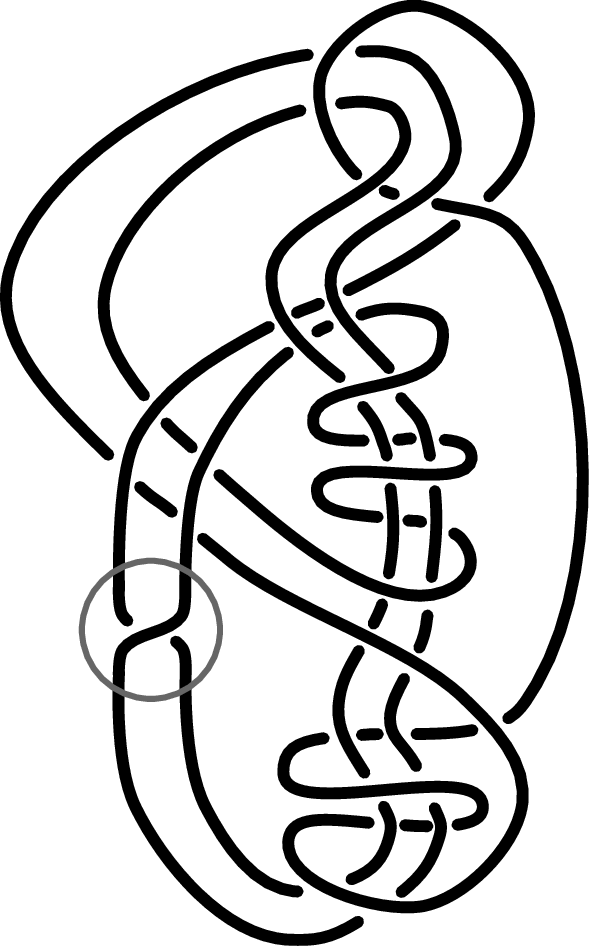} 
\caption{The branch set for some integer surgery $S^3_n(K)$. Note that $\Khred(\tau(n))\cong\bF^{20}\oplus\bF^{36}\oplus\bF^{39}\oplus\bF^{16}$ so that $\chi=59-52=7$ and $n=\pm7$.}\label{fig:14n11893-homology}\end{center}\end{figure}

By switching the circled crossing of Figure \ref{fig:14n11893-homology} from positive to negative, we can determine \begin{align*}
\Khred(\tau(-9)) &\cong \bF^{20}\oplus\bF^{36}\oplus\bF^{41}\oplus\bF^{16} \\
\Khred(\tau(-7)) &\cong \bF^{20}\oplus\bF^{36}\oplus\bF^{39}\oplus\bF^{16}
\end{align*}  to conclude that $\wmin=\wmax=4$ and, as $T$ is generic, this determines the width of the branch set for {\em any} surgery on $K$. We conclude: 

\begin{theorem}$14^n_{11893}$ does not admit finite fillings; two Khovanov homology groups suffice. \end{theorem}

Note that this example does not establish Khovanov homology obstructions as stronger than Heegaard Floer homology, however it does show that Khovanov homology can provide more information than obstructions arising from the Alexander polynomial (an obstruction that is a consequence of Heegaard Floer homology).  While it is possible that the full knot Floer homology of $K$ obstructs L-space surgeries, this example shows that in certain settings the Khovanov homology obstructions may be more convenient from a computational standpoint when the question of finite fillings is of interest. Further, these obstructions may allow one to rule out finite fillings among L-spaces, a distinction that can be subtle.

\bibliographystyle{plain}
\bibliography{obstructions}
\end{document}